\documentclass[11pt]{article}

\usepackage{latexsym}
\usepackage{amssymb}
\usepackage{amsthm}
\usepackage{amscd}
\usepackage{amsmath}
\usepackage{mathrsfs}
\usepackage{graphicx}
\usepackage{hyperref}
\usepackage{shuffle}
\usepackage{tikz-cd}
\usepackage[all]{xy}
\input xy \xyoption{frame}

\usepackage[colorinlistoftodos]{todonotes}
\usetikzlibrary{chains,scopes,decorations.markings}

\numberwithin{equation}{section}

\theoremstyle{definition}
\newtheorem* {theorem*}{Theorem}
\newtheorem* {conjecture*}{Conjecture}
\newtheorem{theorem}{Theorem}[section]

\theoremstyle{definition}

\theoremstyle{definition}

\newtheorem* {example*}{Example}

\newtheorem{lemma}[theorem]{Lemma}
\theoremstyle{definition}

\theoremstyle{definition}

\newtheorem{proposition}[theorem]{Proposition}
\newtheorem{corollary}[theorem]{Corollary}

\newtheorem {remark}[theorem]{Remark}
\theoremstyle{definition}
\newtheorem {example}[theorem]{Example}
\theoremstyle{definition}

\theoremstyle{definition}

\theoremstyle{definition}

\theoremstyle{definition}

\def\modu{\ (\mathrm{mod}\ }

\def\({\left(}
\def\){\right)}
\newcommand{\sP}{\mathscr{P}}

\newcommand{\CC}{\mathbb{C}}
\newcommand{\QQ}{\mathbb{Q}}
\newcommand{\cP}{\mathcal{P}}

\newcommand{\cK}{\mathcal{K}}

\newcommand{\cR}{\mathcal{R}}

\def\cX{\mathcal{X}}

\def\NN{\mathbb{N}}

\def\CC{\mathbb{C}}
\def\RR{\mathbb{R}}
\def\ZZ{\mathbb{Z}}
\def\Aut{\mathrm{Aut}}

\def\Ind{\mathrm{Ind}}

\def\Res{\mathrm{Res}}

\def\Irr{\mathrm{Irr}}

\newcommand{\cM}{\mathcal{M}}

\newcommand{\cL}{\mathcal{L}}

\newcommand{\one}{{1\hspace{-.11cm} 1}}

\newcommand{\sgn}{\mathrm{sgn}}

\def\barr{\begin{array}}
\def\earr{\end{array}}
\def\ba{\begin{aligned}}
\def\ea{\end{aligned}}
\def\be{\begin{equation}}
\def\ee{\end{equation}}

\def\qquand{\qquad\text{and}\qquad}
\def\quand{\quad\text{and}\quad}

\def\quord{\quad\text{or}\quad}

\def\cH{\mathcal H}
\def\cM{\mathcal M}

\def\ds{\displaystyle}

\def\id{\mathrm{id}}

\def\ben{\begin{enumerate}}
\def\een{\end{enumerate}}

\def\fpf{{\tt {FPF}}}

\def\D{\hat D}

\newcommand{\xRightarrow}[2][]{\ext@arrow 0359\Rightarrowfill@{#1}{#2}}

\def\sI{\mathscr{I}}

\def\arcstart{\ \xy<0cm,-.06cm>\xymatrix@R=.1cm@C=.2cm }
\newcommand{\arcstartc}[1]{\ \xy<0cm,-.15cm>\xymatrix@R=.1cm@C=#1cm}

\def\m{\mathfrak{m}}
\def\n{\mathfrak{n}}

\def\invsim_i{\overset{\mathrm{i}}{\underset{\mathrm{inv}}{\sim}}}

\def\H{\mathcal{H}}
\def\A{\mathsf{A}}
\def\BC{\mathsf{BC}}
\def\D{\mathsf{D}}

\def\cP{\mathscr{P}}

\def\WD{W^{\mathsf{D}}}

\def\m{\mathbf{m}}
\def\n{\mathbf{n}}

\usepackage[colorinlistoftodos]{todonotes}

\def\TT{{\mathbb T}}
\def\Ad{\mathrm{Ad}}
\def\SS{{\mathbb{S}}}

\def\W{W^{\BC}}
\def\WD{W^{\D}}

\def\aut{\mathsf{aut}}
\def\height{\mathsf{ht}}

\def\fpf{\mathrm{fpf}}

\def\MTI{\mathsf{Index}}
\def\EvenRows{\mathsf{ERows}}
\def\EvenCols{\mathsf{ECols}}

\def\BC{\mathsf{B}}

\def\OR{\mathsf{OR}}
\def\OC{\mathsf{OC}}

\def\EvenBiRows{\mathsf{ERows}_B}
\def\EvenBiCols{\mathsf{ECols}_B}

\def\EvenDRows{\mathsf{ERows}_D}
\def\EvenDCols{\mathsf{ECols}_D}

\def\OddRows{\mathsf{ORows}}
\def\OddCols{\mathsf{OCols}}

\def\OddBiRows{\mathsf{ORows}_B}
\def\OddBiCols{\mathsf{OCols}_B}

\def\OddDRows{\mathsf{ORows}_D}
\def\OddDCols{\mathsf{OCols}_D}

\def\WH{W^{\mathsf{H}}}

\def\sgn{\operatorname{sgn}}

\def\ncols{\mathsf{ncols}}

\def\cP{\mathcal{P}}

\def\sP{\cP}

\def\piR{\pi_{\mathsf{R}}}
\def\piL{\pi_{\mathsf{L}}}

\definecolor{darkred}{rgb}{0.7,0,0} 
\newcommand{\defn}[1]{{\color{darkred}\emph{#1}}} 

\usepackage{fullpage}
\usepackage{ytableau}

\begin{document}
\title{Perfect models for finite Coxeter groups}
\author{
Eric MARBERG\thanks{
%Department of Mathematics, Hong Kong University of Science and Technology,  
\tt emarberg@ust.hk
}
\and
Yifeng ZHANG\thanks{
%Department of Mathematics, Hong Kong University of Science and Technology, 
\tt yzhangci@connect.ust.hk
}
}
%\date{}
\date{Department of Mathematics \\ Hong Kong University of Science and Technology}

\maketitle

\abstract{ 
A model for a finite group is a set of linear characters of subgroups that can be induced 
to obtain every irreducible character exactly once. A perfect model for a finite Coxeter group
is a model in which the relevant subgroups are the quasiparabolic centralizers of perfect involutions.
In prior work, we showed that perfect models give rise to interesting examples of $W$-graphs.
Here, we classify which finite Coxeter groups have perfect models. Specifically, we prove that 
the irreducible finite Coxeter groups with perfect models are those of types $\mathsf{A}_{n}$, $\mathsf{B}_n$, $\mathsf{D}_{2n+1}$, $\mathsf{H}_3$, or $\mathsf{I}_2(n)$. We also show that up to a natural form of equivalence,
outside types $\mathsf{A}_3$, $\mathsf{B}_n$, and  $\mathsf{H}_3$, each irreducible finite Coxeter group has at most one perfect model. Along the way, we also prove a technical result about representations of finite Coxeter groups, namely, that induction from standard parabolic subgroups of corank at least two is never multiplicity-free.
}

\setcounter{tocdepth}{2}
%\tableofcontents

\section{Introduction}\label{intro-sect}

A \defn{model} for a finite group $G$ is a set of linear characters $\sigma_i : H_i \to \CC$
of subgroups such that adding up the induced characters 
$\sum_i \Ind_{H_i}^G(\sigma_i)$ gives the multiplicity-free sum $\sum_{\psi \in \Irr(G)} \psi$ of all complex irreducible characters
of $G$. A model for $G$ lets one construct an explicit $G$-representation,
with a natural basis relative to which the elements of $G$ act
as monomial matrices, containing each irreducible $G$-representation exactly once.

\begin{example}\label{ex-intro}
For a positive integer $n$ let $S_n$ be the symmetric group of permutations of $[n] := \{1,2,\dots,n\}$.
Embed $S_i \times S_{n-i}\subseteq S_n$ as the subgroup of elements preserving $\{1,2,\dots,i\}$.
This subgroup acts on itself by $(g,h) : (x,y) \mapsto (g^\ast xg^{-1}, hyh^{-1})$
where $g^\ast \in S_i$ is the permutation mapping $ a \mapsto i + 1 - g(i+1-a)$.
Let $H_i$ denote the stabilizer subgroup of $1 \in S_i\times S_{n-i}$
and define $\sigma_i  (x,y) = \sgn(y) $.
Then $\{ \sigma_i :H_i\to \{\pm 1\} : i=0,2,4,\dots 2\lfloor \frac{n}{2}\rfloor\} $ is a model for $S_n$ \cite{IRS}.
\end{example}

In our previous work \cite{MZ2020} we introduced the notion of a \defn{perfect model} for a finite Coxeter group.
The preceding construction gives an example of such a model for the symmetric group.
For more general Coxeter groups the precise definition of a perfect model goes as follows.
%Our goal in this paper is to classify which groups have perfect models.

Let $(W,S)$ be a finite Coxeter system.
Define $\Aut(W,S)$ to be the set of automorphisms $\theta \in \Aut(W)$ with $\theta(S) = S$.
Let $W^+$ be the set of pairs $(w,\theta) \in W \times \Aut(W,S)$,
viewed as a group with multiplication 
$(u,\alpha)(v,\beta) = (u\alpha(v),\alpha\beta).$
We view $W$ as a subgroup of $W^+$ by identifying $w \in W$ with $(w,\id) \in W^+$.
%This means that if $z=(w,\theta) \in W^+$ then $z^{-1} = (\theta(w)^{-1}, \theta^{-1})$.
%Extend $\ell$ to $W^+$ by setting $\ell(w,\theta) = \ell(w)$. We identify $w \in W$ 
%with the element $(w,\id) \in W^+$ and 
%$\theta \in \Aut(W,S)$ with the element $(1,\theta) \in W^+$.
%Every $\alpha \in \Aut(W,S)$ extends to an automorphism of $W^+$ by the formula
%\[
%\alpha : (w,\theta) \mapsto (\alpha(w),\alpha \theta  \alpha^{-1}) = (1,\alpha) (w,\theta) (1,\alpha)^{-1}.\]
%
%\subsection{Perfect involutions}
%
%The set of \defn{reflections} in $W$ is $T := \{ wsw^{-1} : (w,s) \in W\times S\}$.
%An element $z =(w,\theta)\in W^+$ is an \defn{involution} if $z=z^{-1}$,
%which holds if and only if $\theta=\theta^{-1}$ and $w^{-1} = \theta(w)$.

An element $z \in  W^+$ is a \defn{perfect involution} if 
$ z^2 = (zt)^4 = 1$ for all $t \in \{ wsw^{-1} : (w,s) \in W\times S\}$.
For example, every fixed-point-free involution in $S_n$ when $n$ is even is perfect.
Rains and Vazirani introduced the notion of perfect involutions in \cite{RV} as an example of a \defn{quasiparabolic set}; see Section~\ref{qp-sect} for further discussion of this background.

Let $\sI = \sI(W,S)$ denote the set of perfect involutions in $W^+$.
The group $W$ acts on $\sI$ by conjugation.
%If $z=(w,\theta) \in \sI$ and $v \in W$ then $vzv^{-1} = (vw\theta(v)^{-1},\theta) \in \sI$.
%We refer to the orbits of this $W$-action as the \defn{perfect conjugacy classes} of $(W,S)$.
Given a subset $J \subseteq S$, write $W_J := \langle s \in J\rangle$ and let \[\sI_J := \sI(W_J,J) \subseteq W_J^+ := (W_J)^+.\]
A \defn{(perfect) model triple} $\TT  =(J, \cK, \sigma)$ for $(W,S)$ consists of a subset $J \subseteq S$,
a $W_J$-conjugacy class $\cK \subseteq \sI_J$, and a linear character $\sigma :W_J\to \{\pm1\}$.\footnote{
Since Coxeter groups are generated by involutions, any linear character $\sigma:W_J \to \CC$ takes values in $\{\pm 1\}$.
More generally, every character $\chi \in \Irr(W)$ takes values in a subfield of $\RR \subsetneq \CC$ \cite[Thm. 5.3.8]{GeckP}.}
The \defn{character} of 
 $\TT$  is 
 \be\label{chTT-eq}\chi^\TT := \Ind_{W_J}^W \Res^{W_J}_{C_J(z)} (\sigma)\ee
where $z \in \cK$ is arbitrary and $C_J(z) := \{ w \in W_J : wz=zw\}$.
A \defn{perfect model} for $W$, finally, is a set of model triples $\sP$ 
such that  $\sum_{\TT \in \sP} \chi^\TT =\sum_{\psi \in \Irr(W)} \psi$.
%is the sum of all irreducible (complex) characters of $W$.

\begin{example}
If $s_1,s_2,\dots,s_{n-1}$ are the usual simple generators of $S_n$, 
then the perfect model corresponding to Example~\ref{ex-intro}
comes from taking $J = \{ s_1,s_2,\dots,s_{i-1},s_{i+1},s_{i+2},\dots,s_{n-1}\}$
and $\cK = \{ (g^\ast g^{-1}, \ast) \in (S_i\times S_{n-i})^+  : g  \in S_i \}$ for $i=0,2,4,\dots, 2\lfloor \frac{n}{2}\rfloor$.
\end{example} 

Our goal in this article is to classify which finite Coxeter groups
have perfect models. Before stating our main results,
we briefly explain why such models are interesting to consider.

A \defn{Gelfand model} for a group or algebra is a semisimple module containing exactly one constituent in each isomorphism class of irreducible representations.
Each
perfect model  
gives rise to a pair of Gelfand models for the Iwahori-Hecke algebra $\H(W)$ of $(W,S)$ with some nice properties.
The perfect model for $S_n$ described above leads in this way to the $\H(S_n)$-representation previously studied in \cite{APR2007}, for example.

There are simple formulas for the action of the standard generators of $\H(W)$
in these Gelfand models.
Each module also has a unique \defn{bar operator} that is compatible with the usual bar operator of $\H(W)$,
and a unique bar invariant \defn{canonical basis} \cite{Marberg,MZ2020}
analogous to the Kazhdan-Lusztig basis of $\H(W)$.
The action of the standard generators of $\H(W)$ on these canonical bases may be encoded as 
\defn{$W$-graphs} in the sense of \cite{KL}. These objects then provide examples of \defn{Gelfand $W$-graphs}:
$W$-graphs whose corresponding 
Iwahori-Hecke algebra representations are Gelfand models.
It is natural to study perfect models 
in order to better understand which Gelfand $W$-graphs can arise from this construction.

Another reason to be interested in perfect models is for their connection to models of finite groups of Lie type.
One can view the perfect model for $S_n$ in Example~\ref{ex-intro}
as the ``$q\to 1$ limit'' of the so-called \defn{Klyachko model} for the finite general linear group $\mathsf{GL}(n,q)$ \cite{HowlettZ,Klyachko}. We do not know of much related work on Klyachko models for the other classical finite groups of Lie type,
but we expect that such models should be similarly related to perfect models for classical Weyl groups.

We now summarize our results.
The following combines Theorems~\ref{red-thm} and each of the main theorems in Sections~\ref{a-sect}, \ref{b-sect}, \ref{d-sect}, and \ref{e-sect}.

\begin{theorem}\label{thm1}
A finite Coxeter group has a perfect model if and only if each of irreducible factors has a perfect
model. An irreducible finite Coxeter group has a perfect model if and only if it is of type $\A_{n-1}$, 
$\BC_n$, $\D_{2n+1}$,   $\mathsf{H}_3$, or $\mathsf{I}_2(n+1)$ for an integer $n\geq 2$.
\end{theorem}

An \defn{involution model} for a finite group $G$ 
is a model $\{ \lambda_i : H_i \to \CC\}$ in which the subgroups $H_i$ range over the centralizers of the distinct 
conjugacy classes of involutions $g=g^{-1} \in G$.
Such models are natural to consider when $G$ has all real representations, since then
the Frobenius-Schur involution counting theorem asserts that $\sum_{\psi \in \Irr(G)} \psi(1) = |\{ g\in G :g=g^{-1}\}|$.

Involution models for finite Coxeter groups were studied and classified in \cite{Baddeley,BaddeleyThesis,IRS,Vinroot}.
Comparing Theorem~\ref{thm1} with the main result in \cite{Vinroot} gives the following corollary.

\begin{corollary}
A finite Coxeter group has a perfect model if and only if it has an involution model.
\end{corollary}

We do not know of an explanation 
for this phenomenon that avoids appealing to the case-by-case classification of both kinds of models.
More general kind of involution models for complex reflection groups have been studied and classified in \cite{Caselli2009,CF,CM,M2011,M2012}.
It would be interesting to know if the notion of a perfect model can be extended to that context.

% An interesting feature of any $W$-graph is its decomposition into strongly connected components, called \defn{cells}. 
Besides settling existence questions, our results also establish some uniqueness properties of perfect models.
A finite Coxeter group may have many different perfect models, each producing different Gelfand $W$-graphs.
We can show, however, that these $W$-graphs are all isomorphic after possibly ignoring edge labels and reversing edge orientations. We do this 
by studying a form on equivalence for perfect models introduced in \cite{MZ2020}; see Section~\ref{model-equiv-sect} for the  definition.
Perfect models that are equivalent give rise to essentially the same $W$-graphs, in a way that will be made precise below.
%Specifically, each weakly connected component in the Gelfand $W$-graphs associated to one perfect model is either isomorphic or anti-isomorphic to
%a weakly connected component in  one of the Gelfand $W$-graphs associated to an equivalent perfect model.

In  \cite{MZ2020} we described the Gelfand $W$-graphs associated to specific perfect models
for each classical Weyl group, excluding type $\D_{2n}$. 
The following result (combining Theorems~\ref{a-thm}, \ref{b-thm} and \ref{d-thm} and Proposition~\ref{i2-prop})
gives a sense in which these Gelfand $W$-graphs are canonical.

\begin{theorem}\label{equiv-thm}
If $W$ is an irreducible finite Coxeter group not of type $\A_3$, $\BC_n$, or $\mathsf{H}_3$
then $W$ has at most one equivalence class of perfect models.
\end{theorem}

If $W$ is of type $\BC_n$ for $n\neq 3$ then there are exactly two equivalence classes of perfect models,
one of which is a trivial ``refinement'' of the other; see Theorem~\ref{b-thm}.
There are a few additional models when $W$ is of type $\A_3$, $\BC_3$, or $\mathsf{H_3}$;
see Examples~\ref{a-extra-ex} and \ref{b3-ex} and Proposition~\ref{h3-prop}.

Our proofs rely on the following property which may be of independent interest.
 This theorem extends results in \cite{APVM,Anderson} which address the case when $\chi=\one$ is the trivial character:

\begin{theorem}\label{tech1-thm}
If $(W,S)$ is an irreducible finite Coxeter system and $J\subseteq S$ has $|S - J| \geq 2$,
then $\Ind_{W_J}^W(\chi)$ is not multiplicity-free for any irreducible character $\chi \in \Irr(W_J)$.
\end{theorem}

Section~\ref{prelim-sect} contains some background and general 
results about perfect models. Sections~\ref{a-sect}, \ref{b-sect}, and \ref{d-sect}
classify the perfect models up to equivalence for each classical Weyl group.
Section~\ref{e-sect} briefly explains the perfect model classification for the remaining irreducible finite Coxeter groups.
Appendix~\ref{app-sect}, finally contains the proofs of Theorem~\ref{tech1-thm} and another technical result.

\subsection*{Acknowledgements}

This work was partially supported by grants ECS 26305218 and GRF 16306120
from the Hong Kong Research Grants Council.
%We thank Jay Taylor for several helpful comments.

\section{Preliminaries}\label{prelim-sect}

\subsection{Restriction and induction}

Suppose $H\subseteq G$ are finite groups. 
Write $\Irr(H)$ and $\Irr(G)$ for the corresponding sets of complex irreducible characters.
If $\psi : H \to \CC$ and $\chi : G \to \CC$ are class functions (that is, maps constant on conjugacy classes),
then we denote the restriction of $\chi$ to $H$ by $\Res_H^G(\chi)$
and the class function of $G$ induced from $\psi$ by $\Ind_H^G(\psi)$.
An explicit formula for the induced function is 
\be\label{ind-eq}
\Ind_H^G(\chi)(x) = \frac{1}{|H|} \sum_{\substack{g \in G \\ gxg^{-1} \in H}} \chi(gxg^{-1})\quad\text{for }x \in G.
\ee
Induction from $H$ to $G$ is the unique linear operation such that
$\langle \chi,\Ind_H^G(\psi)\rangle_G = \langle \Res_H^G(\chi),\psi\rangle_H$
for all $\chi \in \Irr(G)$ and $\psi \in \Irr(H)$, where 
$\langle \cdot,\cdot\rangle_G$ is the bilinear form on class functions of $G$  
relative to which $\Irr(G)$ is an orthonormal basis.

\subsection{Extended Coxeter groups}

Let $(W,S)$ be a finite Coxeter system with length function $\ell : W \to \NN =\{0,1,2,\dots\}$.
The group $W$ always has at least two linear characters, given by the trivial character 
$\one : w \mapsto 1$ and the sign character $\sgn : w \mapsto (-1)^{\ell(w)}$.

Define $W^+ = W \rtimes \Aut(W,S)$ as in the introduction. We extend $\ell$ to $W^+$ by setting $\ell(w,\theta) = \ell(w)$.
Besides identifying $w \in W$ 
with  $(w,\id) \in W^+$, we also identify 
each $\theta \in \Aut(W,S)$ with the element $(1,\theta) \in W^+$
and view $\Aut(W,S) \subseteq W^+$ as a subgroup in this way.
Every $\alpha \in \Aut(W,S)$ extends to an automorphism of $W^+$ by the formula
\[
\alpha : (w,\theta) \mapsto (\alpha(w),\alpha \theta  \alpha^{-1}) = (1,\alpha) (w,\theta) (1,\alpha)^{-1}.\]
Suppose $z=(w,\theta) \in W^+$. Then $z^{-1} = (\theta(w)^{-1}, \theta^{-1})$,
so $z^2=1$
 if and only if $\theta=\theta^{-1}$ and $w^{-1} = \theta(w)$.
The conjugation action of  $g \in W$ on $W^+$ is $gzg^{-1} = (g\cdot w\cdot \theta(g)^{-1},\theta)$.
We refer to the orbits of this $W$-action in the set $\sI$ of perfect involutions as \defn{perfect conjugacy classes}.

\subsection{Quasi-parabolic sets}\label{qp-sect}

Introduced by Rains and Vazirani in \cite{RV}, 
a \defn{quasi-parabolic $W$-set} is a set $X$ with a height function $\height : X \to \ZZ$ 
and a left $W$-action satisfying a short list of technical axioms.
 The motivating example is the set of distinguished coset representatives $W^J := \{ w \in W : \ell(sw) > \ell(w)\text{ for all }s \in J\}$
where $J \subseteq S$ and $\height = \ell$.
 The quasi-parabolic axioms ensure 
% that there is a partial order on $X$ analogous to the Bruhat order of $W$ restricted to $W^J$, and 
 that there are simple formulas for a module of the Iwahori-Hecke algebra of $W$ deforming the permutation representation of $W$ on $X$.
 
The set of perfect involutions $\sI$ in $W^+$ is an example of a quasi-parabolic $W$-set,
relative to the conjugation action of $W$ and the height function $\height(z) := \lfloor\frac{\ell(z)}{2}\rfloor$ \cite[\S4]{RV}.
This is perhaps the most interesting general construction of a quasiparabolic set that is not isomorphic to one of the ``parabolic'' examples $W^J$.
This fact has many consequences; we note here just one technical property.
An element $z \in W^+$ is \defn{$W$-minimal} if $\ell(szs) \geq \ell(z)$ for all $s \in S$. 
Because $\sI$ is quasi-parabolic, each perfect conjugacy class contains a unique $W$-minimal element
 \cite[Cor. 2.10]{RV}.
This element is also the unique minimal-length element in its class.

\subsection{Dual model triples}

Recall the notion of a \defn{(perfect) model triple} for $(W,S)$ from the introduction.
We do not distinguish between model triples $(J,\cK,\sigma_1)$ and $(J,\cK,\sigma_2)$
when $\Res^{W_J}_{C_J(z)} (\sigma_1) = \Res^{W_J}_{C_J(z)} (\sigma_2)$,
as these give rise to the same character via \eqref{chTT-eq}.

Given a subset $J\subseteq S$ let $w_J$ denote the longest element of $W_J$ and define $w_0 := w_S$.
For $w \in W$ let $\Ad(w) \in \Aut(W)$
denote the inner automorphism $x \mapsto wxw^{-1}$.
Then $\Ad(w_J) \in \Aut(W_J,J)$ and the element $w_J^+ := (w_J, \Ad(w_J))$ is a central involution in $W_J^+$,
so $w_J^+ \in \sI_J$. Let $w_0^+ := w_S^+$.
The \defn{dual} of a model triple $\TT = (J,\cK,\sigma)$  
is   $\TT^\vee := (J^\vee, \cK^\vee, \sigma^\vee)$ where 
\[
\ba 
J^\vee &:= \Ad(w_0)(J) = w_0 J w_0 ,\\
\cK^\vee &:= \Ad(w_0) \cdot w_J^+ \cdot \cK\cdot  \Ad(w_0) = \{ (w_0   x w_J  w_0, \Ad(w_0)  \Ad(w_J)  \theta  \Ad(w_0)) : (x,\theta) \in \cK\}, \\
\sigma^\vee &:=  \sigma\circ \Ad(w_0).
\ea
\]
Since $w_0$ and $w_J^+$ are involutions, it is easy to see that $(\TT^\vee)^\vee = \TT$.
It holds by \cite[Prop. 3.33]{MZ2020} that if $\TT$ is a model triple for $(W,S)$ then so is $\TT^\vee$ and $\chi^\TT = \chi^{\TT^\vee}$.

\subsection{Model equivalence}\label{model-equiv-sect}

Let  $\TT = (J,\cK,\sigma)$ be a model triple for $(W,S)$.
Given  $\alpha \in \Aut(W,S)$,
define \[\TT^\alpha := (\alpha^{-1}(J), \alpha^{-1}(\cK), \sigma\circ \alpha).\]
 As explained in \cite[\S3.5]{MZ2020}, this is also a model triple of $(W,S)$
with $\chi^{\TT^\alpha} =\chi^{\TT} \circ \alpha$.

Suppose  $\TT' = (J',\cK',\sigma')$ is another model triple for $W$.
We write $\TT \equiv \TT'$ if $J=J'$
and
it holds
that $C_{J}(z) = C_{J'}(z')$ and $\Res^{W_J}_{C_J(z)}(\sigma) = \Res^{W_J}_{C_J(z)}(\sigma')$
where $z \in \cK$ and $z' \in \cK'$
are the unique minimal-length elements in each $W_J$-conjugacy class.
In this case   $\chi^{\TT} = \chi^{\TT'}$.

Let $\sim$ denote the transitive closure of the relation on model triples
that has $\TT \sim \TT'$ when $\TT \equiv \TT'$ or $ \TT^\vee = \TT'$ or $\TT^\alpha = \TT'$
for an inner automorphism $\alpha \in \Aut(W,S) \cap \{ \Ad(w) : w \in W\}$.
When $\TT \sim \TT'$ we say that the model triples are \defn{strongly equivalent}.
The following is clear:

\begin{proposition}
Strongly equivalent model triples for $(W,S)$ have the same characters.
\end{proposition}

Finally write $\sgn : w \mapsto (-1)^{\ell(w)}$ for the sign character of $W$
and define $\overline \TT :=  (J,\cK, \sigma \sgn)$.
This is another model triple for $(W,S)$ with $\chi^{\overline\TT} =  \chi^{\TT} \sgn$.

We define $\approx$ to be the transitive closure of the relation on model triples for $(W,S)$
that has $\TT \approx \TT'$ whenever $\TT \sim  \TT'$, $\overline\TT=\TT'$,  or $ \TT^\alpha = \TT'$ for an outer automorphism $\alpha \in \Aut(W,S)$.
When $\TT \approx \TT'$ we say that the two model triples are \defn{equivalent}.
When $\sP$ and $\sP'$ are sets of model triples,
we write $\sP \approx \sP'$ and say that $\sP$ and $\sP'$
are \defn{equivalent}
if there is a bijection $\sP \to \sP'$ 
such that if $\TT\mapsto \TT'$ then $\TT \approx \TT'$.

Here is why this is an appropriate notion of equivalence.
 To each perfect model there is a pair of associated $W$-graphs
$\Upsilon^\m(\sP)$ and $ \Upsilon^\n(\sP)$ \cite{MZ2020}.
If $\sP$ and $\sP'$ are equivalent perfect models for $W$ then there is a canonical bijection between  
the sets of vertices in the unions
$\Upsilon^\m(\sP) \sqcup \Upsilon^\n(\sP)$  and $\Upsilon^\m(\sP') \sqcup \Upsilon^\n(\sP')$.
This bijection restricts on
each weakly-connected component of the underlying graphs
to a map that
 is either an isomorphism or an anti-isomorphism onto its image \cite[Cor. 3.35]{MZ2020}. 
% In particular, this bijection maps cells to cells.

\subsection{Factorizable model triples}

A character of a finite group is \defn{multiplicity-free} if it is a sum of distinct irreducible characters.
We say that a model triple $\TT$ is \defn{multiplicity-free}
if its character $\chi^\TT$ is multiplicity-free. All model triples appearing in a perfect model must have this property.

The \defn{Coxeter diagram} of $(W,S)$ is the graph with vertex set $S$ that has an edge 
between two elements $s,t \in S$
whenever $st\neq ts$; this edge is labeled by the order of the product $st \in W$.
The \defn{irreducible components} of $(W,S)$ are the subsystems $(W_J,J)$
where $ J  \subseteq S$
is the set of vertices in a connected component of the Coxeter diagram.
 A Coxeter system $(W,S)$ is \defn{irreducible} if it has exactly one irreducible component. 

%\begin{theorem}
%Suppose $(W,S)$ is an irreducible finite Coxeter system and $J\subseteq S$ has $|S - J| \geq 2$.
%%If $\Ind_{W_J}^W(\chi)$ is multiplicity-free for any $\chi \in \Irr(W_J)$ then all but at most one element of $S-J$ is central in $W$.
%%If $(W,S)$ is irreducible 
%Then $\Ind_{W_J}^W(\chi)$ is not multiplicity-free for any $\chi \in \Irr(W_J)$.
%\end{theorem}
%
%\eric{TODO: there are some relevant references to include, which consider the same result but just for the trivial character $\chi=\one$ }
%
%\begin{proof}
%case by case proof, put proof in an appendix so can refer to A/BC/D notation for characters
%\end{proof}

For $z = (w,\theta) \in W^+$ let $\aut(z) := \theta$.
Suppose $\TT=(J,\cK,\sigma)$ is a model triple.
Then the set $\{ \aut(z) : z \in \cK\}$ has just one element,
which we denote by $\aut(\cK) \in \Aut(W_J,J)$.
We say that $\TT$ is \defn{factorizable} if $\aut(\cK)$ preserves each irreducible component of $(W_J,J)$.

\begin{theorem}
\label{tech2-thm}
 If $W$ is irreducible then every multiplicity-free model triple for $W$ is factorizable. 
\end{theorem}

We prove this result in Section~\ref{app-sect}.

\subsection{Models for reducible groups}\label{reducible-sect}

Let $(W,S)$ be a finite Coxeter system.
Suppose $L_1$, $L_2$, \dots, $L_k$ are disjoint, nonempty sets such that $S = L_1\sqcup L_2\sqcup \dots \sqcup L_k$
and every $s \in L_i$ commutes with every $t \in L_j$ for all $1 \leq i < j \leq k$.
Let $W_i = W_{L_i}$ for $i \in [k]$. 
The subsystems $(W_i,L_i)$ might be the irreducible factors of $(W,S)$, for example, or they might be larger subgroups.

Each automorphism of $W_i$ extends to an automorphism of $W$ fixing all elements of $W_{j}$ for $ i \neq j$,
so we may view $W_i^+\subseteq W^+$ 
and $\sI_i := \sI(W_i, L_i)\subseteq\sI=\sI(W,S)$.
%for each $ i \in [k]$.
Each $w \in W$ can be written uniquely as $w =w_1w_2\cdots w_k$ with $w_i \in W_{i}$,
so given functions $f_i: W_{i}\to \CC$  for $i \in [k]$  
we may define $f_1 \otimes f_2 \otimes \cdots\otimes f_k : W \to \CC$ by  $w \mapsto f_1(w_1)f_2(w_2)\cdots f_k(w_k)$.
This gives a bijection 
\[\Irr(W_{1}) \times \Irr(W_{2})\times \cdots \times \Irr(W_k) \to \Irr(W).\]
If $\TT_i = (J_i,\cK_i,\sigma_i)$
is a model triple for $(W_{i},L_i)$ for each $i \in [k]$,
then we define 
\[\TT_1 \otimes \TT_2\otimes\dots \otimes \TT_k := ( J, \cK, \sigma)\]
 where
 $
  J:=J_1 \sqcup J_2\sqcup \dots \sqcup J_k,
$ $ \cK  := \cK_1 \cK_2 \cdots \cK_k,
  $ and 
$  \sigma :=\sigma_1 \otimes \sigma_2\otimes \dots\otimes \sigma_k.
  $
Note that $\cK$ is a well-defined subset of $\sI(W_{J}, J)$, although the latter might not be a subset of  $ \sI \subseteq W^+$
if there are Coxeter automorphisms of $W_J$ that do not extend to $W$.

It is straightforward to see that if $\TT_i$ is a factorizable model triple for $(W_i,L_i)$ for each $i \in [k]$
then 
$\TT_1 \otimes \TT_2\otimes\dots \otimes \TT_k$ is a factorizable model triple for $(W,S)$
with  $\chi^{\TT_1 \otimes \TT_2\otimes\dots \otimes \TT_k} = \chi^{\TT_1}\otimes \chi^{\TT_2}\otimes \cdots \otimes \chi^{\TT_k}$. Every factorizable model triple for $(W,S)$ arises in this way.

\begin{theorem}\label{red-thm}
A finite Coxeter system $(W,S)$ has a perfect model if and only if each of its irreducible factors has a perfect model.
\end{theorem}

\begin{proof}
If each irreducible factor of $(W,S)$ has a perfect model then a perfect model for $W$ is obtained
by tensoring together the corresponding model triples. 

Suppose instead that $(W,S)$ is a reducible Coxeter system with a perfect model.
Then there exists a nonempty subset $S'\subsetneq S$ such that $(W_{S'},S')$ is irreducible.
Let $S'' := S\setminus S'$.
We will show that $W_{S'}$ also has a perfect model. This is nontrivial primarily because 
although $W=W_{S'}\times W_{S''}$, the extended group $W^+$
is not always isomorphic to $W^+_{S'}\times W^+_{S''}$.

Suppose $\TT=(J, \cK, \sigma)$ is a model triple for $W$.
Let  $\theta := \aut(\cK) \in \Aut(W_J,J)$. Then
we can express $J = A\sqcup B \sqcup C\sqcup D$ for disjoint subsets $A,B\subseteq S'$
and $C,D\subseteq S''$ with $\theta(A) = A$, $\theta(B) = C$, $\theta(C) = B$, and $\theta(D) =D$.

In this setup $(W_B,B) \cong (W_C,C)$
and all elements $a \in A$, $b \in B$, $c \in C$, and $d \in D$ must pairwise commute.
 Additionally, the minimal-length element of $\cK$ 
must have the form $z_{\min} := (z_Az_D,\theta)$ for some $z_A \in W_A$ and $z_D \in W_D$.
 Let $\theta_A = \theta|_{W_A}$ and $\theta_D = \theta|_{W_D}$.
Then the centralizer of $z_{\min}$ in $W_{J}=W_{A}\times W_{B} \times W_{C} \times W_D$ is 
\[H := C_{W_A}((z_A,\theta_A)) \times \Delta_\theta(W_B \times W_C)  \times  C_{W_D}((z_D,\tau_D))
\]
where
$ \Delta_\theta(W_B \times W_C) := \{ b\cdot\theta(b) : b \in W_B\} \subseteq W_B\times W_C$.
Define  
\[
\sigma_A := \Res^{W_J}_{W_A}(\sigma),\quad
\sigma_D := \Res_{W_D}^{W_J}(\sigma),\quand
\sigma_B (b) = \sigma(b \cdot \theta(b))\text{ for } b\in W_B.\]
Since all characters of finite Coxeter groups are real-valued, Frobenius reciprocity implies that
\[ \chi^\TT = \Ind_H^W \Res_H^{W_{J}}(\sigma) = 
\Ind_{W_J} ^W \Bigl( \chi^{\TT}_A \otimes \Bigl(\sum_{\psi \in \Irr(W_B)}  \sigma_B\psi  \otimes  \psi\circ\theta \Bigr) \otimes \chi^{\TT}_D\Bigr)
\]
where
$
\chi^{\TT}_A := \Ind_{C_{W_A}((z_A,\theta_A))}^{W_A} \Res^{W_A}_{C_{W_A}((z_A,\theta_A))}(\sigma_A)
$
and
$
\chi^{\TT}_D := \Ind_{C_{W_D}((z_D,\theta_D))}^{W_D} \Res^{W_D}_{C_{W_D}((z_D,\theta_D))}(\sigma_D).$
Since $W = W_{S' \sqcup S''} = W_{S'}\times W_{S''}$ we can rewrite this as
\[
\ba
\chi^{\TT} &%= \sum_{\psi \in \Irr(W_B)} \Ind_{W_{A}\times W_{B} \times W_{C} \times  W_D} ^W \( \chi^{\TT_A} \boxtimes \sigma_B \psi \boxtimes \bar \psi\theta^{-1} \boxtimes \chi^{\TT}_D\)\\&
= \sum_{\psi \in \Irr(W_B)} \Ind_{W_{A}\times W_{B}}^{W_{S'}}  \( \chi^{\TT}_A \otimes \sigma_B \psi\)  \otimes
 \Ind_{W_{C}\times W_{D}}^{W_{S''}}  \(  \psi\circ \theta \otimes \chi^{\TT}_D\).
\ea
\]

A basis for the class functions on $W$ is given by the irreducible characters $ \chi \otimes \psi  $ for $ \chi \in \Irr(W_{S'})$ and $\psi \in \Irr(W_{S''})$. 
Let $\cL$ be the linear map from class functions on $W$ to class functions on $W_{S'}$ that sends
$ \chi \otimes \psi \mapsto \chi$ if $\psi =\one$ and to zero otherwise. Since for $\psi \in \Irr(B)$ we have
\[
\left\langle \one,    \Ind_{W_{C}\times W_{D}}^{W_{S''}}  \(  \psi\circ \theta \otimes \chi^{\TT}_D\)\right\rangle_{W_{S''}}
=
\left\langle \one,    \psi\circ \theta \otimes \chi^{\TT}_D\right\rangle_{W_C\times W_D}
= \begin{cases} 
\left\langle \one, \chi^{\TT}_D\right\rangle_{W_D} &\text{if $\psi= \one$} \\ 
  0&\text{if $\psi \neq \one$},\end{cases}
  \] it follows that
$ \cL(\chi^\TT) =\langle \one, \chi^{\TT}_D\rangle_{W_D}    \Ind_{W_{A}\times W_{B}}^{W_J}  \( \chi^{\TT}_A \otimes \sigma_B \) . $
Define $\cK'$ to be the $W_{A\sqcup B}$-conjugacy class of $(z_A, \theta'_A)$ where
$ \theta'_A : ab\mapsto \theta(a)b$ for $a \in W_A$ and $b \in W_B$.
Then
$
\Ind_{W_{A}\times W_{B}}^{W_J}  \( \chi^{\TT_A} \otimes \sigma_B \) 
$
is just the character of the model triple 
$ \TT' := (A \sqcup B, \cK', \sigma_A\otimes \sigma_B)$ for $(W_{S'},S')$.

Let $\sP$ be a perfect model for $W$. 
Then $\langle \one_{W_D}, \chi^{\TT}_D\rangle_{W_D} \in \{0,1\}$ for all $\TT \in \sP$ and 
$\sum_{\TT \in \sP} \chi^\TT = \sum_{\chi \in \Irr(W_{S'}) } \sum_{\psi \in \Irr(W_{S''})} \chi \otimes \psi.$
Define $\sP' = \{\TT' :  \TT \in \sP \text{ has } \langle \one_{W_D}, \chi^{\TT}_D\rangle_{W_D} = 1\}$. Then 
$
\sum_{\SS \in \sP'} \chi^\SS =  \sum_{\TT \in \sP} \cL(\chi^\TT) = \cL\(\sum_{\TT \in \sP} \chi^\TT\) = \sum_{\chi \in \Irr(W_{S'})} \chi
$
so $\sP'$ is a perfect model for $W_{S'}$.
\end{proof}

 \subsection{Classical Weyl groups}

By Theorem~\ref{red-thm} our main classification problem reduces to  understanding which irreducible finite Coxeter groups have perfect models---in particular those groups in the three infinite families of classical Weyl groups.

Let $(i,i+1) $ for $i \in \ZZ$ denote the permutation of $ \ZZ$ interchanging $i$ and $i+1$ while fixing all other points.
The group of permutations of $\ZZ$ with finite support is $S_\ZZ := \langle (i,i+1) : i \in \ZZ\rangle$.
We realize the  classical Weyl groups as subgroups of $S_\ZZ$ in the following way.
Define $s_0 := (-1,1)$. For integers $i>0$ let $s_i := (i,i+1)(-i,-i-1)$ and $s_{-i} := (i,-i-1)(-i,i+1)$.
For $n\geq 1$ set 
\[S_{n+1} := \langle s_1,s_2,\dots,s_{n} \rangle \quand \W_n := \langle s_0,s_1,s_2\dots,s_{n-1} \rangle.\]
For each $n\geq 2$
set \[ \WD_n := \langle s_{-1},s_1,s_2,\dots,s_{n-1}\rangle\quad\text{where}\quad S_1 = \WD_1 := \{ 1 \} \subseteq S_\ZZ.\]
These are the finite Coxeter groups of types $\A_{n-1}$ (for $n\geq 1$), $\BC_n$ (for $n\geq 2$), and $\D_n$ (for $n\geq 4$) relative to the given simple generators. 
Note that 
\[
\W_1 = \langle s_0\rangle\cong  S_2,
\quad 
\WD_2 = \langle s_{-1},s_1\rangle \cong S_2 \times S_2,
\quand \WD_3 = \langle s_{-1},s_2, s_1\rangle \cong S_4.
\]
The elements of $\W_n$ are the permutations $w \in S_\ZZ$  with $w(-i) = -w(i)$ for all $i \in [n] := \{1,2,\dots,n\}$ and with $w(i) =i$ if $|i| > n$.
The group $S_n$ is the subgroup of such permutations which preserve $[n]$, and the group $\WD_n$ is the subgroup
of elements $w \in W_n$ with an even number of sign changes, that is, with $|\{ i \in [n] : w(i) < 0\}|\equiv 0 \modu 2)$.

Assume $W \in \{ S_n, \W_n, \WD_n\}$ is one of these classical Weyl groups.
Each element $w \in W$ is uniquely determined by its \defn{one-line representation}, which is the word $w_1w_2\cdots w_n$ where $w_i = w(i)$ and where we write negative numbers $-1,-2,-3,\dots$ as $\bar 1,\bar 2,\bar 3,\dots$, respectively.
When $i \geq 0$ and $s_i \in W$, one has $\ell(ws_i) < \ell(w)$ for $w \in W$ if and only if $w_i > w_{i+1}$;
when $w \in W=\WD_n$ one has $\ell(ws_{-1}) < \ell(w)$ if and only if $-w_1 > w_2$ \cite[Props. 8.1.2 and 8.2.2]{CCG}.

%\subsection{Central perfect involutions}
%
%Let $\sZ = \sZ(W,S):=\{ z=z^{-1} \in W^+ : wz=zw\text{ for all }w \in W\}$. 
%If $z \in \sZ$ then $ \{ z\}$ is a perfect $W$-conjugacy class.
%For $J\subseteq S$ let $\sZ_J := \sZ(W_J,J)$.
%
%\begin{lemma}
%Suppose  $\TT = (J,\cK,\sigma)$ is a model triple for $(W,S)$ 
%and $z \in \sZ_J$ has $\aut(\cK)(z) =z$.
%Then $\TT' := (J, \cK z, \sigma)$ is another model triple for $(W,S)$ with the same character as $\TT$.
%\end{lemma}
%
%It does not necessarily hold that $\TT\equiv \TT'$ in this situation.
%
%\begin{proof}
%
%\end{proof}
%

\section{Model classification in type A}\label{a-sect}

In this section we fix a positive integer $n$ and consider the Coxeter group $W = S_n$ with generating set $S = \{s_1,s_2,\dots,s_{n-1}\}$. Our main result here is Theorem~\ref{a-thm}.

\subsection{Perfect conjugacy classes in type A}

The longest element in $S_n$ is the reverse permutation $w_0 = n\cdots 321$
and the only nontrivial Coxeter automorphism is $\Ad(w_0)$.
Let $\cK_{\id}^{S_n} := \{1\}$ and when $n$ is even define $\cK^{S_n}_{\fpf}$ to be the set of fixed-point-free involutions in $S_n$.
Let $\cK^{S_n}_{\id^+} :=  \{w_0^+ \}$  and  $\cK^{S_n}_{\fpf^+} :=  \cK^{S_n}_{\fpf}\cdot w_0^+$. 
The unique minimal-length elements
in
$\cK^{S_n}_{\fpf}$ and $\cK^{S_n}_{\fpf^+}$ are then
\[
s_1s_3s_5\cdots s_{n-1} \in S_n
\quand (1,\Ad(w_0)) \in S_n^+.\]
The perfect conjugacy classes in $S_n^+$ are $\cK^{S_n}_{\id}$ and $\cK^{S_n}_{\id^+}$,
together with $\cK^{S_n}_{\fpf}$ and $\cK^{S_n}_{\fpf^+}$ when $n$ is even \cite[Ex. 9.2]{RV}.
One has  $\cK^{S_1}_{\id} = \cK^{S_1}_{\id^+}$
and  $\cK^{S_2}_{\id} = \cK^{S_2}_{\fpf^+}$ and $\cK^{S_2}_{\id^+} = \cK^{S_2}_{\fpf}$.

%Suppose $\alpha_1,\alpha_2,\dots,\alpha_l$ are positive integers summing to $n$
%and $z_i = (w_i, \theta_i) \in S_{\alpha_i}^+$ for each $i \in [l]$.
%Define $z_1 \boxtimes z_2$ to be the element $(w,\theta) \in S_{\alpha_1+\alpha_2}^+$
%where $w \in S_{\alpha_1+\alpha_2}$ is the permutation mapping 
%$x \mapsto w_i(x)$ for $x \in [\alpha_1]$ and $\alpha_1 + x \mapsto \alpha_1 + w_2(x)$ for $x \in [\alpha_2]$,

\subsection{Model indices in type A}\label{mti-a-sect}

%Let $\MTI(S_n)$ denote the set of $3\times 1$ and $3\times 2$ arrays of the form
%\[
% \Psi = \left[\begin{smallmatrix} n  \\ 
%\kappa  \\ 
%\sigma 
%\end{smallmatrix}\right]
%\quord
% \Phi = \left[\begin{smallmatrix} \alpha_1 & \alpha_2  \\ 
%\beta_1 & \beta_2  \\ 
%\gamma_1 & \gamma_2
%\end{smallmatrix}\right]
%\]
%where $\alpha_1$ and $\alpha_2$ are positive integers with $\alpha_1+\alpha_2=n$;
%$\kappa$, $\beta_1$, and $\beta_2$ are symbols in the set $\{ \id, \id^+, \fpf, \fpf^+\}$,
%and where $\sigma,\gamma_1,\gamma_2 \in \{\one,\sgn\}$ are linear characters.
%We further require that $\kappa \in \{ \id,\id^+\}$ when $n$ is odd and that
% $\beta_i \in \{ \id,\id^+\}$ when $\alpha_i$ is odd.
% From this data, we define

Let $\MTI(S_n)$ denote the set of 3-line arrays of the form
\[
\Theta = \left[\begin{smallmatrix} \alpha_1 & \alpha_2 & \dots & \alpha_l \\ 
\beta_1 & \beta_2 & \dots & \beta_l \\ 
\gamma_1 & \gamma_2 & \dots &\gamma_l
\end{smallmatrix}\right]
\]
where $\alpha_1,\alpha_2,\dots,\alpha_l$ are positive integers summing to $n$,
each $\beta_i$ is a symbol in  $\{\id, \id^+, \fpf, \fpf^+\}$, 
and each $\gamma_i \in \{\one, \sgn\}$ subject to 
the requirement that
 $\beta_i \in \{ \id, \id^+\}$
when $\alpha_i \in  \{1,3,5,7\dots\}$.
We refer to $\Theta $ as a \defn{model index} for $S_n$.
Let $\Theta = \left[\begin{smallmatrix} \alpha_1 & \alpha_2 & \dots & \alpha_l \\ 
\beta_1 & \beta_2 & \dots & \beta_l \\ 
\gamma_1 & \gamma_2 & \dots &\gamma_l
\end{smallmatrix}\right] \in \MTI(S_n)$. For  $i \in [l]$ define 
\[
%\label{Ji-eq}
\ba
J_i &:= \{ s_j \in S_n :  \alpha_1 + \dots + \alpha_{i-1} < j < \alpha_1 + \dots + \alpha_i \} \ea
 \]
and let $\varphi_i$ be the isomorphism $  S_{\alpha_i } \to \langle J_i\rangle$
 mapping $s_j \mapsto s_{\alpha_1 + \alpha_2 + \dots +\alpha_{i-1} + j}$ for $j \in [\alpha_i-1]$.
 This extends to an isomorphism $S_{\alpha_i}^+ \cong \langle J_i\rangle^+$ 
 via $(w,\theta) \mapsto (\varphi_i(w), \varphi_i \theta\varphi_i^{-1})$.
Let $\cK_i$ be the image of $\cK^{S_{\alpha_i}}_{\beta_i}$ under $\varphi_i$.
%and define
%$\sigma_i : \langle J_i \rangle \to \{ \pm 1\}$ to be the linear character with $\sigma_i(s) = \gamma_i(s)$
%for each $s \in J_i$.
Using the notation in Section~\ref{reducible-sect}, we define a model triple  %a factorizable model triple for $S_n$ by 
\[ \TT^\Theta :=( J_1, \cK_1, \gamma_1) \otimes ( J_2, \cK_2, \gamma_2) \otimes \dots \otimes ( J_l, \cK_l, \gamma_l)
. \]
Every factorizable model triple (and therefore
every multiplicity-free model triple by Theorem~\ref{tech2-thm})
 for $S_n$ arises as $\TT^\Theta$ for some $\Theta \in \MTI(S_n)$.
This representation is almost unique. However, $\TT^\Theta$ is unaltered by the following modifications to $\Theta$:
\begin{itemize}
\item when $\alpha_i=1$, changing $\beta_i=\id$ to $\id^+$ (or vice versa) or $\gamma_i=\one$ to $\sgn$ (or vice versa);

\item when $\alpha_i=2$,  changing $\beta_i = \id$ to $ \fpf^+$ (or vice versa) or $\beta_i = \id^+$ to $ \fpf$ (or vice versa).
\end{itemize}
 %Define  $\chi_\A^\Theta := \chi^{\TT^\Theta}$.
  In view of Theorem~\ref{tech2-thm},
 the character 
 \[\chi_\A^\Theta\ := \chi^{\TT^\Theta}\] is never multiplicity-free if $\Theta$ has more than two columns.
 However, it will be useful later to allow these more general indices.

%Given $\Theta = \left[\begin{smallmatrix} \alpha_1  & \dots & \alpha_l \\ 
%\beta_1  & \dots & \beta_l \\ 
%\gamma_1 & \dots &\gamma_l
%\end{smallmatrix}\right]$,
Define 
$\Theta^\ast := \left[\begin{smallmatrix} \alpha_l & \dots  & \alpha_2& \alpha_1 \\ 
 \beta_l & \dots  & \beta_2 & \beta_1 \\ 
\gamma_l & \dots  &\gamma_2&\gamma_1
\end{smallmatrix}\right]$, 
$\Theta^\vee := \left[\begin{smallmatrix} \alpha_l & \dots  &\alpha_2& \alpha_1 \\ 
\beta^\vee_l & \dots  & \beta^\vee_2& \beta^\vee_1 \\ 
\gamma_l & \dots  & \gamma_2 &\gamma_1
\end{smallmatrix}\right]$,
and $\overline\Theta := \left[\begin{smallmatrix} \alpha_1 & \alpha_2  & \dots & \alpha_l \\ 
\beta_1 & \beta_2  & \dots & \beta_l \\ 
\bar\gamma_1 &\bar\gamma_2 & \dots &\bar\gamma_l
\end{smallmatrix}\right]$ 
where 
\[ \beta^\vee_i:=\begin{cases} \fpf &\text{if $\beta_i = \fpf^+$} \\
\fpf^+ &\text{if $\beta_i = \fpf$} \\
\id &\text{if $\beta_i = \id^+$} \\
  \id^+ &\text{if $\beta_i = \id$} \end{cases}
  \quand
  \bar \gamma_i:= \gamma_i \sgn = \begin{cases} \one &\text{if $\gamma_i = \sgn$} \\
\sgn &\text{if $\gamma_i = \one$.} \end{cases}\]
 It is straightforward to check %for $\Theta \in \MTI(S_n)$
 that 
 $\TT^{\Theta^\ast} = (\TT^\Theta)^{\Ad(w_0)}$,    
  $\TT^{\Theta^\vee} = (\TT^\Theta)^\vee$,
  and $\TT^{\overline\Theta} = \overline{\TT^\Theta}$.
 Likewise, if $ \Theta'$ is formed from $\Theta$ by changing any entries  $\id^+$ in the second row to $\id$,
 then $\TT^{\Theta} \equiv \TT^{\Theta'}$ since  $\cK^{S_{\alpha_i}}_{\id}$ and $\cK^{S_{\alpha_i}}_{\id^+}$ are singleton 
 sets with the same centralizers in $S_{\alpha_i}$.\footnote{
  If $\Theta''$ is formed from $\Theta$ 
 by changing any $\fpf^+$ entries    to $\fpf$, then $\chi_\A^\Theta = \chi_\A^{\Theta''}$ always holds
 but we could have $\TT^\Theta\not\equiv \TT^{\Theta''}$.
If $(J,\cK,\sigma) \equiv (J',\cK',\sigma')$ then $J=J'$ and 
 the $W_J$-centralizers of the minimal-length elements of $\cK$ and $\cK'$ must be equal. 
 For $\cK = \cK^{S_{\alpha_i}}_{\fpf}$ and $\cK'=\cK^{S_{\alpha_i}}_{\fpf^+}$
 these centralizers are conjugate but not equal.
 }

%Define $\chi_\A^\Theta := \chi^{\TT^\Theta}$.
If $\Theta,\Psi \in \MTI(S_n)$
then we write $\Theta \equiv \Psi$ if 
$\TT^\Theta \equiv \TT^\Psi$, 
$\Theta\sim \Psi$ if $\TT^\Theta \sim \TT^\Psi$, and
 $\Theta \approx\Psi$ if $\TT^\Theta\approx \TT^\Psi$.
 In the second two cases we say that $\Theta$ and $\Psi$ are \defn{strongly equivalent} and \defn{equivalent}.
 If $\cM$ and $\cM'$ are sets of model indices then we write 
 $\cM \equiv \cM'$,
 $\cM \sim \cM'$,
 or
 $\cM \approx \cM'$
 if there is a bijection $\cM \to \cM'$
 with
 $\Theta \equiv \Theta'$, $\Theta \sim \Theta'$, or $\Theta \approx \Theta'$, respectively, whenever
$\Theta\mapsto \Theta'$.

%\begin{lemma}
%Let $\Theta \in \MTI(S_n)$. Then 
%$\TT^{\Theta^\vee} = (\TT^\Theta)^\vee$, $\TT^{\Theta^\ast} = (\TT^\Theta)^{\Ad(w_0)}$, and  $\TT^{\overline\Theta} = \overline{\TT^\Theta}$,
%so $\Theta \sim \Theta^\vee \sim \Theta^\ast  \approx \overline\Theta$.
% Form $ \Theta'$ from $\Theta$ by changing any entries equal to $\id^+$  to $\id$.
%Form $\Theta''$ from $\Theta'$ by replacing any entries equal to $\fpf^+$  to $\fpf$.
%Then $\Theta \equiv \Theta'$ and $\chi_\A^\Theta=\chi_\A^{\Theta'} = \chi_\A^{\Theta''}$.
%\end{lemma}
%
%\begin{proof}
%straightforward
%\end{proof}

 \subsection{Littlewood-Richardson coefficients in type A}\label{LR-sect1}

%\subsection{Irreducible characters in type A}\label{char-a-sect}

The irreducible characters of $S_n$ are indexed by partitions of $n$,
that is, by weakly decreasing sequences of positive integers $\lambda = (\lambda_1 \geq \lambda_2 \geq \dots \geq \lambda_l >0)$
with $  \lambda_1 + \lambda_2 + \dots + \lambda_l = n$. 
The diagram of a partition $\lambda$ is the set $\D_\lambda := \{ (i,j) : i >0\text{ and }1\leq j \leq \lambda_i\}$.
The \defn{transpose} of $\lambda$ is the unique partition $\lambda^\top$
with $\D_{\lambda^\top} = \{ (j,i) : (i,j) \in \D_\lambda\}$.

We write $\lambda \vdash n$ to indicate that $\lambda$ is a partition of $n$,
and $\chi^\lambda $ for the irreducible character of $S_n$ indexed by $\lambda \vdash n$ following the standard
construction explained in \cite[\S5.4]{GeckP}.
The linear characters of $S_n$ are $\one = \chi^{(n)}$ and $\sgn = \chi^{(1^n)}$
where $(1^n) := (1,1,\dots,1) \vdash n$.
It is well-known that $\chi^\lambda \sgn = \chi^{\lambda^\top}$ for any $\lambda \vdash n$.

%\subsection{Littlewood-Richardson coefficients in type A}\label{LR-sect1}

Let $p, q \in \NN$ with $n=p+q$.
We identify $S_p\times S_q$ with the subgroup $\langle s_i : p\neq i \in [n-1]\rangle \subseteq S_{n}$
and write $u\times v \in S_{p+q}$ for the image $(u,v) \in S_p\times S_q$ under this inclusion.
Given functions $f : S_p \to \CC$ and $g : S_q \to \CC$
 define $f\boxtimes g : S_p\times S_q \to \CC$ to be the map sending $u\times v \mapsto f(u)g(v)$
 for all $u \in S_p$ and $v \in S_q$.
 %The map $(\chi,\psi) \mapsto \chi \boxtimes \psi$ is a bijection $\Irr(S_p)\times \Irr(S_q) \to \Irr(S_p\times S_q)$.
If $f$ and $g$ are class functions, then we further define
\be f\bullet_\A g = \Ind_{S_p\times S_q}^{S_{p+q}}(f\boxtimes g).\ee
This is a commutative, associative, and bilinear operation.
If $\Theta = \left[\begin{smallmatrix} \alpha_1   & \dots & \alpha_l \\ 
\beta_1   & \dots & \beta_l \\ 
\gamma_1 & \dots &\gamma_l
\end{smallmatrix}\right]\in \MTI(S_n)$ then  
\be\label{theta-bullet-eq}
\chi_\A^\Theta = \chi_\A^{ \left[\begin{smallmatrix} \alpha_1  \\ 
\beta_1   \\ 
\gamma_1 
\end{smallmatrix}\right]}
\bullet_\A 
\chi_\A^{ \left[\begin{smallmatrix} \alpha_2  \\ 
\beta_2   \\ 
\gamma_2 
\end{smallmatrix}\right]}
\bullet_\A \cdots \bullet_\A
\chi_\A^{ \left[\begin{smallmatrix} \alpha_l  \\ 
\beta_l   \\ 
\gamma_l 
\end{smallmatrix}\right]}.
\ee
Let $c_{\lambda\mu}^\nu \in \NN$  denote the \defn{Littlewood-Richardson coefficients}
satisfying $\chi^\lambda\bullet_\A  \chi^\mu = \sum_{\nu} c_{\lambda\mu}^\nu \chi^\nu$.
The \defn{Pieri rules} \cite[Ex. 6.3]{GeckP} state that if $p\in \NN$ and $\mu = (p)$ (respectively, $\mu =(1^p)$)
then
$ c^\nu_{\lambda\mu} = 1$  if $\D_\nu \setminus \D_\lambda$ consists of $p$ cells in distinct columns (respectively, rows)
and otherwise $ c^\nu_{\lambda\mu} = 0$.

%Let $\cR^\A_n$ denote the $\CC$-vector space of class functions $S_n\to\CC$.
%Then $\bullet_\A$ makes $\cR^\A := \bigoplus_{n \in \NN} \cR^\A_n$ into a graded algebra.
%%with unit $\chi^{\emptyset}$ where $\emptyset$ denotes the empty partition.
%Let $\Sym(x)$ denote the graded algebra of symmetric power series in $\CC[[x_1,x_2,\dots]]$ of bounded degree.
%It is well-known \cite{?} that the linear map with $\chi^\lambda \mapsto s_\lambda$ sending each irreducible character 
%to a Schur function is an isomorphism $\cR^\A \xrightarrow{\sim} \Sym(x)$.

\subsection{Perfect models in type A}

%Fix a model index $\Theta \in \MTI(S_n)$ and 
%write $\ncols(\Theta)$ for its number of columns.
%
%\begin{lemma} If  $\ncols(\Theta) > 2$ then $\chi_\A^\Theta$ is not multiplicity-free.
%\end{lemma}
%
%\begin{proof}
%By the $|S-J| \geq 2$ theorem.
% \end{proof}

%Define $\Theta \in \MTI(S_n)$ to be multiplicity-free if $\TT^\Theta$ is multiplicity-free.

%In the next two lemmas we fix an index $\Theta \in \MTI(S_n)$.
When $n$ is even,
let $\EvenRows(n)$ denote the set of partitions $\lambda \vdash n$ whose parts $\lambda_i$ are all even,
and let $\EvenCols(n) = \{ \lambda^\top : \lambda \in \EvenRows(n)\}$
where $\lambda^\top$ is the transpose of a partition $\lambda$.
Then 
\be\label{irs-eq}
\chi_\A^{\left[\begin{smallmatrix} 
n   \\ 
\fpf    \\ 
\one  
\end{smallmatrix}\right]}
=
\chi_\A^{\left[\begin{smallmatrix} 
n   \\ 
\fpf^+    \\ 
\one  
\end{smallmatrix}\right]} = \sum_{\lambda \in \EvenRows(n)} \chi^\lambda
\quand
\chi_\A^{\left[\begin{smallmatrix} 
n   \\ 
\fpf    \\ 
\sgn  
\end{smallmatrix}\right]}
=
\chi_\A^{\left[\begin{smallmatrix} 
n   \\ 
\fpf^+    \\ 
\sgn  
\end{smallmatrix}\right]} = \sum_{\lambda \in \EvenCols(n)} \chi^\lambda\ee
by \cite[Lem. 1]{IRS}.
Fix a model index $\Theta \in \MTI(S_n)$ and 
write $\ncols(\Theta)$ for its number of columns.

\begin{lemma}\label{a-lem}
The character $\chi_\A^\Theta$ is not multiplicity-free if 
$\ncols(\Theta)>2$. Suppose $\Theta =  \left[\begin{smallmatrix} 
\alpha_1 & \alpha_2 \\ 
\beta_1 & \beta_2   \\ 
\gamma_1 & \gamma_2
\end{smallmatrix}\right] $ has exactly two columns. 
Then $\chi_\A^\Theta$ is not multiplicity-free if
any of the following holds:
\ben
\item[(a)] $\alpha_1 \in \{4,6,8,\dots\}$, $\beta_1\in \{\fpf,\fpf^+\}$, 
$\alpha_2 \geq 2$, and $\gamma_1=\gamma_2$.

\item[(b)] $\alpha_2 \in \{4,6,8,\dots\}$, $\beta_2\in \{\fpf,\fpf^+\}$, 
$\alpha_1 \geq 2$, and $\gamma_1=\gamma_2$.

\item[(c)] $\alpha_1,\alpha_2 \in \{4,6,8,\dots\}$ and $\beta_1,\beta_2 \in \{\fpf,\fpf^+\}$.
\een
\end{lemma}

\begin{proof}
Suppose $\alpha_1 \in \{4,6,8,\dots\}$, $\beta_1\in \{\fpf,\fpf^+\}$, and
$\alpha_2 \geq 2$. Then it follows from \eqref{irs-eq} that
 $ \chi_\A^{ \left[\begin{smallmatrix} \alpha_1  \\ 
\beta_1   \\ 
\one 
\end{smallmatrix}\right]}$ has $ \chi^{(\alpha_1)} + \chi^{(\alpha_1-2,2)}$ as a constituent 
and 
$ \chi_\A^{ \left[\begin{smallmatrix} \alpha_2  \\ 
\beta_2  & \\ 
\one 
\end{smallmatrix}\right]}$ has $ \chi^{(\alpha_2)}$ as a constituent,
regardless of the value of $\beta_2\in\{\id,\id^+,\fpf,\fpf^+\}$.
But $ \chi^{(\alpha_1)} \bullet_A  \chi^{(\alpha_2)}$ and 
$ \chi^{(\alpha_1-2,2)} \bullet_A  \chi^{(\alpha_2)}$ both have $ \chi^{(\alpha_1+\alpha_2-2,2)}$ as a constituent, 
so $\chi_\A^\Theta$ is not multiplicity-free when $\gamma_1=\gamma_2=\one$ by \eqref{theta-bullet-eq}.
Since $\chi_\A^\Theta = \chi_\A^{\Theta^\ast} = \chi_\A^{\overline\Theta}$,
it follows that this character
is  not multiplicity-free in cases (a) and (b).

It remains to show that $\chi_\A^\Theta$ is not multiplicity-free
when $\alpha_1,\alpha_2 \in \{4,6,8,\dots\}$, $\beta_1,\beta_2 \in \{\fpf,\fpf^+\}$,
and $\gamma_1 \neq \gamma_2$.
In this case $\chi_\A^\Theta = 
\sum_{\lambda \in \EvenRows(p)}
 \sum_{\mu \in \EvenCols(q)}
  \sum_{\nu \vdash p+q} c_{\lambda\mu}^\nu \chi^\nu$ 
  for some $\{p,q\} = \{\alpha_1,\alpha_2\}$ by 
  \eqref{theta-bullet-eq} and \eqref{irs-eq}.
  If $p\geq q$ then for $\nu = (p,q/2,q/2)$ and $\mu=(q/2,q/2)$ 
  we have $c_{\lambda\mu}^\nu = 1$ for both
  $\lambda = (p)$ and $\lambda =(p-2,2)$
since
 $
 \chi^{(p-2,2)}  = \chi^{(p-2)} \bullet_\A \chi^{(2)} - \chi^{(p-1)} \bullet_\A \chi^{(1)}
 $, so $\chi_\A^\Theta$ is not multiplicity-free.
 If $p \leq q$ then we reach the same conclusion by 
 considering $\nu = (q,p/2,p/2)^\top$, $\lambda=(p/2,p/2)^\top$, 
and 
  $\mu = (q)^\top$ or $(q-2,2)^\top$.
%
%  
%Stembridge has classified precisely which partitions $\lambda$ and $\mu$
%have $c_{\lambda\mu}^\nu \in \{0,1\}$ for all $\nu$ \cite[Thm. 3.1]{Stembridge}. 
%It is a tedious but elementary exercise to show that $c_{\lambda\mu}^\nu \geq 2$ 
%for some $\nu$ when 
%$\lambda \in \EvenRows(\alpha_1)$ and $ \mu \in \EvenCols(\alpha_2)$.
\end{proof}

%\begin{lemma} Suppose $\alpha_1 \in \{4,6,8,\dots\}$ and $\beta_1\in \{\fpf,\fpf^+\}$, 
%$\alpha_2 >1$ and $\beta_2 \in \{\id,\id^+\}$, and $\gamma_1=\gamma_2$.
% Then $\chi_\A^\Theta$ is not multiplicity-free.
%\end{lemma}
%
%\begin{proof}
%We have $ \chi_\A^{ \left[\begin{smallmatrix} \alpha_1  \\ 
%\beta_1  & \\ 
%\one 
%\end{smallmatrix}\right]} = \chi^{(\alpha_1)} + \chi^{(\alpha_1-2,2)} + $ other terms, and 
%$ \chi_\A^{ \left[\begin{smallmatrix} \alpha_2  \\ 
%\beta_2  & \\ 
%\one 
%\end{smallmatrix}\right]} = \chi^{(\alpha_2)}$.
%But $ \chi^{(\alpha_1)} \bullet_A  \chi^{(\alpha_2)}$ and 
%$ \chi^{(\alpha_1-2,2)} \bullet_A  \chi^{(\alpha_2)}$ both have $ \chi^{(\alpha_1+\alpha_2-2,2)}$ as a constituent.
%Argument in $\sgn$ case similar.
% \end{proof}
%
%\begin{lemma} If $\alpha_1,\alpha_2 \in \{4,6,8,\dots\}$ and $\beta_1,\beta_2 \in \{\fpf,\fpf^+\}$
% then $\chi_\A^\Theta$ is not multiplicity-free.
%\end{lemma}
%
%\begin{proof}
%If $\gamma_1=\gamma_2$ then claim follows by previous lemma since 
%$ \chi_\A^{ \left[\begin{smallmatrix} \alpha_2  \\ 
%\beta_2  & \\ 
%\gamma_2 
%\end{smallmatrix}\right]} $ has $ \chi_\A^{ \left[\begin{smallmatrix} \alpha_2  \\ 
%\id  & \\ 
%\gamma_2 
%\end{smallmatrix}\right]}=\gamma_2 $ as a constituent.
%Assume $\gamma_1 \neq \gamma_2$. Argument in v2.
% \end{proof}

Let $\OddRows(n,q)$ be the set of partitions $\lambda\vdash n$ with exactly $q$ odd parts,
and define $\OddCols(n,q) = \{ \lambda^\top : \lambda\in\OddRows(n,q)\}$.

\begin{proposition}\label{one-col-a-prop} Suppose $\Theta \in \MTI(S_n)$ 
and $\chi_\A^\Theta$ is multiplicity-free.
Then $\Theta$ is strongly equivalent to a model index
of one of the following forms:
\ben
\item[(a)] $\left[\begin{smallmatrix} n  \\ \id \\ \sigma  \end{smallmatrix}\right] $ 
for either linear character $\sigma\in \{\one,\sgn\}$,
in which case  $\chi_\A^\Theta = \sigma$.
\item[(b)] $\left[\begin{smallmatrix} n  \\ \fpf \\ \one  \end{smallmatrix}\right] $
with $n\in\{4,6,8,\dots\}$,
in which case  $\chi_\A^\Theta = \sum_{\lambda \in \EvenRows(n)} \chi^\lambda$.
\item[(c)] $ \left[\begin{smallmatrix} n  \\ \fpf \\ \sgn  \end{smallmatrix}\right]  $ 
with $n\in\{4,6,8,\dots\}$,
in which case  $\chi_\A^\Theta = \sum_{\lambda \in \EvenCols(n)} \chi^\lambda$.

\item[(d)] $ \left[\begin{smallmatrix} 
k & n-k \\ 
\id & \id   \\ 
\one & \sgn
\end{smallmatrix}\right]$ with $2<k<n-2$, in which case $\chi_\A^\Theta =\chi^{(k+1,1^{n-k-1})} + \chi^{(k,1^{n-k})}$.

\item[(e)] $\left[\begin{smallmatrix} 
k & n-k \\ 
\id & \id   \\ 
\one & \one
\end{smallmatrix}\right]$ with $0<k<n$, in which case $\chi_\A^\Theta =\sum_{j=0}^{\min(k,n-k)}  \chi^{(n-j,j)}$.

\item[(f)] $\left[\begin{smallmatrix} 
k & n-k \\ 
\id & \id   \\ 
\sgn & \sgn
\end{smallmatrix}\right]$ with $0<k<n$, in which case $\chi_\A^\Theta =\sum_{j=0}^{\min(k,n-k)}  \chi^{(2^j,1^{n-2j})}$.

\item[(g)] $\left[\begin{smallmatrix} 
2k & n-2k \\ 
\fpf & \id   \\ 
\one & \sgn
\end{smallmatrix}\right]$ with $0<k<n/2$, in which case $\chi_\A^\Theta =\sum_{\lambda \in \OddRows(n,n-2k)} \chi^\lambda$.

\item[(h)] $\left[\begin{smallmatrix} 
2k & n-2k \\ 
\fpf & \id   \\ 
\sgn & \one
\end{smallmatrix}\right]$ with $0<k<n/2$, in which case $\chi_\A^\Theta =\sum_{\lambda \in \OddCols(n,n-2k)} \chi^\lambda$.

\een
\end{proposition}

\begin{proof}
Lemma~\ref{a-lem} implies that $\Theta$ must have at most two columns
and not more than one entry in the second row equal to $\fpf$ or $\fpf^+$;
moreover, if the second row of $\Theta$ has an entry equal to $\fpf$ or $\fpf^+$
then the two linear characters in the third row must be distinct.
As we can change any entries in the second row of $\Theta$ from $\id^+$ to $\id$
without altering its strong equivalence class, and since $\Theta\sim \Theta^\vee\sim\Theta^\ast$,
it follows that $\Theta$ is strongly equivalent to one
of the cases listed. The reason why case (d) has $2<k<n-2$ rather than $0<k<n$
is because \[ 
\left[\begin{smallmatrix} 
1 & n-1 \\ 
\id & \id   \\ 
\one & \sgn
\end{smallmatrix}\right] = \left[\begin{smallmatrix} 
1 & n-1 \\ 
\id & \id   \\ 
\sgn & \sgn
\end{smallmatrix}\right],
\ \ 
\left[\begin{smallmatrix} 
n-1 & 1 \\ 
\id & \id   \\ 
\one & \sgn
\end{smallmatrix}\right] = \left[\begin{smallmatrix} 
n-1 & 1 \\ 
\id & \id   \\ 
\one & \one
\end{smallmatrix}\right],
\ \ 
\left[\begin{smallmatrix} 
2 & n-2 \\ 
\id & \id   \\ 
\one & \sgn
\end{smallmatrix}\right] \equiv \left[\begin{smallmatrix} 
2 & n-2 \\ 
\fpf & \id   \\ 
\one & \sgn
\end{smallmatrix}\right],
\ \ 
\left[\begin{smallmatrix} 
n-2 & 2 \\ 
\id & \id   \\ 
\one & \sgn
\end{smallmatrix}\right] \equiv 
%\left[\begin{smallmatrix} 
%2 & n-2 \\ 
%\id & \id   \\ 
%\sgn & \one
%\end{smallmatrix}\right]\equiv 
\left[\begin{smallmatrix} 
2 & n-2 \\ 
\fpf & \id   \\ 
\sgn & \one
\end{smallmatrix}\right].
\]
The given character formulas are consequences of the Pieri rules and \eqref{irs-eq}.
 \end{proof}

Let $ \TT^{ \left[\begin{smallmatrix} 
n & 0 \\ 
\fpf & \id   \\ 
\gamma_1 & \gamma_2
\end{smallmatrix}\right]} :=  \TT^{\left[\begin{smallmatrix} 
n  \\ 
\fpf    \\ 
\gamma_1 
\end{smallmatrix}\right]}$ when $n$ is even
and
$ \TT^{ \left[\begin{smallmatrix} 
0 & n \\ 
\fpf & \id   \\ 
\gamma_1 & \gamma_2
\end{smallmatrix}\right]} :=  \TT^{\left[\begin{smallmatrix} 
n  \\ 
\id    \\ 
\gamma_2
\end{smallmatrix}\right]}$.
 When $\cX$ is a set of model triples,
 let $\overline\cX := \{ \overline \TT : \TT \in \cX\}$.
Then $\cX \approx \overline\cX$ where $\approx$ is
 the equivalence relation from Section~\ref{model-equiv-sect}.

\begin{theorem} \label{a-thm}
%Assume $n\geq 1$.
The following set is a  perfect model for $S_n$:  
\[ 
\cP^{\A}_{n-1} := \left\{ \TT^\Theta : \Theta = \left[\begin{smallmatrix} 
2k & n-2k \\ 
\fpf & \id   \\ 
\one & \sgn
\end{smallmatrix}\right] \text{ for }0\leq k \leq \lfloor\tfrac{n}{2}\rfloor \right\}.
%\quand
%\left\{ \TT^\Theta : \Theta = \left[\begin{smallmatrix} 
%2k & n-2k \\ 
%\fpf & \id   \\ 
%\sgn & \one
%\end{smallmatrix}\right] \text{ for }0\leq k \leq \lfloor\tfrac{n}{2}\rfloor \right\}.
\]
If $n\notin \{2, 4\}$ then
each perfect model for $S_n$ is strongly equivalent to $\cP^{\A}_{n-1} $ or $ \overline {\cP^{\A}_{n-1}}$.
%so there is only one equivalence class of perfect models.
\end{theorem}

\begin{proof}
The claim that $\cP^{\A}_{n-1} $ is a perfect model is well-known \cite{IRS}, or can be seen as
an immediate consequence of Proposition~\ref{one-col-a-prop}.
The difficult part of the theorem is showing that every perfect model is strongly equivalent
to  $\cP^{\A}_{n-1} $ or $ \overline {\cP^{\A}_{n-1}}$ when $n\notin \{2,4\}$.

Suppose $\cM$ is a set of model indices $\Theta \in \MTI(S_n)$
such that $\{\TT^\Theta : \Theta \in \cM\}$ is a perfect model for $S_n$.
Every perfect model for $S_n$ arises in this way. After replacing $\cM$ by 
a strongly equivalent set of indices, and updating our notation for model indices
to allow zeros in the first row,
we may assume
by Proposition~\ref{one-col-a-prop} that every $\Theta \in \cM$ has the form
\[
\Theta^{H}_{k} = \left[\begin{smallmatrix} 
k & n-k \\ 
\id & \id   \\ 
\one & \sgn
\end{smallmatrix}\right],
\ \ 
\Theta^\one_l=\left[\begin{smallmatrix} 
l & n-l \\ 
\id & \id   \\ 
\one & \one
\end{smallmatrix}\right],\ \ 
\Theta^{\sgn}_l=  \left[\begin{smallmatrix} 
l & n-l \\ 
\id & \id   \\ 
\sgn & \sgn
\end{smallmatrix}\right],
\ \ 
\Theta^{\OR}_m = \left[\begin{smallmatrix} 
n-m & m \\ 
\fpf & \id   \\ 
\one & \sgn
\end{smallmatrix}\right],
\text{ or } 
\Theta^{\OC}_m =\left[\begin{smallmatrix} 
n-m & m \\ 
\fpf & \id   \\ 
\sgn & \one
\end{smallmatrix}\right]\]
where $2<k<n-2$, $0<l<n$, and $m \in \Delta := \{ 0 \leq j \leq n : j \equiv n \modu 2)\}$.\footnote{
We can restrict $0<l<n$ since $\Theta^\one_0 \equiv  \Theta^\one_n \equiv \Theta^{\OC}_{n}$
and $\Theta^{\sgn}_0 \equiv  \Theta^{\sgn}_n \equiv \Theta^{\OR}_{n}$.}  
We now argue that the only possibility for $\cM$ is $\cX^{\OR} := \{ \Theta^{\OR}_m : m \in \Delta\}$ or $\cX^{\OC} := \{ \Theta^{\OC}_m : m \in \Delta\}$.
For small values of $n$ this claim can be checked by 
a short computer calculation
using Proposition~\ref{one-col-a-prop}. In the following argument we assume $n>10$.

If $\mu$ and $\lambda$ are partitions with $\D_\mu\subseteq\D_\lambda$ then we write $\mu \subseteq\lambda$.
To simplify our notation, we say that ``$\lambda$ is a constituent of $\Theta$'' as an abbreviation for  ``$\chi^\lambda$ is a constituent of $\chi^\Theta$.'' With this convention, every $\lambda \vdash n$ is a constituent of exactly one $\Theta\in \cM$,
and
the constituents of $\Theta^{\OR}_m $ and $\Theta^{\OC}_m $ are the partitions whose diagrams have $m$ odd rows or $m$ odd columns, respectively.
The formulas in Proposition~\ref{one-col-a-prop} tell us that no constituents $\lambda$ of $\Theta^H_k$, $\Theta^\one_l$, or $\Theta^{\sgn}_l$
have $(3,2,1) \subseteq \lambda$. Thus, $\cM$ must share at least one element  with $\cX^{\OR}$ or $\cX^{\OC}$. We argue below
that in fact, exactly one of the intersections $\cM \cap \cX^{\OR}$ or $\cM \cap \cX^{\OC}$ is nonempty:
\begin{itemize}
\item If $n=4a+1 \equiv 1\modu 4)$ then the partition $(3,2,1)\subseteq \lambda := (2a,2a,1)\vdash n$ has one odd row and one odd column,
so either $\Theta^{\OR}_1  \in \cM$ or $\Theta^{\OC}_1 \in \cM$.
By considering the partitions of the form $(p,q)\vdash n$ and their transposes,
one finds that there are partitions of $n$ with $1$ odd row and $m$ odd columns, or with $1$ odd column and $m$ odd rows, for every $m \in \Delta$.
Thus if  $\Theta^{\OR}_1  \in \cM$ then $\cM \cap \cX^{\OC} = \varnothing$ and 
if  $\Theta^{\OC}_1  \in \cM$ then $\cM \cap \cX^{\OR} = \varnothing$.

\item If $n=4a+3 \equiv 3 \modu 4)$ then   $(3,2,1)\subseteq \lambda := (2a,2a,3)\vdash n$ has $1$ odd row and $3$ odd columns, so either $\Theta^{\OR}_1  \in \cM$ or $\Theta^{\OC}_3 \in \cM$.
If $\Theta^{\OR}_1  \in \cM$ then it follows as in the previous case that $\cM \cap \cX^{\OC} = \varnothing$.
Assume $\Theta^{\OC}_3 \in \cM$. Since $(2a,2a-1,3,1)\vdash n$ has $3$ odd rows and $3$ odd columns,
we must have $\Theta^{\OR}_3 \notin \cM$.
But $ \lambda^\top = (3,3,3,2^{2a-3})$ has $3$ odd rows and $1$ odd column,
so  $\Theta^{\OC}_1 \in \cM$, which implies that $\cM \cap \cX^{\OR} = \varnothing$ as in the previous case.

\item If $n = 2a+2$ is even then  $(3,2,1)\subseteq \lambda := (a+1,a,1)\vdash n$
has $2$ odd rows and $2$ odd columns, so either  $\Theta^{\OR}_2  \in \cM$ or $\Theta^{\OC}_2 \in \cM$.
By considering the partitions of the form $(p,q,r)\vdash n$ and their transposes,
one finds that there are partitions of $n$ with $2$ odd rows and $m$ odd columns, or with $2$ odd columns and $m$ odd rows, for every $m \in \Delta\setminus\{n\}$.
Suppose  $\Theta^{\OR}_2  \in \cM$. Then $\cM$ contains no elements of $\cX^{\OC}$ except possibly $\Theta^{\OC}_n$,
whose unique constituent $(1^n)$ has zero odd rows.
The partition $(3,2,1)\subseteq \mu := (n-4,2,2)\vdash n$ has zero odd rows and $n-4$ odd columns.
Since $\Theta^{\OC}_{n-4}\notin \cM$ we must have $\Theta^{\OR}_0  \in \cM$, so $\cM \cap \cX^{\OC} =\varnothing$.
A symmetric argument shows that if $\Theta^{\OR}_2  \in \cM$ then $\cM \cap \cX^{\OR}$ is empty.
\end{itemize}
This completes the proof of the claim.

The claim shows that we are in one of two cases.
The first case is that $\cM\cap \cX^{\OR}$ is nonempty and $\cM \cap \cX^{\OC} =\varnothing$.
The partitions of the form $(3+a,2,1,1^{n-a-6})$  for $0\leq a \leq n-6$ and $ (n-4,2,2)$
all contain $(3,2,1)$ so must appear as constituents of elements of $\cM\cap \cX^{\OR}$.
The number of odd rows in these partitions range over all $m\equiv n \modu 2)$ with $0\leq m\leq n-4$,
so $\cX^{\OR} \setminus \{ \Theta^{\OR}_{n-2}, \Theta^{\OR}_n\}$ must be a subset of $\cM$.
The only partitions of $n$ not appearing as constituents of this subset 
are $ (1^n)$, $ (2,1^{n-1})$, and $(3,1^{n-3})$.

The remaining models indices that could be in $\cM$  are 
$\Theta^{\OR}_{n-2}$, $\Theta^{\OR}_n$,
$ \Theta^{H}_k$ for $2<k<n-2$,
 $\Theta^{\one}_l$ for $0<l<n$, or $\Theta^{\sgn}_l$ for $0<l<n$.
Among these, the only ones containing $(3,1^{n-3})$ as a constituent are
$\Theta^{\OR}_{n-2}$ and 
$ \Theta^{H}_3$. Since the latter index shares the constituent $(4,1^{n-4})$
with $\Theta^{\OR}_{n-4} \in \cM$,
we must have $\Theta^{\OR}_{n-2} \in \cM$.
But now the only partition of $n$ not accounted for as a constituent of some $\Theta \in \cM$ is $(1^n)$.
The only index that could be in $\cM$ that has $(1^n)$ as its unique constituent is  $\Theta^{\OR}_{n}$,
so we conclude that $\cM \supseteq \{\Theta^{\OR}_m : m \in \Delta\} $ whence $\cM= \{\Theta^{\OR}_m : m \in \Delta\} = \cP^{\A}_{n-1} $.

The other case that could arise is to have $\cM\cap \cX^{\OC}\neq \varnothing$ and $\cM \cap \cX^{\OR} =\varnothing$.
But then $\overline{\cM}$ belongs to the case just considered so must be equal to $\cP^{\A}_{n-1} $,
which means   $\cM = \overline{\cP^{\A}_{n-1} }$.
\end{proof}

\begin{example}\label{a-extra-ex}
When $n=2$ there is one other perfect model consisting of just $\{ (\varnothing, \{1\}, \one)\}$.
When $n=4$ there is a single additional equivalence of perfect models for $S_n$, which contains
\[
 \left\{ \TT^\Theta : \Theta = \left[\begin{smallmatrix} 
1 & 3 \\ 
\id & \id   \\ 
\one & \sgn
\end{smallmatrix}\right] \text{ or }
 \left[\begin{smallmatrix} 
2 & 2 \\ 
\id & \id   \\ 
\one & \one
\end{smallmatrix}\right]
 \right\}
\approx 
 \left\{ \TT^\Theta : \Theta = 
  \left[\begin{smallmatrix} 
1 & 3 \\ 
\id & \id   \\ 
\one & \one
\end{smallmatrix}\right] \text{ or }
 \left[\begin{smallmatrix} 
2 & 2 \\ 
\id & \id   \\ 
\sgn & \sgn
\end{smallmatrix}\right]
\right\}.
\] 
These models belong to a family of type $\D$ constructions, which arise here because $S_4 \cong \WD_3$.
\end{example}

\begin{remark}\label{all-a-models-rmk}
Assume $n\geq 5$. Theorem~\ref{a-thm} tells us that 
if $\cM$ is a perfect model for $S_n$ then $\{\chi^\TT : \TT \in \cM\}$
contains exactly one of  $\one$ or $\sgn$.
We refer to $\cM$ as a \defn{$\one$-model} if $\one \in \{\chi^\TT : \TT \in \cM\}$
and as a \defn{$\sgn$-model} if $\sgn \in \{\chi^\TT : \TT \in \cM\}$. 
One can determine whether $\cM$ is a $\one$-model or a $\sgn$-model
by examining whether
\[
\sum_{\lambda \in \OddCols(n,n-2\lfloor\frac{n}{2}\rfloor)} \chi^\lambda \in \{\chi^\TT : \TT \in \cM\}
\quord
\sum_{\lambda \in \OddRows(n,n-2\lfloor\frac{n}{2}\rfloor)} \chi^\lambda \in \{\chi^\TT : \TT \in \cM\}.
\]
It will be useful to enumerate all the ways of specifying a perfect model for $S_n$,
since there is some redundancy in our notation.
To construct a $\sgn$-model for $S_n$, first
choose $\Phi_0$ to be one of
\be\label{m1-eq}
\left[\begin{smallmatrix} 
 n \\ 
 \id   \\ 
 \sgn
\end{smallmatrix}\right] 
\quord
\left[\begin{smallmatrix} 
 n \\ 
 \id^+   \\ 
 \sgn
\end{smallmatrix}\right] 
\ee
Then let $\Phi_1$ be one of
\be\label{m2-eq}
\left[\begin{smallmatrix} 
2 & n-2 \\ 
\star & \id   \\ 
\one & \sgn
\end{smallmatrix}\right] ,
\quad
\left[\begin{smallmatrix} 
2 & n-2 \\ 
\star & \id^+  \\ 
\one & \sgn
\end{smallmatrix}\right] ,
\quad
\left[\begin{smallmatrix} 
 n-2 & 2 \\ 
 \id  & \star   \\ 
\sgn & \one 
\end{smallmatrix}\right] ,
\quord
\left[\begin{smallmatrix} 
 n-2 & 2 \\ 
 \id^+  & \star   \\ 
\sgn & \one 
\end{smallmatrix}\right]
\ee
where each $\star$ is any symbol  $\{\id,\id^+,\fpf,\fpf^+\}$.
For $2 \leq k \leq \lceil n/2 \rceil -2$ define $\Phi_k$ to be one of
\be\label{m3-eq}
 \left[\begin{smallmatrix} 
2k & n-2k \\ 
\fpf^\star & \id^\star   \\ 
\one & \sgn
\end{smallmatrix}\right] 
\quord
 \left[\begin{smallmatrix} 
 n-2k & 2k \\ 
 \id^\star & \fpf^\star   \\ 
\sgn & \one 
\end{smallmatrix}\right]
\ee
where $\fpf^\star \in \{ \fpf,\fpf^+\}$ and $\id^\star \in \{\id,\id^+\}$ are arbitrary.
When $n$ is even, let $ \Phi_{n/2-1}$ be one of
\be\label{m4-eq}
\left[\begin{smallmatrix} 
 n-2& 2 \\ 
\fpf & \star   \\ 
\one & \sgn
\end{smallmatrix}\right] ,
\quad
\left[\begin{smallmatrix} 
 n-2& 2 \\ 
\fpf^+ & \star   \\ 
\one & \sgn
\end{smallmatrix}\right] ,
\quad
\left[\begin{smallmatrix} 
2 & n-2 \\ 
\star & \fpf    \\ 
\sgn & \one 
\end{smallmatrix}\right] ,
\quord
\left[\begin{smallmatrix} 
2 &  n-2 \\ 
\star & \fpf^+    \\ 
\sgn & \one 
\end{smallmatrix}\right]
\ee
where each $\star$ is any symbol  $\{\id,\id^+,\fpf,\fpf^+\}$,
and let 
$\Phi_{ n/2 }$ be one of 
\be\label{m5-eq}
\left[\begin{smallmatrix} 
 n \\ 
\fpf    \\ 
\one 
\end{smallmatrix}\right] 
\quord
\left[\begin{smallmatrix} 
 n\\ 
\fpf^+   \\ 
\one 
\end{smallmatrix}\right]. 
\ee
Finally, when $n$ is odd 
choose $\Phi_{(n-1)/2 }$ to be
 one of
\be\label{m6-eq}
\left[\begin{smallmatrix} 
 n-1& 1 \\ 
\fpf^\star & \id^\star   \\ 
\one & \one
\end{smallmatrix}\right] ,
\quad
\left[\begin{smallmatrix} 
 n-1& 1 \\ 
\fpf^\star & \id^\star   \\ 
\one & \sgn
\end{smallmatrix}\right] ,
\quad
\left[\begin{smallmatrix} 
1 &  n-1 \\ 
\id^\star & \fpf^\star    \\ 
\one & \one 
\end{smallmatrix}\right] ,
\quord
\left[\begin{smallmatrix} 
1 &  n-1 \\ 
\id^\star & \fpf^\star    \\ 
\sgn & \one 
\end{smallmatrix}\right]
\ee
where $\fpf^\star \in \{ \fpf,\fpf^+\}$ and $\id^\star \in \{\id,\id^+\}$ are arbitrary.
Then $\cM = \{ \TT^{\Phi_k}: k=0,1,\dots\lfloor n/2\rfloor\}$
is a $\sgn$-model for $S_n$ and every $\sgn$-model arises in this way.
The $\one$-models for $S_n$ are constructed in the same way after interchanging 
  ``$\one$'' and ``$\sgn$''
in \eqref{m1-eq}--\eqref{m6-eq}.
\end{remark}

\section{Model classification in type B}\label{b-sect}

In this section we fix an integer $n\geq 2$ and consider the Coxeter group $W = \W_n$ with generating set $S = \{s_0,s_1,\dots,s_{n-1}\}$. 
Recall that $\W_n$ is the subgroup of permutations $w \in S_\ZZ$
with $w(-i) = -w(i)$ for all $i \in \ZZ$ and $w(i) = i$ for all $i>n$. 
Our main result here is Theorem~\ref{b-thm}.

\subsection{Perfect conjugacy classes in type B}

The longest element in $\W_n$ is the central element $w_0 =\bar 1 \bar 2 \bar 3 \cdots \bar n$.
If $n=2$ then there is a single nontrivial Coxeter automorphism (interchanging $s_0$ and $s_1$) and otherwise
there are no such automorphisms.
Given integers $p,q> 0$ with $p+q=n$, 
let $\cK_{(p,q)}^{\W_n} $ be the set of $w\in \W_n$ with 
\[| \{ i \in [n] : w(i) = i\}| = p\quand |\{ i \in [n] : w(i) =-i\}| = q.\]
Let 
$\cK_{\id}^{\W_n} := \{ 1\}$ and $ \cK_{\id^+}^{\W_n}:= \{ w_0\}=\{w_0^+\}.$
When $n$ is even define $\cK^{\W_n}_{\fpf}$ to be the set of involutions $z=z^{-1} \in W_n$ with $|z(i)| \neq i$ for all $i \in [n]$.
The perfect conjugacy classes in $(\W_n)^+$ consist of $\cK^{\W_n}_{\id}$,  $\cK^{\W_n}_{\id^+}$, and $\cK^{\W_n}_{(p,q)}$
for all $p,q> 0$ with $p+q=n$,
along with $\cK^{\W_n}_\fpf$ when $n$ is even  \cite[Ex. 9.2]{RV}.
%If $n \in \{1,2\}$ but $\cK^{S_1}_{\id} = \cK^{S_1}_{\id^+}$ and  $\cK^{S_2}_{\id} = \cK^{S_2}_{\fpf^+}$ and $\cK^{S_2}_{\id^+} = \cK^{S_2}_{\fpf}$.
The unique minimal-length elements of $\cK_{(p,q)}^{\W_n}$ and $\cK^{\W_n}_{\fpf}$ (when $n$ is even) are  
\[\bar 1 \bar 2 \cdots \bar q (q+1)(q+2)\cdots n \quand s_1s_3s_5\cdots s_{n-1}.\]

\subsection{Model indices in type B}

%\subsection{Linear characters in type B}

The linear characters of $\W_n$ are given as follows.
Let $\one_{++} :=\one$ be the trivial character and let $\one_{--} := \sgn$ be the sign character.
Define $\one_{+-}$ to be the linear character of $\W_n$ mapping $s_0 \mapsto 1$ and $s_i \mapsto -1$ for $i>0$.
Define $\one_{-+} := \one_{+-}\sgn$ to be the linear character of $\W_n$ mapping $s_0 \mapsto -1$ and $s_i \mapsto 1$ for $i>0$.
If $n=1$ then $\one_{+-} = \one$ and $\one_{-+}  = \sgn$.
If $n\geq 2$ then these four linear characters are distinct.
If $n$ is even and $z = s_1s_3s_5\cdots s_{n-1} \in \cK^{\W_n}_{\fpf}$
then the centralizer subgroup $C_\fpf := \{  w\in \W_n : wz=zw\}$ 
does not contain $s_0$ so 
\[\Res^{\W_n}_{C_\fpf}(\one_{-+}) = \Res^{\W_n}_{C_\fpf}(\one)
\quand
\Res^{\W_n}_{C_\fpf}(\one_{+-}) = \Res^{\W_n}_{C_\fpf}(\sgn).
\]

%\subsection{Model indices in type B}

Let $\MTI(\W_n)$ denote the set of $3\times 2$ arrays, to be called \defn{model indices} for $\W_n$, of the form
\[
\Theta = \left[\begin{smallmatrix} \alpha_0 & \alpha_1 \\ 
\beta_0& \beta_1  \\ 
\gamma_0 & \gamma_1 
\end{smallmatrix}\right]
\]
where $\alpha_0,\alpha_1 \geq 0$  are integers with $\alpha_0+\alpha_1=n$;
$\beta_0$ is a symbol in  $\{\id, \id^+,\fpf\}$ or a pair of positive integers $(p,q)$ with $p+q=n$;
$\beta_1$ is a symbol in  $\{\id, \id^+, \fpf, \fpf^+\}$;
and 
$\gamma_0 \in \{\one, \sgn, \one_{+-},\one_{-+}\}$ 
and $\gamma_1 \in \{\one, \sgn\}$ are linear characters of $\W_{\alpha_0}$ and $S_{\alpha_1}$.
We further require that:
\begin{itemize}
\item if  $\beta_0 =\fpf$ then $\alpha_0 \in \{0,2,4,6,\dots\}$
and if $\beta_1 \in \{ \fpf, \fpf^+\}$ then $\alpha_1 \in \{4,6,8\dots\}$;
\item  if $\alpha_0 \leq 1$ or $\beta_0 = \fpf$ then  $\gamma_0 \in \{\one,\sgn\}$.
\end{itemize}
%We refer to $\Theta $ as a \defn{model index}.
Suppose $\Theta= \left[\begin{smallmatrix} \alpha_0 & \alpha_1 \\ 
\beta_0& \beta_1  \\ 
\gamma_0 & \gamma_1 
\end{smallmatrix}\right] \in \MTI(\W_n)$. Let $J_0 := \{ s_{j-1} : j \in [\alpha_0] \}$ and
$J_1 := \{ s_j :  \alpha_0 < j <   n \}.$
% Let $\varphi_0 $ be the identity map on $(\W_{\alpha_0})^+ = \langle J_0\rangle^+$ and
 Write $\varphi_1:  S_{\alpha_1 }^+ \to \langle J_1\rangle^+$  for the isomorphism
 sending $s_j \mapsto s_{\alpha_0 + j}$ for $j \in [\alpha_1-1]$.
% This extends to an isomorphism $S_{\alpha_1}^+ \xrightarrow{\sim} \langle J_1\rangle^+$. 
% via the formula $(w,\theta) \mapsto (\varphi_i(w), \varphi_i \theta\varphi_i^{-1})$.
Let $\cK_0 := \cK^{\W_{\alpha_0}}_{\beta_0}$ and let $\cK_1$ be the image of $\cK^{S_{\alpha_1}}_{\beta_1}$ under $\varphi_1$.
%When $\alpha_0,\alpha_1>0$
Now set 
%\be \TT^\Theta :=( J_0, \cK_0, \gamma_0) \otimes( J_1, \cK_1, \gamma_1). \ee
\[ \TT^\Theta :=
\begin{cases} ( J_0, \cK_0, \gamma_0)& \text{if }\alpha_0=n, \\
( J_1, \cK_1, \gamma_1)& \text{if }\alpha_1=n, \\
( J_0, \cK_0, \gamma_0) \otimes( J_1, \cK_1, \gamma_1) &\text{otherwise}.
\end{cases}
 \]
In this way $\Theta$ indexes a factorizable model triple for $\W_n$. Also define
\[\chi_\BC^\Theta := \chi^{\TT^\Theta}.\]
Note that if $\alpha_i=0$ then $\TT^\Theta$ does not depend on $\beta_i$ or $\gamma_i$.
Moreover, if $\alpha_1 =1$ then
$\TT^\Theta$ is unaffected by changing $\beta_1=\id$ to $\id^+$ (or vice versa) or $\gamma_1=\one$ to $\sgn$ (or vice versa).
%\item Changing $\beta_1= \id$ to $ \fpf^+$ (or vice versa) or $\beta_1 = \id^+$ to $ \fpf$ (or vice versa) when $\alpha_1=2$.
%\end{itemize}
By Theorems~\ref{tech1-thm} and \ref{tech2-thm},
 every multiplicity-free model triple for $\W_n$ with $n\geq 2$ arises as $\TT^\Theta$ for some $\Theta \in \MTI(\W_n)$.\footnote{
 This is true, 
 even with our restrictions on $\gamma_0$ when $\beta_0=\fpf$, because 
 we
 do not distinguish between model triples $(J,\cK,\sigma_1)$ and $(J,\cK,\sigma_2)$
with $\Res^{W_J}_{C_J(z)} (\sigma_1) = \Res^{W_J}_{C_J(z)} (\sigma_2)$.
 }

Given $\Theta = \left[\begin{smallmatrix} \alpha_0  &\alpha_1 \\ 
\beta_0  &\beta_1 \\ 
\gamma_0 &\gamma_1 
\end{smallmatrix}\right] \in \MTI(\W_n)$,
define $\Theta^\vee := \left[\begin{smallmatrix} \alpha_0  &\alpha_1  \\ 
\beta^\vee_0  &\beta^\vee_1  \\ 
\gamma_0 &\gamma_1 
\end{smallmatrix}\right]$ and $\overline\Theta := \left[\begin{smallmatrix} \alpha_0  &\alpha_1   \\ 
\beta_0  &\beta_1   \\ 
\bar\gamma_0 &\bar\gamma_1
\end{smallmatrix}\right]$ 
where 
\[ \beta^\vee_i:=\begin{cases} (q,p) &\text{if $i=0$ and $\beta_0  =(p,q)$} 
\\
\fpf &\text{if $i=0$ and $\beta_0=\fpf$, or if $\beta_i = \fpf^+$} \\
\fpf^+ &\text{if $i=1$ and $\beta_i = \fpf$} \\
\id &\text{if $\beta_i = \id^+$} \\
  \id^+ &\text{if $\beta_i = \id$} \end{cases}
  \quand
  \bar \gamma_i:=\sgn   \gamma_i.\]
%Define $\chi_\BC^\Theta := \chi^{\TT^\Theta}$.
We adapt the relations $\equiv$, $\sim$, and $\approx$ to (sets of) model indices in $\MTI(\W_n)$
just as we did for elements of $\MTI(S_n)$.
% 
% If $\Theta,\Psi \in \MTI(\W_n)$
%then we write $\Theta \equiv \Psi$ if 
%$\TT^\Theta \equiv \TT^\Psi$, 
%$\Theta\sim \Psi$ if $\TT^\Theta \sim \TT^\Psi$, and
% $\Theta \approx\Psi$ if $\TT^\Theta\approx \TT^\Psi$.
% In the second two cases we say that $\Theta$ and $\Psi$ are \defn{strongly equivalent} and \defn{equivalent}.
%
%\begin{lemma}
%Let $\Theta \in \MTI(\W_n)$. Then
It is straightforward to check that
$\TT^{\Theta^\vee} = (\TT^\Theta)^\vee$ and  $\TT^{\overline\Theta} = \overline{\TT^\Theta}$.
%so $\Theta \sim \Theta^\vee \approx \overline\Theta$.
 Likewise, if $ \Theta'$ is formed from $\Theta$ by changing any entries equal to $\id^+$  to $\id$,
 then $\Theta \equiv \Theta'$.
%Form $\Theta''$ from $\Theta'$ by replacing any entries equal to $\fpf^+$  to $\fpf$.
%Then $\Theta \equiv \Theta'$ and $\chi_\BC^\Theta=\chi_\BC^{\Theta'} = \chi_\BC^{\Theta''}$.
%\end{lemma}
%
% \begin{proof}
% straightforward
% \end{proof}

\subsection{Littlewood-Richardson coefficients in type B}\label{LR-bc-sect}

%\subsection{Irreducible characters in type B}\label{char-b-sect}

The irreducible characters of $\W_n$ are indexed by \defn{bipartitions} of $n$,
that is, by ordered pairs of partitions $(\lambda,\mu)$ 
with $|\lambda| + |\mu| = n$.
To indicate that   $(\lambda,\mu)$ is a bipartition of $n$ we write  $(\lambda,\mu)\vdash n$.
Let $\chi^{(\lambda,\mu)} $ denote the irreducible character of $\W_n$ indexed by $(\lambda,\mu) \vdash n$ following the 
construction given  in \cite[\S5.5]{GeckP}.
Then, as explained before \cite[Lem. 5.5.5]{GeckP},
\be
\ba \one =\one_{++} &= \chi^{((n),\emptyset)} \\
\one_{-+} &= \chi^{(\emptyset,(n))}
\ea
\qquand
\ba
\sgn=\one_{--} &= \chi^{(\emptyset, (1,1,\dots,1))} \\
\one_{+-} &= \chi^{((1,1,\dots,1),\emptyset)}.
\ea
\ee
By  \cite[Thm. 5.5.6]{GeckP} one also has 
 \be
 \chi^{(\lambda,\mu)} \sgn =  \chi^{(\mu^\top,\lambda^\top)}
,
\quad
 \chi^{(\lambda,\mu)} \one_{-+} =  \chi^{(\mu,\lambda)}
,\quand
 \chi^{(\lambda,\mu)} \one_{+-}=  \chi^{(\lambda^\top,\mu^\top)}
.\ee

 % so $\one_{+-} = \chi^{((1,1,\dots,1),\emptyset)}$
% and $\one_{-+} = \chi^{(\emptyset,(n))}$.

%\subsection{Littlewood-Richardson coefficients in type B}\label{LR-bc-sect}

Let $p ,q\in \NN$.
% Observe that  $(-p-1,p+1) = s_p\cdots s_2s_1s_0s_1s_2\cdots s_p \in S_\ZZ$.
%Define
%\[ J = \{ s_{i-1} : i \in [p]\}
%\quand
%K = \begin{cases}
%\varnothing &\text{if }q= 0 \\
%\{(-p-1,p+1),s_{p+1},s_{p+2},\dots,s_{p+q-1}\} &\text{if }q>0.
%\end{cases}
%\]
We identify $\W_p\times \W_q$   in the usual way (see \cite[\S5.5]{GeckP}) with the subgroup
of permutations in $\W_{p+q}$ fixing  $\{ \pm 1, \pm2 ,\dots,\pm p\}$
and $\{ \pm (p+1), \pm(p+2),\dots,\pm (p+q)\}$.
%
% $\langle J\sqcup K \rangle \subseteq \W_{p+q}$.
Write $u\times v$ for the image of $(u,v)  \in \W_p\times \W_q$ in $\W_{p+q}$.
Given  $f : \W_p \to \CC$ and $g : \W_q \to \CC$
 define $f\boxtimes g : \W_p\times \W_q \to \CC$ by the formula $u\times v \mapsto f(u)g(v)$.
When $f$ and $g$ are class functions, let
\be f\bullet_\BC g := \Ind_{\W_p\times \W_q}^{\W_{p+q}}(f\boxtimes g)
.\ee
This is an associative, commutative, and bilinear operation.
If $\Theta = \left[\begin{smallmatrix}  \alpha_0& \alpha_1 
% & \alpha_2 & \dots & \alpha_l 
 \\ 
\beta_0 & \beta_1  
%& \beta_2 & \dots & \beta_l
 \\ 
\gamma_0 & \gamma_1 
%& \gamma_2& \dots &\gamma_l
\end{smallmatrix}\right]\in \MTI(\W_n)$ then 
\be\label{theta-bullet-bc-eq}
\chi_\BC^\Theta = \chi_\BC^{ \left[\begin{smallmatrix} \alpha_0  \\ 
\beta_0  \\ 
\gamma_0  
\end{smallmatrix}\right]}
\bullet_\BC 
\Ind_{S_{\alpha_1}}^{\W_{\alpha_1}} \Big(\chi_\A^{ \left[\begin{smallmatrix} \alpha_1  \\ 
\beta_1   \\ 
\gamma_1 
\end{smallmatrix}\right]}\Bigr)
\quad\text{where we define $\chi_\BC^{ \left[\begin{smallmatrix} \alpha_0  \\ 
\beta_0  \\ 
\gamma_0  
\end{smallmatrix}\right]} := \chi_\BC^{ \left[\begin{smallmatrix} \alpha_0 & 0  \\ 
\beta_0  & \id \\ 
\gamma_0  & \one
\end{smallmatrix}\right]}
$}\ee
and where we interpret the characters on the right as the trivial characters of 
$\W_0:=\{1\}$ or $S_0:=\{1\}$
 when $\alpha_0=0$ or $\alpha_1=0$.
 
 %
%Let $\cR^\BC_n$ denote the $\CC$-vector space of complex-valued class functions on $\W_n$.
%The operation  $\bullet_\BC$  makes $\cR^\BC := \bigoplus_{n \in \NN} \cR^\BC_n$ into a graded algebra.
%If $\Sym(x,y)$ is the algebra of bounded degree symmetric power series in $\CC[[x_1,y_1,x_2,y_2,\dots]]$
%% where $\{x_i\}$ and $\{y_i\}$ are two countable sets of commuting variables.
%the the linear map sending 
%$\chi^{(\lambda,\mu)} \mapsto s_\lambda(x) s_\mu(y)$ is an isomorphism of graded algebras
%$\cR^\BC \xrightarrow{\sim} \Sym(x,y)$ \cite[Lem. 6.1.3]{GeckP}, so
%Moreover, this map sends 
%\[\chi^{(\lambda,\mu)}\bullet_\BC \chi^\gamma \mapsto s_\lambda(x) s_\mu(y) \Delta(s_\nu)
%\quad\text{where  $\Delta(s_\nu) := \sum_{\alpha,\beta} c_{\alpha\beta}^\gamma s_\alpha(x) s_\beta(y)$.}\]
If $(\lambda_1,\lambda_2)$ and $(\mu_1,\mu_2)$ are bipartitions then by \cite[Lem. 6.1.3]{GeckP} we have
 \be\label{LR-bc-eq}
\chi^{(\lambda_1,\lambda_2)}\bullet_\BC
\chi^{(\mu_1,\mu_2)}
=
\sum_{\nu_1,\nu_2} c_{\lambda_1\mu_1}^{\nu_1} c_{\lambda_2\mu_2}^{\nu_2} \chi^{(\nu_1,\nu_2)}
\ee
where  $c_{\lambda\mu}^\nu \in \NN$ are the Littlewood-Richard coefficients from Section~\ref{LR-sect1}.
In addition, one has
\be\label{LR-ba-eq}
\Ind_{S_n}^{\W_n}(\chi^\nu) = \sum_{\lambda,\mu} c_{\lambda\mu}^\nu \chi^{(\lambda,\mu)}
\ee
by \cite[Lem. 6.1.4]{GeckP}.
Thus, the Pieri rules for $S_n$ imply that 
\be\label{pieri-bc-eq}
 \Ind_{S_n}^{\W_n}(\one) = \sum_{p+q=n} \chi^{((p),(q))}
 \quand
\Ind_{S_n}^{\W_n}(\sgn) = \sum_{p+q=n} \chi^{((1^p),(1^q))}.\ee

Let $\EvenBiRows(n)$ be the set of bipartitions $(\lambda,\mu)\vdash n$
where $\lambda$ and $\mu$ both have all even parts.
Define $\EvenBiCols(n) = \{ (\lambda^\top, \mu^\top) : (\lambda,\mu) \in \EvenBiRows(n)\}$. 
%These sets are empty if $n$ is odd.
If $n$ is even then
% gives a type B analogue of \eqref{irs-eq} when $n$ is even:
\be\label{bbb-prop1}
\ds\chi_\BC^{ \left[\begin{smallmatrix} n  \\ 
\fpf  \\ 
\one  
\end{smallmatrix}\right]}
= \sum_{(\lambda,\mu) \in \EvenBiRows(n)} \chi^{(\lambda,\mu)}
\quand
\ds\chi_\BC^{ \left[\begin{smallmatrix} n   \\ 
\fpf  \\ 
\sgn  
\end{smallmatrix}\right]}
= \sum_{(\lambda,\mu) \in \EvenBiCols(n)} \chi^{(\lambda,\mu)}
\ee
by \cite[Prop. 1]{Baddeley}. We note one other character formula for use in the next section:

%Suppose $p$ and $q$ are positive integers with $p+q=n$.
%Define $\Lambda(p,q)$ to be the set of partitions of the form
%$ \lambda = (\max\{p,q\}+r, \min\{p,q\}-r)$ for $0 \leq r \leq \min\{p,q\}$.

\begin{proposition}\label{bbb-prop2}
Suppose $p,q>0$ are integers with $p+q=n$.
Define $\Lambda=\Lambda(p,q)$ to be the set of partitions of the form
$ \lambda = (\max\{p,q\}+r, \min\{p,q\}-r)$ for $0 \leq r \leq \min\{p,q\}$.
Then %If $p$ and $q$ are positive integers with $p+q=n$ is even then
\[
\ds\chi_\BC^{ \left[\begin{smallmatrix} n   \\ 
(p,q)  \\ 
\one  
\end{smallmatrix}\right]}
%=\chi_\BC^{ \left[\begin{smallmatrix} n   \\ 
%(p,q)  \\ 
%\one_{++}  
%\end{smallmatrix}\right]}
= \sum_{\lambda \in \Lambda} \chi^{(\lambda,\emptyset)}
,\quad
\ds\chi_\BC^{ \left[\begin{smallmatrix} n   \\ 
(p,q)  \\ 
\one_{-+}  
\end{smallmatrix}\right]}
= \sum_{\lambda \in \Lambda} \chi^{(\emptyset,\lambda)}
,
\quad
\ds\chi_\BC^{ \left[\begin{smallmatrix} n   \\ 
(p,q)  \\ 
\sgn  
\end{smallmatrix}\right]}
%=\chi_\BC^{ \left[\begin{smallmatrix} n   \\ 
%(p,q)  \\ 
%\one_{--} 
%\end{smallmatrix}\right]}
= \sum_{\lambda \in \Lambda} \chi^{(\emptyset,\lambda^\top)}
,
\quad
\chi_\BC^{ \left[\begin{smallmatrix} n   \\ 
(p,q)  \\ 
\one_{-+}  
\end{smallmatrix}\right]}
= \sum_{\lambda \in \Lambda} \chi^{(\lambda^\top,\emptyset)}
.\]
\end{proposition}

\begin{proof}
The centralizer of  
$\bar 1 \bar 2 \cdots \bar q (q+1)(q+2)\cdots n\in\cK_{(p,q)}^{\W_n}$ is
 $\W_q\times \W_p$ and the restriction of any linear character $\sigma$ of $\W_n$
 to this subgroup is $\sigma \boxtimes \sigma$. The result then follows from \eqref{LR-bc-eq}.
\end{proof}

\subsection{Model projections in type B}

 We define two maps $\piL,\piR: \MTI(\W_n) \to \MTI(S_n)\sqcup\{0\}$.
Let $\Theta = \left[\begin{smallmatrix} \alpha_0 & \alpha_1 \\ 
\beta_0& \beta_1\\ 
\gamma_0 & \gamma_1
\end{smallmatrix}\right] \in \MTI(\W_n)$.
Let $(\lambda,\mu) \in \{((n),\emptyset),(\emptyset,(n)),((1^n),\emptyset),(\emptyset,(1^n))\}$
be such that $\gamma_0 = \chi^{(\lambda,\mu)}$.
If $\alpha_0=0$ then set 
\[
 \piL(\Theta) = \piR(\Theta) :=  \left[\begin{smallmatrix} \alpha_1  \\ 
 \beta_1    \\ 
 \gamma_1  
\end{smallmatrix}\right].
\]
Assume $\alpha_0,\alpha_1>0$.
If $\alpha_0 \in \{2,4,6,\dots\}$ and $\beta_0=\fpf$, then 
$\gamma_0 \in \{ \one,\sgn\}$ and we set 
 \[\piL(\Theta) = \piR(\Theta) := \Theta= \left[\begin{smallmatrix} \alpha_0 & \alpha_1 \\ 
\beta_0& \beta_1\\ 
\gamma_0 & \gamma_1
\end{smallmatrix}\right]  \in \MTI(S_n).\]
If $\beta_0 \in \{\id,\id^+\}$ 
 then define 
\[
\piL(\Theta) :=
\begin{cases}
 \left[\begin{smallmatrix}\alpha_0& \alpha_1  \\ 
\beta_0& \beta_1\\ 
\chi^\lambda &  \gamma_1   
\end{smallmatrix}\right]
&\text{if }\mu=\emptyset \\
0&\text{if }\mu \neq \emptyset
\end{cases}
\quand
\piR(\Theta) :=
\begin{cases}
 \left[\begin{smallmatrix}\alpha_0& \alpha_1   \\ 
\beta_0& \beta_1 \\ 
\chi^\mu &  \gamma_1 
\end{smallmatrix}\right]
&\text{if }\lambda=\emptyset \\
0&\text{if }\lambda \neq \emptyset.
\end{cases}
\]
If $\beta_0 = (p,q)$ for positive integers $p$ and $q$ with $p+q=\alpha_0$ then define
\[
\piL(\Theta) :=
\begin{cases}
 \left[\begin{smallmatrix}p & q& \alpha_1  \\ 
\id &\id & \beta_1   \\ 
\chi^\lambda & \chi^\lambda &  \gamma_1 
\end{smallmatrix}\right]
&\text{if }\mu=\emptyset \\
0&\text{if }\mu \neq \emptyset
\end{cases}
\quand
\piR(\Theta) :=
\begin{cases}
 \left[\begin{smallmatrix}p & q& \alpha_1   \\ 
\id & \id & \beta_1  \\ 
\chi^\mu & \chi^\mu &  \gamma_1   
\end{smallmatrix}\right]
&\text{if }\lambda=\emptyset \\
0&\text{if }\lambda \neq \emptyset.
\end{cases}
\]
When $\alpha_0>0$ but $\alpha_1=0$,
we form $\piL(\Theta)$ and $\piR(\Theta)$ by
applying the same formulas as above, and then deleting the last column
if the result is nonzero.

Let $\cR^\A_n$ and $\cR^\BC_n$ denote the $\CC$-vector spaces of complex-valued class functions on $S_n$ and $\W_n$,
respectively, 
and set $\cR^\A := \bigoplus_{n \in \NN} \cR^\A_n$ and $\cR^\BC := \bigoplus_{n \in \NN} \cR^\BC_n$.
We use the same symbols $\piL$ and $\piR$ to denote the linear maps $\cR^\BC \to \cR^\A$ 
with 
\[\piL(\chi^{(\lambda,\mu)}) := \begin{cases} \chi^\lambda &\text{if }\mu=\emptyset \\
0&\text{if }\mu\neq\emptyset
\end{cases}
\quand
\piR(\chi^{(\lambda,\mu)}) := \begin{cases} \chi^\mu &\text{if }\lambda=\emptyset \\
0&\text{if }\lambda\neq\emptyset
\end{cases}
\]
for all bipartitions $(\lambda,\mu)$. Finally, set $\chi_\A^0 := 0 \in \cR^\A$.

\begin{lemma}\label{bca-pi-lem}
If $\Theta  \in \MTI(\W_n)$ then $\piL(\chi_\BC^\Theta) = \chi_\A^{\piL(\Theta)}$
and
$\piR(\chi_\BC^\Theta) = \chi_\A^{\piR(\Theta)}$.
\end{lemma}

\begin{proof}
One has $c^{\emptyset}_{\lambda\mu} =0$ if $\lambda$ or $\mu$ is nonempty and $c^\emptyset_{\emptyset\emptyset}=1$.
Likewise, one has $c^\nu_{\lambda\emptyset}=0$ if $\lambda\neq \nu$ and $c^{\nu}_{\nu\emptyset}= 1$.
It follows from these observations via \eqref{LR-bc-eq} and \eqref{LR-ba-eq} that $\piL(\chi \bullet_\BC \Ind_{S_n}^{\W_n}(\psi)) = \piL(\chi) \bullet_\A \psi$
for all $\chi \in \cR^\BC$ and $\psi \in \cR^\A_n$.
In view of this identity and \eqref{theta-bullet-eq} and \eqref{theta-bullet-bc-eq},
we see that to show $\piL(\chi_\BC^\Theta) = \chi_\A^{\piL(\Theta)}$
for all $\Theta  \in \MTI(\W_n)$
it suffices to prove this identity when $\Theta = \left[\begin{smallmatrix} \alpha_0  \\ 
\beta_0 \\ 
\gamma_0 
\end{smallmatrix}\right] $ as in \eqref{bbb-prop1} and Proposition~\ref{bbb-prop2}.
But this follows immediately from the definition of $\piL$ and Proposition~\ref{one-col-a-prop}.
The proof that $\piR(\chi_\BC^\Theta) = \chi_\A^{\piR(\Theta)}$ is similar.
\end{proof}

\subsection{Perfect models in type B}

Fix a model index $\Theta =\left[\begin{smallmatrix} 
\alpha_0 & \alpha_1 \\ 
\beta_0 & \beta_1   \\ 
\gamma_0 & \gamma_1
\end{smallmatrix}\right] \in \MTI(\W_n)$

\begin{lemma}\label{b-lem}
The character $\chi_\BC^\Theta$ is not multiplicity-free if any of the following conditions hold:
\ben
\item[(a)] $\alpha_1 \in \{4,6,8,\dots\}$ and $\beta_1\in \{\fpf,\fpf^+\}$.
\item[(b)] $\alpha_0 \in \{2,4,6,\dots\}$, $\beta_0 =\fpf$, $\alpha_1\geq 2$, and $\gamma_0=\gamma_1 \in \{\one,\sgn\}$.
\item[(c)] $\beta_0 = (p,q)$ for some $p,q>0$ with $p+q=\alpha_0$ and $\alpha_1>0$. 
\een
\end{lemma}

\begin{proof}
Suppose (a) holds and let $\Psi := \left[\begin{smallmatrix} \alpha_1  \\ 
\beta_1   \\ 
\gamma_1 
\end{smallmatrix}\right]$ and $m := \alpha_1$. 
To prove that $\chi_\BC^\Theta$ is not multiplicity-free,
it suffices by \eqref{theta-bullet-bc-eq} to show that  $\Ind_{S_{m}}^{\W_{m}} (\chi_\A^\Psi)$ is not multiplicity-free. 
As $\chi_\A^{\overline\Psi} = \chi_\A^{\Psi}\sgn$, we may assume  $\gamma_1=\one$.
Proposition~\ref{one-col-a-prop}
and \eqref{LR-ba-eq} imply that $\Ind_{S_{m}}^{\W_{m}} (\chi_\A^\Psi)=\sum_{\nu \in \EvenRows(m)}\sum_{\lambda,\mu} c_{\lambda\mu}^\nu \chi^{(\lambda,\mu)}$, which is not multiplicity-free as $c^{(m)}_{(m-2)(2)} = c^{(m-2,2)}_{(m-2)(2)} = 1$.
 
 Suppose (b) holds. We may assume $\beta_1 \in \{\id,\id^+\}$ given  the previous paragraph.
% since if $\alpha_1 \in \{4,6,8\}$
% and $\beta_1 \in \{\fpf,\fpf^+\}$ then $\chi_\BC^\Theta$ is not multiplicity-free by the previous paragraph,
% and if $\alpha_1=2$ then changing $\beta_1$ from $\fpf/\fpf^+$ to $\id^+/\id$ has no effect on $\TT^\Theta$.
% If $\alpha_0 \in \{4,6,8\}$ then 
% Lemmas~\ref{a-lem} and \ref{bca-pi-lem}
% imply that one of $\piL(\chi_\BC^\Theta)$ or  $\piR(\chi_\BC^\Theta)$ is not multiplicity-free (according to whether $\gamma_0=\one$ or $\gamma_0=\sgn$), so $\chi_\BC^\Theta$ must not be multiplicity-free. 
% If $\alpha_0=2$ then 
Let $m = \alpha_0$.
It is straightforward from Section~\ref{LR-bc-sect}
 to check that $\chi_\BC^\Theta$ 
 contains either $\chi^{((n-m),(m))}$ (when $\gamma_0=\gamma_1=\one$) or $\chi^{((1^{n-m}),(1^m))}$
 (when $\gamma_0=\gamma_1=\sgn$)
  as a
 constituent with multiplicity greater than one.
  
 Finally, if (c) holds then one of $\piL(\chi_\BC^\Theta)$ or  $\piR(\chi_\BC^\Theta)$ is nonzero.
 But if $\piL(\chi_\BC^\Theta)=\chi_\A^{\piL(\Theta)}$ is nonzero then it 
cannot be multiplicity-free since $\piL(\Theta)$ has three columns.
The character  $\piR(\chi_\BC^\Theta)$  likewise cannot be nonzero and multiplicity-free.
Hence $\chi_\BC^\Theta$ must not be multiplicity-free.
\end{proof}

%\begin{lemma}\label{bc-fpf-lem}
% Suppose $\Theta =\left[\begin{smallmatrix} 
%\alpha_0 & \alpha_1 \\ 
%\beta_0 & \beta_1   \\ 
%\gamma_0 & \gamma_1
%\end{smallmatrix}\right] \in \MTI(\W_n)$
%has $\alpha_1 \in \{4,6,8,\dots\}$ and $\beta_1\in \{\fpf,\fpf^+\}$.
% Then $\chi_\BC^\Theta$ is not multiplicity-free.
%\end{lemma}
%
%\begin{proof}
%Let $\Psi := \left[\begin{smallmatrix} \alpha_1  \\ 
%\beta_1   \\ 
%\gamma_1 
%\end{smallmatrix}\right]$ and $n := \alpha_1$. 
%It suffices by \eqref{theta-bullet-bc-eq} to show that  $\Ind_{S_{n}}^{\W_{n}} (\chi_\A^\Psi)$ is not multiplicity-free. 
%As $\chi_\A^{\overline\Psi} = \chi_\A^{\Psi}\sgn$, we may assume  $\gamma_1=\one$.
%Proposition~\ref{one-col-a-prop}(b)
%and \eqref{LR-ba-eq} imply that $\ds\Ind_{S_{n}}^{\W_{n}} (\chi_\A^\Psi)=\sum_{\nu \in \EvenRows(n)}\sum_{\lambda,\mu} c_{\lambda\mu}^\nu \chi^{(\lambda,\mu)}$, which is not multiplicity-free as $c^{(n)}_{(n-2)(2)} = c^{(n-2,2)}_{(n-2)(2)} = 1$.
% \end{proof}

Let $\OddBiRows(n,q)$ be the set of bipartitions $(\lambda,\mu)\vdash n$ such that $\lambda \cup \mu$ has exactly $q$ odd parts.
Define $\OddBiCols(n,q) = \{ (\lambda^\top,\mu^\top) : (\lambda,\mu)\in\OddBiRows(n,q)\}$.

\begin{proposition} \label{bc-index-prop}
Suppose $\Theta \in \MTI(\W_n)$ % has $\ncols(\Theta)=2$
and $\chi^\Theta_\BC$ is
multiplicity-free.
Then
 $\Theta$ has one of the following forms:
\ben

\item[(a)] $\left[\begin{smallmatrix} 
k & n-k \\ 
\id/\id^+ & \id/\id^+   \\ 
\gamma_0 & \gamma_1
\end{smallmatrix}\right]$ for some $0 \leq k \leq n$,
$\gamma_0 \in \{ \one,\sgn,\one_{+-},\one_{-+}\}$, and $\gamma_1 \in \{\one,\sgn\}$.

\item[(b)] $\left[\begin{smallmatrix} 
2k & n-2k \\ 
\fpf & \id/\id^+   \\ 
\one & \sgn
\end{smallmatrix}\right]$ for some $0\leq k  \leq \lfloor n/2\rfloor$, in which case $\chi_\BC^\Theta =\sum_{(\lambda,\mu) \in \OddBiRows(n,n-2k)} \chi^{(\lambda,\mu)}$.

\item[(c)] $\left[\begin{smallmatrix} 
2k & n-2k \\ 
\fpf & \id/\id^+   \\ 
\sgn & \one
\end{smallmatrix}\right]$ for some $0\leq k \leq \lfloor n/2\rfloor $, in which case $\chi_\BC^\Theta =\sum_{(\lambda,\mu) \in \OddBiCols(n,n-2k)} \chi^{(\lambda,\mu)}$.

\item[(d)] $\left[\begin{smallmatrix} n & 0 \\ (p,q)  & \id/\id^+ \\ \gamma_0 & \one  \end{smallmatrix}\right] $ for some $p,q>0$ such that $p+q=n$
and $\gamma_0 \in \{ \one,\sgn,\one_{+-},\one_{-+}\}$.
%in which case  $\chi_\BC^\Theta = \sum_{(\lambda,\mu) \in \Lambda_\Delta(p,q;\sigma)} \chi^{(\lambda,\mu)}$.

\een

\end{proposition}

\begin{proof}
The given cases account for all model indices in $\MTI(\W_n)$ not excluded by Lemma~\ref{b-lem}.
The formulas in parts (b) and (c) follow by combining \eqref{theta-bullet-bc-eq}--\eqref{LR-ba-eq} with \eqref{bbb-prop1}.
\end{proof}

\begin{theorem}\label{b-thm} 
Assume $n\geq 2$.  
The following sets are inequivalent perfect models for $\W_n$:  
\[ 
\ba
\cP^{\BC}_n &:=\left\{ \TT^\Theta : \Theta = \left[\begin{smallmatrix} 
2k & n-2k \\ 
\fpf & \id   \\ 
\one & \sgn
\end{smallmatrix}\right] \text{ for }0\leq k \leq \lfloor\tfrac{n}{2}\rfloor \right\},
\\
\hat \cP^{\BC}_n &:=\left\{ \TT^\Theta : \Theta =  \left[\begin{smallmatrix} 
2 & n-2 \\ 
\id & \id   \\ 
\one & \sgn
\end{smallmatrix}\right],\ \left[\begin{smallmatrix} 
2 & n-2 \\ 
\id & \id   \\ 
\one_{-+} & \sgn
\end{smallmatrix}\right], \text{ or } \left[\begin{smallmatrix} 
2k & n-2k \\ 
\fpf & \id   \\ 
\one & \sgn
\end{smallmatrix}\right] \text{ for }k = 0,2,3,4,\dots \lfloor\tfrac{n}{2}\rfloor \right\}.
\ea
\]
If $n\neq 3$ then
each perfect model for $\W_n$ is strongly equivalent to one of $\cP^{\BC}_n$, $ \overline{\cP^{\BC}_n}$, $\hat \cP^{\BC}_n $, or $ \overline{\hat \cP^{\BC}_n}$.
%so there are two equivalence classes of perfect models.
\end{theorem}

\begin{proof}
It is clear from part (b) of Proposition~\ref{bc-index-prop} that $\cP^{\BC}_n$
 is a perfect model for $\W_n$.
It follows that $\hat\cP^{\BC}_n$ is also a perfect model since 
\[ \chi_\BC^{ \left[\begin{smallmatrix} 
2 & n-2 \\ 
\id & \id   \\ 
\one & \sgn
\end{smallmatrix}\right]} + \chi_\BC^{\left[\begin{smallmatrix} 
2 & n-2 \\ 
\id & \id   \\ 
\one_{-+} & \sgn
\end{smallmatrix}\right] } =
\Bigl(
\chi_\BC^{ \left[\begin{smallmatrix} 
2   \\ 
\id     \\ 
\one  
\end{smallmatrix}\right]} + \chi_\BC^{\left[\begin{smallmatrix} 
2   \\ 
\id     \\ 
\one_{-+}  
\end{smallmatrix}\right] } 
\Bigr)\bullet_\BC \Ind_{S_{n-2}}^{\W_{n-2}}\Bigl(\chi_\A^{\left[\begin{smallmatrix} 
  n-2 \\ 
  \id   \\ 
  \sgn
\end{smallmatrix}\right] }\Bigr)
=
 \chi^{\left[\begin{smallmatrix} 
2 & n-2 \\ 
\fpf & \id   \\ 
\one & \sgn
\end{smallmatrix}\right] }.
\]
 
We have checked the desired result using the computer algebra system \textsf{GAP} \cite{GAP} when $n\leq 4$, so assume $n\geq 5$.
Suppose $\cM$ is a set of model indices $\Theta \in \MTI(\W_n)$
such that $\{\TT^\Theta : \Theta \in \cM\}$ is a perfect model for $\W_n$.
Every perfect model for $\W_n$ arises in this way.
Define $\cM_{\mathsf{L}} := \{ \piL(\Theta)  : \Theta \in \cM\}\setminus\{0\}$
and 
$\cM_{\mathsf{R}} :=\{ \piR(\Theta)  : \Theta \in \cM\}\setminus\{0\}$.
Lemma~\ref{bca-pi-lem} implies that
$\{ \TT^\Theta : \Theta \in \cM_{\mathsf{L}}\}$ and $\{ \TT^\Theta : \Theta \in \cM_{\mathsf{R}}\}$
are perfect models for $S_n$.
After possibly replacing $\cM$ by $\overline\cM = \{ \overline \Theta : \Theta \in \cM\}$,
we may assume that $\{ \TT^\Theta : \Theta \in \cM_{\mathsf{L}}\}$ is a $\sgn$-model
 in the terminology of Remark~\ref{all-a-models-rmk}.

Define $ \Theta_k := \left[\begin{smallmatrix} 
2k & n-2k \\ 
\fpf & \id   \\ 
\one & \sgn
\end{smallmatrix}\right]$  for $0 \leq k \leq \lfloor n/2\rfloor$.
Remark~\ref{all-a-models-rmk} tells us that $\cM_{\mathsf{L}}$ must contain elements of  each of
the forms \eqref{m3-eq}, \eqref{m4-eq}, \eqref{m5-eq}, and \eqref{m6-eq}.
There are limited possibilities for  model indices $\Theta \in \MTI(\W_n)$ 
with $\chi_\BC^\Theta$ multiplicity-free
that can serve as the preimages for these elements under $\piL$.
It follows by inspecting Proposition~\ref{bc-index-prop}  
 that 
$\cM$ must contain a unique model index strongly equivalent to 
$\Theta_k$
 for at least each $2 \leq k \leq \lfloor n/2\rfloor$.
 Since $\piR(\Theta_k) =\piL(\Theta_k) = \Theta_k$,
it follows from Remark~\ref{all-a-models-rmk} that $\{ \TT^\Theta : \Theta \in \cM_{\mathsf{R}}\}$ is  
also  a $\sgn$-model.

By similar reasoning, for $\cM_{\mathsf{L}}$ to contain an index of the form \eqref{m1-eq},
 $\cM$ must contain a unique element strongly equivalent to
$
  \Theta_0
  =\left[\begin{smallmatrix} 
0 & n \\ 
\fpf & \id   \\ 
\one & \sgn
\end{smallmatrix}\right]
\sim
  \left[\begin{smallmatrix} 
0 & n \\ 
\id & \id   \\ 
\one & \sgn
\end{smallmatrix}\right]
$ or $
 \Theta'_0:=\left[\begin{smallmatrix} 
n & 0 \\ 
\id & \id   \\ 
\one_{+-} & \one
\end{smallmatrix}\right].
$
For $\cM_{\mathsf{R}}$ to contain an index of the form \eqref{m1-eq},
 $\cM$ must contain a unique element strongly equivalent to
$
  \Theta_0
%  =\left[\begin{smallmatrix} 
%0 & n \\ 
%\fpf & \id   \\ 
%\one  & \sgn
%\end{smallmatrix}\right]
%\quord
$
or
$
 \Theta''_0:=\left[\begin{smallmatrix} 
n & 0 \\ 
\id & \id   \\ 
\sgn & \one
\end{smallmatrix}\right].
$
Likewise, 
for $\cM_{\mathsf{L}}$ to contain an index of the form \eqref{m2-eq},
 $\cM$ must contain a unique element strongly equivalent to
$
  \Theta_1
%  =\left[\begin{smallmatrix} 
%2 & n-2 \\ 
%\fpf & \id   \\ 
%\one & \sgn
%\end{smallmatrix}\right],
$, $\Theta'_1:=\left[\begin{smallmatrix} 
2 & n-2 \\ 
\id & \id   \\ 
\one & \sgn
\end{smallmatrix}\right]
$,
or
$ 
\Psi'_1:= \left[\begin{smallmatrix} 
n-2 & 2 \\ 
\id & \id   \\ 
\one_{+-} & \one
\end{smallmatrix}\right].
$
For $\cM_{\mathsf{R}}$ to contain an index of the form \eqref{m2-eq},
 $\cM$ must contain a unique element strongly equivalent to
  $
  \Theta_1$,
%  =\left[\begin{smallmatrix} 
%2 & n-2 \\ 
%\fpf & \id   \\ 
%\one  & \sgn
%\end{smallmatrix}\right],
%\quad
$
 \Theta''_1:=\left[\begin{smallmatrix} 
2 & n-2 \\ 
\id & \id   \\ 
\one_{-+} & \sgn
\end{smallmatrix}\right]
$,
or
$ \Psi''_1:= \left[\begin{smallmatrix} 
n-2 & 2 \\ 
\id & \id   \\ 
\sgn & \one
\end{smallmatrix}\right].
$

Thus,  $\cM$  contains a subset strongly equivalent to
$\cM^0 \sqcup \cM^1 \sqcup \cM^2$
where
 $\cM^0$ is either $\{ \Theta_0\}$ or $\{\Theta_0',\Theta_0''\}$;
$\cM^1$ is either $\{ \Theta_1\}$, $\{\Theta_1',\Theta_1''\}$,
$\{\Theta_1',\Psi_1''\}$, $\{\Psi_1',\Theta_1''\}$, or $\{\Psi_1',\Psi_1''\}$; and
 $\cM^2 := \{ \Theta_k : 2 \leq k \leq \lfloor n/2\rfloor\}$.
In fact, we must have $\cM \sim \cM^0\sqcup \cM^1\sqcup \cM^2$
since it is impossible to add any further indices  
without violating our assumption that $\{ \TT^\Theta : \Theta \in \cM_{\mathsf{L}}\}$ and $\{ \TT^\Theta : \Theta \in \cM_{\mathsf{R}}\}$
are perfect models for $S_n$.

If $\cM^0 = \{\Theta_0\}$ and $\cM^1 \in\{ \{\Theta_1\},\{\Theta_1',\Theta''_1\}\}$
then $\{\TT^\Theta : \Theta \in \cM\}$ is strongly equivalent to $\cP^{\BC}_n$ or $\hat\cP^{\BC}_n$ as desired.
All of the other choices for
$\cM^0$ and $\cM^1$ are all impossible since 
they would lead to   $\sum_{\Theta \in \cM} \chi_\BC^\Theta(1)
=\sum_{\Theta \in \cM^0\sqcup \cM^1 \sqcup \cM^2} \chi_\BC^\Theta(1)
 < \sum_{k=0}^{\lfloor n/2\rfloor} \chi_\BC^{\Theta_k}(1) = \sum_{\chi \in \Irr(\W_n)} \chi(1)$,
 contradicting the fact that $\{\TT^\Theta : \Theta \in \cM\}$ is a model. We conclude that 
if $\cP$ is a perfect model for $\W_n$ when $n\neq 3$ then $\cP$ or $\overline\cP$ is strongly equivalent to  $\cP^{\BC}_n$ or $\hat \cP^{\BC}_n $.
\end{proof}

\begin{example}\label{b3-ex}
There are 2 more equivalence classes of perfect models for $\W_3$, represented by 
\[ \ba
&\left\{ \TT^\Theta : \Theta = \left[\begin{smallmatrix} 
1 & 2 \\ 
\id & \id   \\ 
\one & \one
\end{smallmatrix}\right],
\left[\begin{smallmatrix} 
2 & 1 \\ 
\id & \id   \\ 
\sgn & \one
\end{smallmatrix}\right],
\left[\begin{smallmatrix} 
3 & 0 \\ 
\id & \id   \\ 
\one_{+-} & \one
\end{smallmatrix}\right],\text{ or }
\left[\begin{smallmatrix} 
3 & 0 \\ 
\id & \id   \\ 
\one_{-+} & \one
\end{smallmatrix}\right]
\right\} \text{ and }
\\
&\left\{ \TT^\Theta : \Theta = \left[\begin{smallmatrix} 
1 & 2 \\ 
\id & \id   \\ 
\one_{-+} & \one
\end{smallmatrix}\right],
\left[\begin{smallmatrix} 
2 & 1 \\ 
\id & \id   \\ 
\one_{+-} & \one
\end{smallmatrix}\right],
\left[\begin{smallmatrix} 
3 & 0 \\ 
\id & \id   \\ 
\one & \one
\end{smallmatrix}\right],\text{ or }
\left[\begin{smallmatrix} 
3 & 0 \\ 
\id & \id   \\ 
\sgn & \one
\end{smallmatrix}\right]
\right\}.
\ea
\]
\end{example}

\section{Model classification in type D}\label{d-sect}

In this section we continue to fix an integer $n\geq 2$ and consider the Coxeter group $W =\WD_n$ with generating set $S = \{s_{-1},s_2,s_3,\dots,s_{n-1}\}$.
Recall that $\WD_n$ is the subgroup of permutations $w \in \W_n$
for which $ |\{ i \in [n] : w(i) < 0\}|$ is even. Our main result here is Theorem~\ref{d-thm}.

\subsection{Perfect conjugacy classes in type D}\label{cc-d-sect}

%Assume $W = \WD_n$ and $S = \{s_{-1},s_1,s_2,\dots,s_{n-1}\}$ where $n\geq 2$.
%Recall that $\WD_n$ is the subgroup of permutations $w \in \W_n$
%for which $ |\{ i \in [n] : w(i) < 0\}|$ is even.
Let $w \mapsto w^{\diamond}$ denote the Coxeter automorphism of $\WD_n$ interchanging $s_{-1}\leftrightarrow s_{1}$
while fixing the other simple generators. 
This is the restriction of $\Ad(s_0) \in \Aut(\W_n)$.
Write $w_0 $ for the longest element in $\WD_n$.
 If $n$ is even then $\diamond$ is an outer automorphism and  $w_0 =\bar 1 \bar 2 \bar 3 \cdots \bar n$ is central.
 If $n$ is odd then $w_0 = 1 \bar 2 \bar 3 \cdots \bar n$ and $\diamond = \Ad(w_0)$.
Thus $w_0^+ := (w_0,\Ad(w_0)) \in W^+$ is equal to $w_0$ if $n$ is even and to $(w_0,\diamond)$ when $n$ is odd.

If $n \neq 4$ then the Coxeter automorphisms of $\WD_n$ are $\{\id,\diamond\}$.
The Coxeter diagram of $\WD_4$ is 
\begin{center}
    \begin{tikzpicture}[xscale=1.2, yscale=1,>=latex,baseline=(z.base)]
    \node at (0,0) (z) {};
      \node at (-1,0) (T0) {$s_1$};
      \node at (0,0) (T1) {$s_2$};
      \node at (1,0) (T2) {$s_3$};
      \node at (0,-1) (T3) {$s_{-1}$};
      \draw[-]  (T0) -- (T1) -- (T2);
      \draw[-]  (T1) -- (T3);
     \end{tikzpicture}
     \end{center}
so there are two Coxeter automorphisms of $\WD_4$ of order three that fix $s_2$. We denote these by
\be
\circlearrowright : s_{-1} \mapsto s_{1} \mapsto s_3 \mapsto s_{-1}
\quand 
\circlearrowleft : s_{-1} \mapsto s_3 \mapsto s_{1} \mapsto s_{-1}.
\ee
There are three nontrivial Coxeter involutions of $\WD_4$, given by
$\diamond$, $\circlearrowright \diamond \circlearrowleft$, and $\circlearrowleft \diamond \circlearrowright$.
The latter two
interchange $s_{-1}\leftrightarrow s_3$ and $s_1 \leftrightarrow s_{3}$, respectively, while fixing the other simple generators.\footnote{
The automorphism  $\circlearrowright \diamond \circlearrowleft$ interchanges $s_{-1}$ and $s_3$ because $s_{-1}^{\circlearrowright \diamond \circlearrowleft} := ((s_{-1}^\circlearrowright)^\diamond))^\circlearrowleft = (s_{1}^\diamond)^\circlearrowleft = s_{-1}^{\circlearrowleft} = s_3$.}

Fix $p,q>  0$ with $p+q=n$. If $q$ is even then 
let $\cK_{(p,q)}^{\WD_n} := \cK_{(p,q)}^{\W_n}$ be the set of $w\in \WD_n$ with 
$| \{ i \in [n] : w(i) = i\}| = p$ and $ |\{ i \in [n] : w(i) =-i\}| = q.$
%Every such element is in $\WD_n$. 
If $q$ is odd then   define 
\[ \cK_{(p,q)}^{\WD_n} := \left\{ (ws_0,\diamond) : w \in \cK_{(p,q)}^{\W_n}\right\}\subseteq (\WD_n)^+.\]
The unique minimal-length element of $\cK_{(p,q)}^{\WD_n}$ is either
\[\bar 1 \bar 2 \cdots \bar q (q+1)(q+2)\cdots n \in \WD_n \quord 
\( 1 \bar 2 \cdots \bar q (q+1)(q+2)\cdots n,\diamond\)\in (\WD_n)^+\]
according to whether $q$ is even or odd.
Let 
$\cK_{\id}^{\WD_n} := \{ 1\}$ and $\cK_{\id^+}^{\WD_n}:= \{ w_0^+\}$.
Recall that if $n$ is even then $\cK^{\W_n}_{\fpf}$ is the set of elements $z=z^{-1} \in W_n$ with $|z(i)| \neq i$ for all $i \in [n]$.
If $z$ belongs to this set and $z(i) = -j$ for some $i,j \in [n]$ then $i\neq j$ and $z(j) = -i$,
so $z \in \WD_n$.
Define \[
\cK^{\WD_n}_{\fpf} := \left \{ z \in \cK^{\W_n}_{\fpf}: |\{ i \in [n] : w(i) < 0\}|\text{ is divisible by $4$}\right\}
\quand \cK^{\WD_n}_{\fpf^\diamond} =\cK^{\W_n}_{\fpf} - \cK^{\WD_n}_{\fpf}.\]
The unique minimal-length elements of $\cK^{\WD_n}_{\fpf}$
and $\cK^{\WD_n}_{\fpf^\diamond}$ are respectively
\[s_1s_3s_5\cdots s_{n-1}
\quand s_{-1}s_3s_5\cdots s_{n-1}.\]
When $n \in \{2,3\}$ or $n\geq 5$ the distinct perfect conjugacy classes in $W^+$ consist of 
$\cK^{\WD_n}_{\id}$, $\cK^{\WD_n}_{\id^+}$, and $\cK^{\WD_n}_{(p,q)}$ for all $p,q>0$ with $p+q=n$,
along with $\cK^{\WD_n}_\fpf$ and $\cK^{\WD_n}_{\fpf^\diamond}$ if $n$ is even  \cite[Ex. 9.2]{RV}.
When $n=1$ the only perfect conjugacy class is $\cK^{\WD_1}_{\id}=\cK^{\WD_1}_{\id^+} = \{1\}$ since $\WD_1=\{1\}$ is trivial.
The case $n=4$ is exceptional.
For $p,q>0$ with $p+q=4$ define
\[
\cK_{(p,q,\circlearrowright)}^{\WD_4} := \left\{ z^\circlearrowright : z \in \cK_{(p,q)}^{\WD_4}\right\}
\quand
\cK_{(p,q,\circlearrowleft)}^{\WD_4} := \left\{ z^\circlearrowleft : z \in \cK_{(p,q)}^{\WD_4}\right\}.
\]
%Also let $u,v,w \in \WD_4$ be the elements
%\[ 
%\ba u & := s_1s_2s_{-1}s_1 s_2s_1 &&= 1 \bar 2 \bar 3  4 , \\
%v & := s_3s_2s_{1}s_3 s_2s_3 &&= 4321  =u^\circlearrowright, \\
%w & := s_{-1}s_2s_{3}s_{-1} s_2s_{-1} &&= \bar 4 32 \bar 1  =  u^\circlearrowleft.
%\ea
%\]
There are $11$ perfect conjugacy classes in $(\WD_4)^+$, consisting of 
 $\cK^{\WD_4}_{\id}$ and $\cK^{\WD_4}_{\id^+}$ together with
\[
\barr{lll}
\cK^{\WD_4}_{(3,1)}\ni(1,\diamond), 
\quad& 
\cK^{\WD_4}_{(3,1,\circlearrowright)} \ni(1,\circlearrowright\diamond \circlearrowleft), 
\quad& 
\cK^{\WD_4}_{(3,1,\circlearrowleft)} \ni(1,\circlearrowleft\diamond \circlearrowright),
\\[-4pt]\\
\cK^{\WD_4}_{(2,2)}\ni s_{-1}s_1, \quad& 
\cK^{\WD_4}_{(2,2,\circlearrowright)} = \cK^{\WD_4}_{\fpf}  \ni s_{1}s_3,
\quad& 
\cK^{\WD_4}_{(2,2,\circlearrowleft)} = \cK^{\WD_4}_{\fpf^\diamond}  \ni s_{3}s_{-1},
\\[-4pt]\\
\cK^{\WD_4}_{(1,3)} \ni (1 \bar 2 \bar 3  4 , \diamond), \quad& \cK^{\WD_4}_{(1,3,\circlearrowright)}  \ni (4321, \circlearrowright\diamond \circlearrowleft) ,
\quad& \cK^{\WD_4}_{(1,3,\circlearrowleft)} \ni (\bar 4 32 \bar 1, \circlearrowleft\diamond \circlearrowright).
\earr
\]
 This listed elements of each class are the unique ones of minimal length.
 
\subsection{Model indices in type D}
%\subsection{Linear characters in type D}
If $n > 2$ then $\one$ and $\sgn$ are the only linear characters of $\WD_n$.
If $n=2$ then there are four linear characters given by
$\one_{++} :=\one$, $\one_{--} := \sgn$, $\one_{+-}$, and $\one_{-+}$,
where the last two satisfy 
$\one_{\pm\mp}: s_{-1} \mapsto \pm 1$ and  $\one_{\pm\mp}:s_{1} \mapsto \mp 1$.
One has
$ \one^\diamond = \one,$ $ \sgn^\diamond=\sgn$,
and
$\one_{\pm\mp}^\diamond = \one_{\mp\pm}$
where for any function $f : \WD_n\to\CC$ we write $f^\diamond $ for the map with $w \mapsto f(w^\diamond)$.

Let $\MTI(\WD_n)$ denote the set of $3\times 2$ arrays,
to be called \defn{model indices} for $\WD_n$,
 of the form
\[
\Theta = \left[\begin{smallmatrix} \alpha_0 & \alpha_1 \\ 
\beta_0& \beta_1  \\ 
\gamma_0 & \gamma_1 
\end{smallmatrix}\right]
\]
where %$\alpha_0 \in  \{0\}\sqcup \{2,3,\dots,n\}$ and 
$ \alpha_1 \in \{-n,n\}\sqcup \{0,1,\dots,n-2\}$  
%have $\alpha_0+|\alpha_1|=n$;
and $\alpha_0 = n-|\alpha_1|$;
where $\beta_0$ is either 
\begin{itemize}
\item a symbol in  $\{\id, \id^+,\fpf,\fpf^\diamond\}$, with $\beta_0\in \{\fpf,\fpf^\diamond\}$ allowed only if $\alpha_0$ is even,
\item when $\alpha_0>2$, a pair of positive integers $(p,q)$ with $p+q=\alpha_0$, or
\item when $\alpha_0=4$, one of the triples 
$(3,1,\circlearrowright)$, 
$(3,1,\circlearrowleft)$,
$(1,3,\circlearrowright)$, or
$(1,3,\circlearrowleft)$;
\end{itemize}
where $\beta_1$ is a symbol in  $\{\id, \id^+, \fpf, \fpf^+\}$,
with $\beta_1 \in \{\fpf,\fpf^+\}$ only if $|\alpha_1|\in \{4,6,8,\dots\}$;
and where $\gamma_0$ and $\gamma_1$ are linear characters of $\WD_{\alpha_0}$ and $S_{|\alpha_1|}$.
Let $\Theta = \left[\begin{smallmatrix} \alpha_0 & \alpha_1 \\ 
\beta_0& \beta_1  \\ 
\gamma_0 & \gamma_1 
\end{smallmatrix}\right] \in \MTI(\WD_n)$. Define 
\[ J_0 := 
\begin{cases}
\varnothing &\text{if }\alpha_0 = 0 \\
\{ s_{-1},s_1,s_2,\dots,s_{\alpha_0-1} \} &\text{otherwise}
\end{cases}
\quand
J_1 := 
\{ s_1,s_2,\dots,s_{|\alpha_1|-1} \}. \]
% Let $\varphi_0 $ be the identity map on $(\W_{\alpha_0})^+ = \langle J_0\rangle^+$ and
 Define $\varphi_1:  S_{|\alpha_1| }^+ \to \langle J_1\rangle^+$  to be the isomorphism
 sending $s_j \mapsto s_{\alpha_0 + j}$ for $j \in [|\alpha_1|-1]$.
 %This extends to an isomorphism $S_{|\alpha_1|}^+ \xrightarrow{\sim} \langle J_1\rangle^+$. 
% via the formula $(w,\theta) \mapsto (\varphi_i(w), \varphi_i \theta\varphi_i^{-1})$.
When $\alpha_0\neq 0$ define $\cK_0 = \cK^{\W_{\alpha_0}}_{\beta_0}$.
When $\alpha_1 \neq 0$ define $\cK_1$ to be the image of $\cK^{S_{|\alpha_1|}}_{\beta_1}$ under $\varphi_1$.
%When $\alpha_0$ and $\alpha_1$ are both nonzero,  
Let
%\be \TT^\Theta :=( J_0, \cK_0, \gamma_0) \otimes( J_1, \cK_1, \gamma_1). \ee
%When $\alpha_0=0$ or $\alpha_1=0$, we define
\[ \TT^\Theta :=
\begin{cases} ( J_0, \cK_0, \gamma_0)& \text{if }\alpha_0=n, \\
( J_1, \cK_1, \gamma_1)& \text{if }\alpha_1=n, \\
( J_1, \cK_1, \gamma_1)^\diamond &\text{if }\alpha_1=-n \\
( J_0, \cK_0, \gamma_0) \otimes( J_1, \cK_1, \gamma_1) &\text{otherwise}.
\end{cases} 
\]
This gives a  model triple $\TT^\Theta$ for $\WD_n$ which is factorizable\footnote{
Note that $\TT^\Theta$ would not be factorizable if we allowed $\beta_0 = (p,q)=(1,1)$ when $\alpha_0=2$.},
and by Theorems~\ref{tech1-thm} and \ref{tech2-thm} every multiplicity-free model triple for $\WD_n$ arises 
from this construction. As usual we also define
\[
\chi_\D^\Theta := \chi^{\TT^\Theta}.
\]
If $\alpha_0 \in \{0,1\}$ then $\TT^\Theta$ does not depend on $\beta_0$ or $\gamma_0$
and if $\alpha_1=0$ then $\TT^\Theta$ does not depend on $\beta_1$ or $\gamma_1$.
If $\alpha_1 =1$ then
$\TT^\Theta$ is unaffected by changing $\beta_1=\id$ to $\id^+$ (or vice versa) or $\gamma_1=\one$ to $\sgn$ (or vice versa).
%\item Changing $\beta_1= \id$ to $ \fpf^+$ (or vice versa) or $\beta_1 = \id^+$ to $ \fpf$ (or vice versa) when $|\alpha_1|=2$.
%\end{itemize}
%By Theorems~\ref{tech1-thm} and \ref{tech2-thm},
% every multiplicity-free model triple for $\WD_n$ with $n\geq 4$ arises as $\TT^\Theta$ for some $\Theta \in \MTI(\WD_n)$.

%\begin{proposition}
%Every multiplicity-free model triple for $\WD_n$ arises as $\TT^\Theta$ for some $\Theta \in \MTI(\WD_n)$.
%\end{proposition}
%
%\begin{proof} TODO? \end{proof}

Let $\Theta = \left[\begin{smallmatrix} \alpha_0  &\alpha_1 \\ 
\beta_0  &\beta_1 \\ 
\gamma_0 &\gamma_1 
\end{smallmatrix}\right] \in \MTI(\WD_n)$. Define $\overline\Theta := \left[\begin{smallmatrix} \alpha_0  &\alpha_1   \\ 
\beta_0  &\beta_1   \\ 
\bar\gamma_0 &\bar\gamma_1
\end{smallmatrix}\right]$ where $\bar \gamma_i:=   \gamma_i\sgn$.
Next let
\be
\Theta^\vee := \left[\begin{smallmatrix} \alpha_0  &\alpha_1^\diamond \\ 
\beta^\vee_0  &\beta^\vee_1 \\ 
\gamma_0^\diamond &\gamma_1 
\end{smallmatrix}\right] \text{ if $n$ is odd}
\quand
\Theta^\vee :=  \left[\begin{smallmatrix} \alpha_0  &\alpha_1 \\ 
\beta^\vee_0  &\beta^\vee_1 \\ 
\gamma_0 &\gamma_1 
\end{smallmatrix}\right] \text{ if $n$ is even,}
\ee
where we set $\alpha_1^\diamond :=-\alpha_1$ when $|\alpha_1|=n$
and $\alpha_1^\diamond := \alpha_1$ otherwise,
and define
\[ \beta^\vee_0:=\begin{cases} (q,p) &\text{if $\beta_0  =(p,q)$} 
\\
(q,p,\circlearrowright) &\text{if $\beta_0  =(p,q,\circlearrowright)$ and $n$ is even} 
\\
(q,p,\circlearrowright) &\text{if $\beta_0  =(p,q,\circlearrowleft)$ and $n$ is odd} 
\\
(q,p,\circlearrowleft) &\text{if $\beta_0  =(p,q,\circlearrowleft)$ and $n$ is even} 
\\
(q,p,\circlearrowleft) &\text{if $\beta_0  =(p,q,\circlearrowright)$ and $n$ is odd} 
\\
\fpf &\text{$\beta_0 = \fpf$ and $\frac{\alpha_0}{2}+n$ is even} \\
\fpf &\text{$\beta_0 = \fpf^\diamond$ and $\frac{\alpha_0}{2}+n$ is odd} \\
\fpf^\diamond &\text{$\beta_0 = \fpf^\diamond$ and $\frac{\alpha_0}{2}+n$ is even} \\
\fpf^\diamond &\text{$\beta_0 = \fpf$ and $\frac{\alpha_0}{2}+n$ is odd} \\
  \id &\text{if $\beta_0 = \id^+$} \\
  \id^+ &\text{if $\beta_0 = \id$} \end{cases}
  \quand
   \beta^\vee_1:=\begin{cases} 
\fpf &\text{$\beta_1 = \fpf^+$} \\
\fpf^+ &\text{if $\beta_1= \fpf$} \\
\id &\text{if $\beta_1 = \id^+$} \\
  \id^+ &\text{if $\beta_1 = \id$}. \end{cases}
\]
Next define $ \Theta^\diamond :=  \left[\begin{smallmatrix} \alpha_0  &\alpha_1^\diamond \\ 
\beta_0^\diamond  & \beta_1 \\ 
\gamma_0^\diamond &\gamma_1 
\end{smallmatrix}\right]
$
where $\alpha_1^\diamond$ is as above and 
\[
\beta_0^\diamond
:=
\begin{cases}
(p,q,\circlearrowleft) &\text{if }\beta_0=(p,q,\circlearrowright) \\
(p,q,\circlearrowright) &\text{if }\beta_0=(p,q,\circlearrowleft) \\
\fpf^\diamond &\text{if }\beta_0=\fpf \\
\fpf &\text{if }\beta_0=\fpf^\diamond \\
\beta_0 &\text{otherwise}.
\end{cases}
\] 
We always have $\gamma_0^\diamond = \gamma_0$
unless $\alpha_0=2$ and $\gamma_0 = \one_{\pm\mp}$.
%Finally let $\chi_\D^\Theta := \chi^{\TT^\Theta}$.
% If $\Theta,\Psi \in \MTI(\WD_n)$
%then we write $\Theta \equiv \Psi$ if 
%$\TT^\Theta \equiv \TT^\Psi$, 
%$\Theta\sim \Psi$ if $\TT^\Theta \sim \TT^\Psi$, and
% $\Theta \approx\Psi$ if $\TT^\Theta\approx \TT^\Psi$.
% In the second two cases we say that $\Theta$ and $\Psi$ are \defn{strongly equivalent} and \defn{equivalent}.

We adapt the relations $\equiv$, $\sim$, and $\approx$ to (sets of) model indices in $\MTI(\WD_n)$
in the same way that we did for elements of $\MTI(S_n)$ and $\MTI(\W_n)$.
 It is easy to check that
 $\TT^{\overline\Theta} = \overline{\TT^\Theta}$ and $\TT^{\Theta^\diamond} = (\TT^\Theta)^\diamond$,
 and that if $ \Theta'$ is formed from $\Theta$ by changing any entries equal to $\id^+$  to $\id$,
 then $\TT^\Theta \equiv \TT^{\Theta'}$.
 It is somewhat more involved, but still straightforward, to verify that
$\TT^{\Theta^\vee} = (\TT^\Theta)^\vee$.
We have $\Theta \sim \Theta^\vee \sim \Theta^\diamond$ when $n$ is odd and 
 $\Theta \sim \Theta^\vee \approx \Theta^\diamond$ when $n$ is even.
%so $\Theta \sim \Theta^\vee \approx \overline\Theta$.
% Finally, if $ \Theta'$ is formed from $\Theta$ by changing any entries equal to $\id^+$  to $\id$,
% then $\Theta \equiv \Theta'$.
 
%\begin{lemma}
%Let $\Theta \in \MTI(\WD_n)$. Then
%$\TT^{\overline\Theta} = \overline{\TT^\Theta}$, 
%$\TT^{\Theta^\vee} = (\TT^\Theta)^\vee$,  and 
%$\TT^{\Theta^\diamond} = (\TT^\Theta)^\diamond$,
%so $\Theta \sim \Theta^\vee \sim \Theta^\diamond$ when $n$ is odd and 
% $\Theta \sim \Theta^\vee \approx \Theta^\diamond$ when $n$ is even.
% Form $ \Theta'$ from $\Theta$ by changing any    entries  equal to $\id^+$   to $\id$.
%%Form $\Theta''$ from $\Theta'$ by changing the entry in the second row and second column
%%to $\fpf$ if this is  $\fpf^+$.
%Then $\Theta \equiv \Theta'$. 
%% and $\chi_\D^\Theta=\chi_\D^{\Theta'} = \chi_\D^{\Theta''}$.
%\end{lemma}
%
% \begin{proof}
% straightforward
% \end{proof}

%\subsection{Irreducible characters in type D}\label{char-d-sect}
 \subsection{Littlewood-Richardson coefficients in type D} \label{LR-d-sect}

If $(\lambda,\mu) \vdash n$  with $\lambda \neq \mu$ then the restricted character
\[ \chi^{\{\lambda,\mu\}} := \Res^{\W_n}_{\WD_n}( \chi^{(\lambda,\mu)}) =  \Res^{\W_n}_{\WD_n}( \chi^{(\mu,\lambda)})
\in \Irr(\WD_n)\]
is irreducible. In this case we refer to $\{\lambda,\mu\}$ as an unordered bipartition of $n$
and write $\{\lambda,\mu\}\vdash n$. If $n$ is even and $\nu \vdash n/2$ then 
\[ \Res^{\W_n}_{\WD_n}( \chi^{(\nu,\nu)})=\chi^{[\nu,+]}+ \chi^{[\nu,-]}\] for two different irreducible characters $\chi^{[\nu,\pm]}$.
The distinct elements of $\Irr(\WD_n)$ 
consist of
$ \chi^{\{\lambda,\mu\}}$ for all $\{\lambda,\mu\}\vdash n$ together with the \defn{degenerate} characters
  $ \chi^{[\nu,\pm]}$ for all $\nu \vdash n/2$ when $n$ is even.
  We distinguish the degenerate irreducible characters
  by requiring that  $\chi^{[\nu,+]}(w_\fpf) - \chi^{[\nu,-]}(w_\fpf) $
  be positive for the element $w_\fpf :=s_1s_3s_5\cdots s_{n-1}$\footnote{
  The convention in \cite[\S3]{Geck} for distinguishing the characters $\chi^{[\nu,\pm]}$ is instead to require that $\chi^{[\nu,+]}(w_\fpf) - \chi^{[\nu,-]}(w_\fpf)) = (-2)^{n/2}\chi^\nu(1)$.  This gives the same labels if $n$ is a multiple of $4$.
 };
by \cite[Lem. 3.5]{Geck} it then holds that
 \be\label{pmpm-eq} \chi^{[\nu,+]}(w_\fpf) - \chi^{[\nu,-]}(w_\fpf)  = 2^{n/2}\chi^\nu(1).\ee 
  This follows the convention of the data returned by the {\tt CharacterTable("WeylD", n)} command in \textsf{GAP} \cite{GAP}.
 
In this notation, the linear characters of $\WD_n$ are
$ \one = \chi^{\{(n),\emptyset\}} $ and
$ \sgn = \chi^{\{(1,1,\dots,1),\emptyset\}}$, together with $ \one_{-+} = \chi^{[(1),+]} $ and $ \one_{+-} = \chi^{[(1),-]} $
when $n=2$.
 As explained in \cite[Lem. 3.5]{Geck} (correcting an error in \cite[Rem. 5.6.5]{GeckP}), one has
\be\label{sgn-d-eq} \chi^{\{\lambda,\mu\}} \sgn =  \chi^{\{\lambda^\top,\mu^\top\}}
\quand  \chi^{[\nu,\pm]} \sgn =  \begin{cases}
\chi^{[\nu^\top,\pm]} &\text{if $n/2$ is even} \\
\chi^{[\nu^\top,\mp]} &\text{if $n/2$ is odd}.
\end{cases}
\ee
If $\lambda\neq \mu$ then $(\chi^{\{\lambda,\mu\}})^\diamond = \chi^{\{\lambda,\mu\}}$
and $(\chi^{[\nu,\pm]})^\diamond = \chi^{[\nu,\mp]}$.
 Finally, note that
 \be\label{bcd-reciprocity}
 \Ind_{\WD_n}^{\W_n}(\chi^{\{\lambda,\mu\}}) = \chi^{(\lambda,\mu)} + \chi^{(\mu,\lambda)}
 \quand
 \Ind_{\WD_n}^{\W_n}(\chi^{[\nu,\pm]}) = \chi^{(\nu,\nu)}
 \ee
 for all $\{\lambda,\mu\} \vdash n$ and $\nu \vdash n/2$ by Frobenius reciprocity.

% \subsection{Littlewood-Richardson coefficients in type D} \label{LR-d-sect}
 
Let $p,q \in \NN$ and recall that $s_{-p} := (p,-p-1)(p+1,-p)$ if $p>0$. Define 
\[ J = \begin{cases}
\varnothing &\text{if }p\leq 1 \\
\{s_{-1},s_1,\dots,s_{p-1}\} &\text{if }p\geq 2
\end{cases}
\quand
K = \begin{cases}
\varnothing &\text{if }q\leq 1 \\
\{s_{-p-1},s_{p+1},s_{p+2},\dots,s_{p+q-1}\} &\text{if }q\geq 2.
\end{cases}
\]
We identify $\WD_p\times \WD_q$ with the subgroup $\langle J\sqcup K \rangle \subseteq \WD_{p+q}$.
Write $u\times v$ for the image of $(u,v)  \in \WD_p\times \WD_q$ in $\WD_{p+q}$.
Given maps $f : \WD_p \to \CC$ and $g : \WD_q \to \CC$
 define $f\boxtimes g : \WD_p\times \WD_q \to \CC$ by the formula $u\times v \mapsto f(u)g(v)$.
When $f$ and $g$ are class functions, let
\be f\bullet_\D g := \Ind_{\WD_p\times \WD_q}^{\WD_{p+q}}(f\boxtimes g)
.\ee
This is another associative  and bilinear operation.
If $\Theta = \left[\begin{smallmatrix}  \alpha_0& \alpha_1  \\ 
\beta_0 & \beta_1   \\ 
\gamma_0 & \gamma_1
\end{smallmatrix}\right]\in \MTI(\WD_n)$ then %using this convention we can write
\be\label{d-theta-bullet-eq}
\chi_\D^\Theta = 
\chi_\D^{ \left[\begin{smallmatrix} \alpha_0  \\ 
\beta_0   \\ 
\gamma_0 
\end{smallmatrix}\right]}
\bullet_\D 
\Ind_{S_{\alpha_1}}^{\WD_{\alpha_1}}\Bigl(\chi_\A^{ \left[\begin{smallmatrix} \alpha_1  \\ 
\beta_1   \\ 
\gamma_1 
\end{smallmatrix}\right]}\Bigr)
\quad\text{where we define }\chi_\D^{ \left[\begin{smallmatrix} \alpha_0  \\ 
\beta_0   \\ 
\gamma_0 
\end{smallmatrix}\right]} := \chi_\D^{ \left[\begin{smallmatrix} \alpha_0 & 0  \\ 
\beta_0  & \id \\ 
\gamma_0  & \one
\end{smallmatrix}\right]}
\ee
and where we interpret the characters on the right as the trivial characters of 
$\WD_0:=\{1\}$ or $S_0:=\{1\}$
 when $\alpha_0=0$ or $\alpha_1=0$.
We need to explain how to evaluate this formula when
 $\alpha_0=0$ and $\alpha_1=-n<0$. For that case, we define
$ S_{-n} := S_n^\diamond = \langle s_{-1} ,s_2,s_3,\dots,s_{n-1}\rangle\subseteq \WD_n,$
 viewing $S_{-1}=\{1\},$
and extend any $f: S_n \to \CC$ 
to a map $S_{-n} \cup S_n \to \CC$ by setting $f(w) := f(w^\diamond)$ for $w \in S_{-n}$.
This is well-defined as $\diamond$ fixes every element of $S_{-n}\cap S_n \cong S_{n-1}$.

There are coefficients $d_{\Lambda\Gamma}^\Upsilon \in \NN$
whenever $\Lambda$, $\Gamma$, and $\Upsilon$ are unordered bipartitions or symbols of the form 
$[\nu,\pm]$
such that
\be\label{LR-d-eq} \chi^\Lambda \bullet_\D \chi^\Gamma
= \sum_{\Upsilon} d_{\Lambda\Gamma}^\Upsilon \chi^\Upsilon.\ee
Taylor \cite[Prop. 2.7]{Taylor} has shown how to express these numbers 
in terms of the Littlewood-Richardson coefficients $c_{\lambda\mu}^\nu$ from Section~\ref{LR-sect1}.
Namely, 
if $\lambda_1\neq \lambda_2$, $\lambda$, $\mu_1\neq \mu_2$, $\mu$,  $\nu_1\neq \nu_2$, $\nu$ are partitions and $\epsilon_\lambda,\epsilon_\mu,\epsilon_\nu \in \{ \pm \}$ are signs then we have:
\ben
\item[(a)] $ d_{\{\lambda_1,\lambda_2\} \{\mu_1,\mu_2\}}^{\{\nu_1,\nu_2\}}
= c_{\lambda_1\mu_1}^{\nu_1} c_{\lambda_2\mu_2}^{\nu_2} 
+
c_{\lambda_1\mu_2}^{\nu_1} c_{\lambda_2\mu_1}^{\nu_2} 
+
c_{\lambda_2\mu_1}^{\nu_1} c_{\lambda_1\mu_2}^{\nu_2} 
+
c_{\lambda_2\mu_2}^{\nu_1} c_{\lambda_2\mu_2}^{\nu_2} $,

\item[(b)] $d_{[\lambda,\epsilon_\lambda] \{\mu_1,\mu_2\}}^{\{\nu_1,\nu_2\}}
=
d_{\{\mu_1,\mu_2\}[\lambda,\epsilon_\lambda]}^{\{\nu_1,\nu_2\}}
= c_{\lambda\mu_1}^{\nu_1} c_{\lambda\mu_2}^{\nu_2} 
+
c_{\lambda\mu_2}^{\nu_1} c_{\lambda\mu_1}^{\nu_2} $,

\item[(c)] $d_{[\lambda,\epsilon_\lambda] [\mu,\epsilon_\mu]}^{\{\nu_1,\nu_2\}}
= c_{\lambda\mu}^{\nu_1} c_{\lambda\mu}^{\nu_2}$,

\item[(d)] $d_{\{\lambda_1,\lambda_2\} \{\mu_1,\mu_2\}}^{[\nu,\epsilon_\nu]}
= c_{\lambda_1\mu_1}^{\nu} c_{\lambda_2\mu_2}^{\nu} 
+
c_{\lambda_1\mu_2}^{\nu} c_{\lambda_2\mu_1}^{\nu}  $,

\item[(e)] $d_{[\lambda,\epsilon_\lambda] \{\mu_1,\mu_2\}}^{[\nu,\epsilon_\nu]}
=
d_{\{\mu_1,\mu_2\}[\lambda,\epsilon_\lambda]}^{[\nu,\epsilon_\nu]}
= c_{\lambda\mu_1}^{\nu} c_{\lambda\mu_2}^{\nu} $, and

\item[(f)] $d_{[\lambda,\epsilon_\lambda] [\mu,\epsilon_\mu]}^{[\nu,\epsilon_\nu]}
= \tfrac{1}{2} c_{\lambda\mu}^{\nu} (c_{\lambda\mu}^{\nu} + \epsilon_1\epsilon_2\epsilon_3)  $
where
to evaluate $\epsilon_1\epsilon_2\epsilon_3 \in \{\pm 1\}$ we replace $\pm$ by $\pm1$.\footnote{
This formula still holds if one defines $\chi^{[\nu,\pm]}$ such that $\chi^{[\nu,+]}(w_\fpf) - \chi^{[\nu,-]}(w_\fpf) = (-2)^{n/2} \chi^\nu(1)$ as in \cite[\S3]{Geck}, since switching to this convention would reverse either zero or two
  of the signs $\epsilon_1$, $\epsilon_2$, $\epsilon_3$.}
\een
Let $\nu\vdash n$.
In view of \eqref{LR-ba-eq} and \eqref{bcd-reciprocity},
if $n$ is odd then
\be \label{LR-da-eq1}
\Ind_{S_n}^{\WD_n}(\chi^\nu)=\Ind_{S_{-n}}^{\WD_n}(\chi^\nu) = \sum_{\{\lambda,\mu\}\vdash n} c_{\lambda\mu}^\nu \chi^{\{\lambda,\mu\}}
\ee
while if $n$ is even then
there are numbers $c_{[\lambda,\pm]}^\nu \in \NN$ for  $\lambda \vdash \frac{n}{2}$ 
with $c_{[\lambda,+]}^\nu+ c_{[\lambda,-]}^\nu = c^{\nu}_{\lambda\lambda}$ and
\be\label{LR-da-eq2}
\ba
\Ind_{S_{\pm n}}^{\WD_n}(\chi^\nu) &= \sum_{\{\lambda,\mu\}\vdash n} c_{\lambda\mu}^\nu \chi^{\{\lambda,\mu\}}
+ \sum_{\lambda\vdash \frac{n}{2}} \(c_{[\lambda,+]}^\nu \chi^{[\lambda,\pm ]} +  c_{[\lambda,-]}^\nu \chi^{[\lambda,\mp]}\).
%\\
%\Ind_{S_{-n}}^{\WD_n}(\chi^\nu) &= \sum_{\{\lambda,\mu\}\vdash n} c_{\lambda\mu}^\nu \chi^{\{\lambda,\mu\}}
%+ \sum_{\lambda\vdash n/2} \(c_{[\lambda,-]}^\nu \chi^{[\lambda,+]} +  c_{[\lambda,+]}^\nu \chi^{[\lambda,-]}\)
%.
\ea
\ee
As noted in \cite[\S1]{Taylor},
it appears to be an open problem to give a general formula for 
$c_{[\lambda,\pm]}^\nu$.
One can compute these numbers without much difficulty 
when $c_{\lambda\lambda}^\nu=1$. %,  as then we at least know that $\{ c_{[\lambda,+]}^\nu,c_{[\lambda,-]}^\nu\} =\{0,1\}$.
For example:

\begin{proposition}\label{pieri-d-prop} If $n$ is even then 
$
\Ind_{S_{ \pm n}}^{\WD_n}(\one) =  \sum_{p=0}^{\frac{n}{2}-1} \chi^{\{(p),(n-p)\}}  +  \chi^{[(\frac{n}{2}),\pm]}
$
and
\[
\Ind_{S_{ \pm n}}^{\WD_n}(\sgn) =    \sum_{p=0}^{\frac{n}{2}-1} \chi^{\{(1^p),(1^{n-p})\}} + \begin{cases}
\chi^{[(1^{n/2}),\pm]} &\text{if $n/2$ is even} \\
\chi^{[(1^{n/2}),\mp]} &\text{if $n/2$ is odd.}
\end{cases}
\]
\end{proposition}

\begin{proof}
The second formula follows from the first by \eqref{sgn-d-eq}.
In view of \eqref{pieri-bc-eq} we just need show that
$ \chi^{[(\frac{n}{2}),+]}$ is a constituent of 
$\Ind_{S_{  n}}^{\WD_n}(\one)$ and $ \chi^{[(\frac{n}{2}),-]}$ is a constituent of 
$\Ind_{S_{ - n}}^{\WD_n}(\one)$. For this,
it is enough by \eqref{pmpm-eq} to check that 
$\Ind_{S_{ n}}^{\WD_n}(\one)(w_\fpf) - \Ind_{S_{ -n}}^{\WD_n}(\one)(w_\fpf) >0$
for $w_\fpf := s_1s_3\cdots s_{n-1}$.
This follows from
 \eqref{ind-eq} since   $g\cdot w_\fpf\cdot g^{-1} \notin S_{-n}$ for all $g \in \WD_n$.
\end{proof}

Let $\EvenDRows(n)$ be the set of unordered bipartitions $\{ \lambda,\mu\}\vdash n$
where $\lambda \neq \mu$ have all even parts.
Let $\EvenDCols(n) = \{ \{\lambda^\top, \mu^\top\} : \{\lambda,\mu\} \in \EvenDRows(n)\}$. 

\begin{proposition}\label{ddd-prop1}
Suppose $n$ is even, $\beta \in \{\fpf,\fpf^\diamond\}$,
%, $\gamma \in \{\one,\sgn\}$, 
and $\Theta =  \left[\begin{smallmatrix} n  \\ 
\beta  \\ 
\gamma  
\end{smallmatrix}\right]$. 
  If $\gamma = \one$ then 
\[\chi_\D^{\Theta}
= 
\sum_{\{\lambda,\mu\} \in \EvenDRows(n)} \chi^{\{\lambda,\mu\}} +
\begin{cases}
 \sum_{\nu \in \EvenRows(n/2)} \chi^{[\nu,+]}
&\text{if }\beta=\fpf \\
 \sum_{\nu \in \EvenRows(n/2)} \chi^{[\nu,-]}
&\text{if }\beta=\fpf^\diamond
\end{cases}\]
 and if $\gamma = \sgn$ then 
\[\chi_\D^{\Theta}
= 
\sum_{\{\lambda,\mu\} \in \EvenDCols(n)} \chi^{\{\lambda,\mu\}} +
\begin{cases}
 \sum_{\nu \in \EvenCols(n/2)} \chi^{[\nu,+]}
&\text{if $\beta=\fpf$}  \\
 \sum_{\nu \in \EvenCols(n/2)} \chi^{[\nu,-]} 
&\text{if $\beta=\fpf^\diamond$} 
\end{cases}\]
% 
%\[
%\chi_\D^{\Theta}
%= 
%\begin{cases}
%\sum_{\{\lambda,\mu\} \in \EvenDRows(n)} \chi^{\{\lambda,\mu\}} + \sum_{\nu \in \EvenRows(n/2)} \chi^{[\nu,\epsilon_\nu]}
%&\text{if }\gamma=\one
%\\
% \sum_{\{\lambda,\mu\} \in \EvenDCols(n)} \chi^{\{\lambda,\mu\}} + \sum_{\nu \in \EvenCols(n/2)} \chi^{[\nu,\epsilon_\nu]}
%&\text{if }\gamma=\sgn
%\end{cases}
%\] for some choice of signs $\epsilon_\nu \in \{\pm\}$.
%In the right-most sums in both of these identities, the signs $\pm$ are left undetermined and  could depend on $\nu$.
where the sums over $\nu \in \EvenRows(n/2)$ and $\nu \in \EvenCols(n/2)$ are zero when $n\not\equiv 0 \modu 4)$.
\end{proposition}

\begin{proof}
Since $\chi_\D^{\Theta^\diamond}=(\chi_\D^\Theta)^\diamond$ and $\chi_\D^{\overline\Theta}=\sgn\chi_\D^\Theta $,
we may assume that $\beta =\fpf$ and $\gamma=\one$.
The centralizers of $s_{ 1}s_3\cdots s_{n-1}$ in $\WD_n$ and $\W_n$ coincide, so $\Ind_{\WD_n}^{\W_n}(\chi_\D^\Theta) = \chi_\BC^{\Theta}$.
By \eqref{bbb-prop1} and \eqref{bcd-reciprocity} we therefore have
$\chi_\D^{\Theta}
= 
\sum_{\{\lambda,\mu\} \in \EvenDRows(n)} \chi^{\{\lambda,\mu\}} +
 \sum_{\nu \in \EvenRows(n/2)} \chi^{[\nu,\epsilon_\nu]}$ for some choice of  $\epsilon_\nu \in \{\pm \}$.
 To show that $\epsilon_\nu$ is always $+$ it suffices by \eqref{pmpm-eq} to check that if $n$ is divisible by $4$
then
\[
\Ind_{C_\fpf}^{\WD_n}(\one)(w_\fpf) - \Ind_{C^\diamond_\fpf}^{\WD_n}(\one)(w_\fpf)=
2^{n/2} \sum_{\nu \in \EvenRows(n/2)} \chi^\nu(1)
\]
where $C_\fpf$ is the $\WD_n$-centralizer of $w_\fpf = s_1s_3\cdots s_{n-1}$
and $C_\fpf^\diamond = (C_\fpf)^\diamond$ is the $\WD_n$-centralizer of $w^\diamond_\fpf := s_{-1}s_3\cdots s_{n-1}$.
Because $|C_\fpf| = |C_\fpf^\diamond|$, the induced character formula \eqref{ind-eq} gives 
\[\ba
\Ind_{C_\fpf}^{\WD_n}(\one)(w_\fpf) - \Ind_{C^\diamond_\fpf}^{\WD_n}(\one)(w_\fpf) &= |\cK^{\WD_n}_{\fpf} \cap C_\fpf| - |\cK^{\WD_n}_{\fpf} \cap C_\fpf^\diamond| \\&=  |\cK^{\WD_n}_{\fpf} \cap C_\fpf|-  |\cK^{\WD_n}_{\fpf^\diamond} \cap C_\fpf|.\ea
\]
Let $A_i^+ := \{2i-1,2i\}$, $A_i^- := \{1-2i,-2i\}$, and $A_i := A_i^+ \sqcup A_i^-$ for $i \in [n/2]$. Each $w \in C_\fpf$
determines a permutation of $A_1,A_2,\dots,A_{n/2}$. If $w \in  \cK^{\WD_n}_{\fpf}\cap C_\fpf$
(respectively,  $w \in \cK^{\WD_n}_{\fpf^\diamond}\cap C_\fpf$)
then this permutation is an involution with $w(A^+_i)=  A^\pm_j$ whenever $w(A_i) = A_j$,
such that the number of  $i \in [n]$ with $w(A_i^+) = A_i^-$ is even
(respectively, odd).
These properties characterize the elements of  $\cK^{\WD_n}_{\fpf}\cap C_\fpf$ and $\cK^{\WD_n}_{\fpf^\diamond}\cap C_\fpf$,
and so there is a bijection from $\cK^{\WD_n}_{\fpf^\diamond}\cap C_\fpf$ 
to the subset of elements $w \in \cK^{\WD_n}_{\fpf}\cap C_\fpf$ with $w(A_i) = A_i$ for at least one $i \in [n/2]$.
Thus $ |\cK^{\WD_n}_{\fpf} \cap C_\fpf|-  |\cK^{\WD_n}_{\fpf^\diamond} \cap C_\fpf|$
is the number of $w \in \WD_n$ with $A_i =w^2(A_i)\neq w(A_i) \in \{A_1,A_2,\dots,A_{n/2}\}$ for each $i \in [n/2]$,
such that if $w(A_i) = A_j$ then $w : A_i \to A_j$ is any of the four bijections with $w(A_i^+) = A_j^\pm$.
The number of such permutations is $4^{n/4}$ times the number of fixed-point-free involutions of $[n/2]$.
This product is $2^{n/2} \sum_{\nu \in \EvenRows(n/2)} \chi^\nu(1)$ by \eqref{irs-eq} as needed.
\end{proof}

%Suppose $p$ and $q$ are positive integers with $p+q=n$.
%Define $\Lambda(p,q)$ to be the set of partitions of the form
%$ \lambda = (\max\{p,q\}+r, \min\{p,q\}-r)$ for $0 \leq r \leq \min\{p,q\}$.

\begin{proposition}\label{ddd-prop2}
Suppose $p$ and $q$ are positive integers with $p+q=n$.
%Define $\Lambda=\Lambda(p,q)$ to be the set of partitions of the form $ \lambda = (\max\{p,q\}+r, \min\{p,q\}-r)$ for $0 \leq r \leq \min\{p,q\}$.
Then %If $p$ and $q$ are positive integers with $p+q=n$ is even then
\[
\ds\chi_\D^{ \left[\begin{smallmatrix} n   \\ 
(p,q)  \\ 
\one  
\end{smallmatrix}\right]}
= \sum_{j=0}^{\min(p,q)}  \chi^{\{(n-j,j),\emptyset\}}
\quand 
\ds\chi_\D^{ \left[\begin{smallmatrix} n   \\ 
(p,q)  \\ 
\sgn  
\end{smallmatrix}\right]}
= \sum_{j=0}^{\min(p,q)} \chi^{\{(2^j,1^{n-j}),\emptyset\}}
.\]
\end{proposition}

\begin{proof}
Let $\Theta := \left[\begin{smallmatrix} n   \\ 
(p,q)  \\ 
\one
\end{smallmatrix}\right]$.
For either parity of $q$,
the centralizer of the unique minimal-length element of $\cK^{\WD_n}_{(p,q)}$ in $\WD_n$
is the intersection $H := (\W_q\times \W_p) \cap \WD_n$.
Let $K := \W_q\times \W_p$.
By Frobenius reciprocity we have 
$
\Ind_H^{K}(\one) = \chi^{((q),\emptyset)}\boxtimes  \chi^{((p),\emptyset)} 
+  \chi^{(\emptyset,(q))}\boxtimes  \chi^{(\emptyset,(p))}$
so 
\[
\Ind_{\WD_n}^{\W_n} (\chi_\D^\Theta)
= \Ind_{\WD_n}^{\W_n}\Ind_H^{\WD_n} (\one)
= \Ind_{K}^{\W_n} \Ind_H^{K}(\one)
= \sum_{j=0}^{\min(p,q)} ( \chi^{((n-j,j),\emptyset)} +  \chi^{((n-j,j),\emptyset)})
\]
by \eqref{LR-bc-eq}. The formula for $\chi_\D^\Theta$ follows by \eqref{bcd-reciprocity}.
The other formula holds since $\chi_\D^{\overline \Theta} = \chi_\D^\Theta \sgn$.
\end{proof}

Our last proposition records a calculation we performed in the 
 algebra system \textsf{GAP} \cite{GAP}:

\begin{proposition}\label{ddd-prop3}
Suppose $\beta \in \{(1,3,\theta), (3,1,\theta)\}$
where $\theta \in \{\circlearrowright,\circlearrowleft\}$.
Then %If $p$ and $q$ are positive integers with $p+q=n$ is even then
\[
\ds\chi_\D^{ \left[\begin{smallmatrix} 4   \\ 
\beta  \\ 
\one  
\end{smallmatrix}\right]}
=\begin{cases} 
 \chi^{\{(4),\emptyset\}} + \chi^{[(2),+]} &\text{if }\theta={\circlearrowright}
 \\
  \chi^{\{(4),\emptyset\}} + \chi^{[(2),-]}&\text{if }\theta={\circlearrowleft}
  \end{cases}
\quand 
\ds\chi_\D^{ \left[\begin{smallmatrix} 4   \\ 
\beta  \\ 
\sgn  
\end{smallmatrix}\right]}
=
\begin{cases} 
\chi^{\{(1,1,1,1),\emptyset\}} + \chi^{[(1,1),+]} &\text{if }\theta={\circlearrowright}
\\
\chi^{\{(1,1,1,1),\emptyset\}} + \chi^{[(1,1),-]} &\text{if }\theta={\circlearrowleft}.
\end{cases}
\]
\end{proposition}

%\begin{proof}
%calculation
%\end{proof}

\subsection{Model projections in type D}

We define a map $\pi_\D: \MTI(\WD_n) \to \MTI(S_n)\sqcup\{0\}$.
Let $\Theta = \left[\begin{smallmatrix} \alpha_0 & \alpha_1 \\ 
\beta_0& \beta_1\\ 
\gamma_0 & \gamma_1
\end{smallmatrix}\right] \in \MTI(\WD_n)$.
First  set 
\[ \pi_\D(\Theta) :=  \left[\begin{smallmatrix} |\alpha_1|  \\ 
 \beta_1    \\ 
 \gamma_1  
\end{smallmatrix}\right]\text{ if }\alpha_0=0
\quand
\pi_\D(\Theta) := 0
\text{ if $\gamma_0 \notin \{\one,\sgn\}$ (only possible if $\alpha_0=2$)}.
\]
 Assume $\alpha_0\neq 0$ 
 and $\gamma_0 \in \{ \one,\sgn\}$.
If $\alpha_1 \neq 0$ then we define
  \[
  \pi_\D(\Theta)  := 
\begin{cases}
  \left[\begin{smallmatrix} \alpha_0 & \alpha_1  \\ 
\fpf & \beta_1 \\ 
\gamma_0  & \gamma_1
\end{smallmatrix}\right]  & \text{if }\beta_0 \in \{ \fpf, \fpf^\diamond\} 
\\[-8pt]\\
  \left[\begin{smallmatrix} \alpha_0  &\alpha_1 \\ 
\id &\beta_1\\ 
\gamma_0 &\gamma_1
\end{smallmatrix}\right] &\text{if }\beta_0 \in \{ \id,\id^+,
(1,3,\circlearrowright), (1,3,\circlearrowleft),
(3,1,\circlearrowright), (3,1,\circlearrowleft) \}
\\[-8pt]\\
 \left[\begin{smallmatrix}p & q& \alpha_1  \\ 
\id &\id & \beta_1   \\ 
\gamma_0 & \gamma_0 &  \gamma_1 
\end{smallmatrix}\right]
&\text{if $\beta_0 = (p,q)$ for $p,q>0$ with $p+q=\alpha_0$}.
\end{cases}
\]
When $\alpha_1=0$,
we form $\pi_\D(\Theta)$ by
applying the same formula, and then deleting the last column
if the result is nonzero.

Define $\cR^\D := \bigoplus_{n \in \NN} \cR^\D_n$
where $\cR^\D_n$ is the $\CC$-vector space of class functions $\WD_n\to\CC$.
We use the same symbol $\pi_\D$
to denote the linear map $\cR^\D \to \cR^\A$ 
with 
\[
\pi_\D(\chi^{[\lambda,\pm]}) :=0
\quand
\pi_\D(\chi^{\{\lambda,\mu\}}) := \begin{cases} \chi^\lambda &\text{if }\mu=\emptyset \\
\chi^\mu &\text{if }\lambda=\emptyset \\
0&\text{otherwise}
\end{cases}
\]
for all partitions $\lambda \neq \mu$. Finally, set $\chi_\A^0 := 0 \in \cR^\A$.

\begin{lemma}\label{da-pi-lem}
If $\Theta  \in \MTI(\WD_n)$ then $\pi_\D(\chi_\D^\Theta) = \chi_\A^{\pi_\D(\Theta)}$.
\end{lemma}

\begin{proof}
%Our argument is similar to the proof of Lemma~\ref{bca-pi-lem}.
The identities \eqref{LR-d-eq}--\eqref{LR-da-eq2} %\eqref{LR-da-eq1}
imply
 that $\pi_\D(\chi \bullet_\D \Ind_{S_n}^{\W_n}(\psi)) = \pi_\D(\chi) \bullet_\A \psi$
for $\chi \in \cR^\D$ and $\psi \in \cR^\A_n$.
Given this plus \eqref{theta-bullet-eq} and \eqref{d-theta-bullet-eq},
 to show that $\pi_\D(\chi_\D^\Theta) = \chi_\A^{\pi_\D(\Theta)}$
for all $\Theta  \in \MTI(\W_n)$
it suffices to prove this identity  for model indices of the form $\Theta = \left[\begin{smallmatrix} \alpha_0  \\ 
\beta_0 \\ 
\gamma_0 
\end{smallmatrix}\right] $.
This follows from the definition of $\pi_\D$ on comparing Proposition~\ref{one-col-a-prop}
with Propositions~\ref{ddd-prop1}, \ref{ddd-prop2}, and \ref{ddd-prop3}.
\end{proof}

\subsection{Perfect models in type D}

Fix a model index $\Theta =\left[\begin{smallmatrix} 
\alpha_0 & \alpha_1 \\ 
\beta_0 & \beta_1   \\ 
\gamma_0 & \gamma_1
\end{smallmatrix}\right] \in \MTI(\WD_n)$.

\begin{lemma}\label{d-lem}
The character $\chi_\D^\Theta$ is not multiplicity-free if any of the following conditions hold:
\ben
\item[(a)] $\alpha_1 \in \{4,6,8,\dots\}$ and $\beta_1\in \{\fpf,\fpf^+\}$.
\item[(b)] $\alpha_0 \in \{4,6,\dots\}$, $\beta_0\in \{\fpf,\fpf^\diamond\}$, $\alpha_1\geq 2$, and $\gamma_0=\gamma_1 \in \{\one,\sgn\}$.
\item[(c)] $\beta_0 = (p,q)$ for some $p,q>0$ with $p+q=\alpha_0\geq 3$ and $\alpha_1\geq 1$. 
\een
\end{lemma}

\begin{proof}
Suppose (a) holds and let $\Psi := \left[\begin{smallmatrix} \alpha_1  \\ 
\beta_1   \\ 
\gamma_1 
\end{smallmatrix}\right]$ and $m := \alpha_1$. 
 As in the proof of Lemma~\ref{b-lem},
it is enough by \eqref{d-theta-bullet-eq} to show that  the character $\Ind_{S_{m}}^{\WD_{m}}(\chi_\A^{\Psi})$ is not multiplicity-free
when $\gamma_1=\one$.
When $m=4$ this can be checked by hand or using a computer algebra system.
The proof of Lemma~\ref{b-lem} shows that $\Ind_{\WD_m}^{\W_m}(\Ind_{S_{m}}^{\WD_{m}}(\chi_\A^{\Psi}))$
contains $\chi^{((m-2),(2))}$ with multiplicity at least two,
and when $m>4$ this can only occur in view of \eqref{bcd-reciprocity} if 
$\chi^{\{(m-2),(2)\}}$ appears with multiplicity at least two in $\Ind_{S_{m}}^{\WD_{m}}(\chi_\A^{\Psi})$.

If (b) or (c) holds then  $\pi_\D(\chi_\D^\Theta)=\chi_\A^{\pi_\D(\Theta)}\neq 0$ 
is not multiplicity-free by Lemma~\ref{a-lem}, so $\chi_\D^\Theta$ must also not be multiplicity-free.
\end{proof}

%
%\begin{lemma} \label{d-fpf-lem}
%Suppose $\Theta =\left[\begin{smallmatrix} 
%\alpha_0 & \alpha_1 \\ 
%\beta_0 & \beta_1   \\ 
%\gamma_0 & \gamma_1
%\end{smallmatrix}\right] \in\MTI(\WD_n)$
%has $\alpha_1\in\{4,6,8,\dots\}$ and $\beta_1\in \{\fpf,\fpf^+\}$.
%Then $\chi_\D^\Theta$ is not multiplicity-free.
%\end{lemma}
%
%
%\begin{proof}
%Let $\Psi := \left[\begin{smallmatrix} \alpha_1  \\ 
%\beta_1   \\ 
%\gamma_1 
%\end{smallmatrix}\right]$ and $n := \alpha_1$. As in the proof of Lemma~\ref{bc-fpf-lem},
%it is enough by \eqref{d-theta-bullet-eq} to show that  the character $\Ind_{S_{n}}^{\WD_{n}}(\chi_\A^{\Psi})$ is not multiplicity-free
%when $\gamma_1=\one$.
%When $n=4$ this can be checked by hand or using a computer algebra system.
%The proof of Lemma~\ref{bc-fpf-lem} shows that $\Ind_{\WD_n}^{\W_n}(\Ind_{S_{n}}^{\WD_{n}}(\chi_\A^{\Psi}))$
%contains $\chi^{((n-2),(2))}$ with multiplicity at least two,
%and when $n>4$ this can only occur in view of \eqref{bcd-reciprocity} if 
%$\chi^{\{(n-2),(2)\}}$ appears with multiplicity at least two in $\Ind_{S_{n}}^{\WD_{n}}(\chi_\A^{\Psi})$.
% \end{proof}

Let $\OddDRows(n,q)$ be the set of unordered bipartitions $\{\lambda,\mu\}\vdash n$ such that $\lambda \cup \mu$ has exactly $q$ odd parts.
Define $\OddDCols(n,q) = \{ \{\lambda^\top,\mu^\top\} : \{\lambda,\mu\}\in\OddDRows(n,q)\}$.

\begin{proposition} \label{d-index-prop}
%Assume $n\geq 6$.
Suppose $\Theta \in \MTI(\WD_n)$ 
and $\chi^\Theta_\D$ is
multiplicity-free.
Then
 $\Theta$ has one of the following forms:
\ben

%\item[] $\left[\begin{smallmatrix} 
%2 & n-2 \\ 
%\id & \id   \\ 
%\one_{\pm\mp} & \one
%\end{smallmatrix}\right]$ in which case
%$ \chi_\D^\Theta = \sum_{1 \leq p < \frac{n}{2}}  
%$

\item[(a)] $\left[\begin{smallmatrix} 
k & n-k \\ 
\id/\id^+ & \id/\id^+   \\ 
\gamma_0 & \gamma_0
\end{smallmatrix}\right]$ 
for some $k \in \{3,4,\dots,n\}$ and $\gamma_0,\gamma_1\in\{\one,\sgn\}$.
%for any values of $\alpha_0,\alpha_1,\gamma_0,\gamma_1$ allowed in the definition of $\MTI(\WD_n)$.

\item[(b)]  $\left[\begin{smallmatrix} 
2 & n-2 \\ 
\id/\id^+/\fpf/\fpf^\diamond & \id/\id^+   \\ 
\gamma_0 & \gamma_1
\end{smallmatrix}\right]$
for some $\gamma_0 \in \{\one,\sgn, \one_{+-},\one_{-+}\}$, and $\gamma_1\in\{\one,\sgn\}$.

\item [(c)] $\left[\begin{smallmatrix} 
0 & n \\ 
\id  & \id/\id^+   \\ 
\one & \gamma_1
\end{smallmatrix}\right]$ or $\left[\begin{smallmatrix} 
0 & -n \\ 
\id  & \id/\id^+   \\ 
\one & \gamma_1
\end{smallmatrix}\right]$ for some  $\gamma_1\in\{\one,\sgn\}$.

\item[(d)] $\left[\begin{smallmatrix} 
2k & n-2k \\ 
\fpf/\fpf^\diamond & \id/\id^+   \\ 
\one & \sgn
\end{smallmatrix}\right]$ for some $0\leq k  \leq \lfloor n/2\rfloor$, in which case 
\[
\chi_\D^\Theta =
\sum_{\{\lambda,\mu \} \in \OddDRows(n,n-2k)} \chi^{\{\lambda,\mu\}}
+\sum_{\nu \in \OddRows(\frac{n}{2},\frac{n}{2}-k)} \chi^{[\nu,\epsilon_\nu]}
\]
for some choice of signs $\epsilon_\nu \in \{\pm\}$, where the second sum is zero if $n$ is odd.

\item[(e)] $\left[\begin{smallmatrix} 
2k & n-2k \\ 
\fpf/\fpf^\diamond & \id/\id^+   \\ 
\sgn & \one
\end{smallmatrix}\right]$ for some  $0\leq k \leq \lfloor n/2\rfloor $, in which case 
\[
\chi_\D^\Theta =\sum_{\{\lambda,\mu \} \in \OddDCols(n,n-2k)} \chi^{\{\lambda,\mu\}}
+\sum_{\nu \in \OddCols(\frac{n}{2},\frac{n}{2}-k)} \chi^{[\nu,\epsilon_\nu]}
\]
for some choice of signs $\epsilon_\nu \in \{\pm\}$, where the second sum is zero if $n$ is odd.

\item[(f)] $\left[\begin{smallmatrix} n & 0 \\ (p,q)  & \id  \\ \gamma_0 & \one  \end{smallmatrix}\right] $ for some $p,q>0$ such that $2<p+q=n$ and $\gamma_0 \in \{\one,\sgn\}$.

\item[(g)] $\left[\begin{smallmatrix} 4 & n-4 \\ \beta  & \id/\id^+ \\ \gamma_0 & \gamma_1  \end{smallmatrix}\right] $ 
for some $\beta \in \{ 
(1,3,\circlearrowright), (1,3,\circlearrowleft),
(3,1,\circlearrowright), (3,1,\circlearrowleft) \}$
and $\gamma_0,\gamma_1 \in \{\one,\sgn\}$.

\een

\end{proposition}

The signs in parts (d) and (e) can be determined using \eqref{LR-d-eq} and Propositions~\ref{pieri-d-prop} and \ref{ddd-prop1},
but we will not need this for our applications.

\begin{proof}
The given cases account for all model indices in $\MTI(\WD_n)$ not excluded by Lemma~\ref{d-lem}.
The formulas in parts (d) and (e) follow by combining \eqref{d-theta-bullet-eq} with Propositions~\ref{pieri-d-prop} and \ref{ddd-prop1}.
 \end{proof}

 \begin{theorem} \label{d-thm}
 Assume $n\geq 4$. If $n$ is even  then $\WD_n$ has no perfect models.
If $n$ is odd then
\[ 
\ba
\cP^{\D}_n &:=\left\{ \TT^\Theta : \Theta = \left[\begin{smallmatrix} 
2k & n-2k \\ 
\fpf & \id   \\ 
\one & \sgn
\end{smallmatrix}\right] \text{ for }0\leq k \leq \lfloor\tfrac{n}{2}\rfloor \right\}
\ea
\]
is a perfect model  for $\WD_n$, and 
each perfect model for $\WD_n$ is strongly equivalent  $\cP^{\D}_n$ or $\overline{\cP^{\D}_n}$.
\end{theorem}

%The model $\cP^{\D}_n$ is also defined when $n=3$. 
The equivalence
class of $\cP^{\D}_3$ gives the extra models for $S_4$ described in Example~\ref{a-extra-ex}.

\begin{proof}
Our argument is similar to the proof of Theorem~\ref{b-thm}. 
When $n$ is odd
it is clear from Proposition~\ref{d-index-prop} that $\cP^{\D}_n$
 is a perfect model for $\WD_n$.
 
Suppose $\cM$ is a set of model indices $\Theta \in \MTI(\WD_n)$
such that $\{\TT^\Theta : \Theta \in \cM\}$ is a perfect model for $\WD_n$.
Every perfect model for $\WD_n$ arises in this way.
Define $\cM_\A := \{ \pi_\D(\Theta)  : \Theta \in \cM\}\setminus\{0\}$.
Then
$\{ \TT^\Theta : \Theta \in \cM_\A\}$ is a perfect model for $S_n$ by Lemma~\ref{da-pi-lem}.
After possibly replacing $\cM$ by $\overline\cM := \{ \overline \Theta : \Theta \in \cM\}$,
we may assume that $\{ \TT^\Theta : \Theta \in \cM_\A\}$ is a $\sgn$-model.

Again let $ \Theta_k := \left[\begin{smallmatrix} 
2k & n-2k \\ 
\fpf & \id   \\ 
\one & \sgn
\end{smallmatrix}\right]$  for $0 \leq k \leq \lfloor n/2\rfloor$. If $n$ is odd then we have 
\[ \Theta_k \sim   \left[\begin{smallmatrix} 
2k & n-2k \\ 
\fpf & \id   \\ 
\one & \sgn
\end{smallmatrix}\right]^\diamond = \left[\begin{smallmatrix} 
2k & n-2k \\ 
\fpf^\diamond & \id   \\ 
\one & \sgn
\end{smallmatrix}\right] \sim \left[\begin{smallmatrix} 
2k & n-2k \\ 
\fpf & \id^+   \\ 
\one & \sgn
\end{smallmatrix}\right]
\sim
\left[\begin{smallmatrix} 
2k & n-2k \\ 
\fpf^\diamond & \id^+   \\ 
\one & \sgn
\end{smallmatrix}\right].
\]
Since $\WD_2$ is abelian we have 
\[\Theta_2 \equiv    \left[\begin{smallmatrix} 
2 & n-2 \\ 
\beta_0 & \beta_1   \\ 
\one & \sgn
\end{smallmatrix}\right] \quad\text{for all $\beta_0 \in \{\id,\id^+,\fpf,\fpf^\diamond\}$ and $\beta_1 \in \{\id,\id^+\}$.}
\]

Now assume $n \geq 5$ is odd.
Remark~\ref{all-a-models-rmk} tells us that $\cM_\A$ must contain elements of  each of 
the forms  \eqref{m3-eq}, \eqref{m4-eq}, \eqref{m5-eq}, and \eqref{m6-eq}.
By considering the limited possibilities for  model indices $\Theta \in \MTI(\WD_n)$ 
with $\chi_\D^\Theta$ multiplicity-free
that can serve as the preimages for these elements under $\pi_\D$,
we deduce from Proposition~\ref{d-index-prop} that 
$\cM$ must contain a unique model index strongly equivalent to 
$\Theta_k$
for at least each $2 \leq k \leq \lfloor n/2\rfloor$.

By similar reasoning, for $\cM_\A$ to contain an index of the form \eqref{m1-eq},
 $\cM$ must contain a unique element strongly equivalent to
$
  \Theta_0
  =\left[\begin{smallmatrix} 
0 & n \\ 
\fpf & \id   \\ 
\one & \sgn
\end{smallmatrix}\right]
\sim
  \left[\begin{smallmatrix} 
0 & n \\ 
\id & \id   \\ 
\one & \sgn
\end{smallmatrix}\right]
\sim
  \left[\begin{smallmatrix} 
0 & -n \\ 
\id & \id   \\ 
\one & \sgn
\end{smallmatrix}\right]
$ or $
 \Psi_0:=\left[\begin{smallmatrix} 
n & 0 \\ 
\id & \id   \\ 
\sgn & \one
\end{smallmatrix}\right]
$,
and
for $\cM_\A$ to contain an index of the form \eqref{m2-eq},
 $\cM$ must contain a unique element strongly equivalent to
$
  \Theta_1
%  =\left[\begin{smallmatrix} 
%2 & n-2 \\ 
%\fpf & \id   \\ 
%\one & \sgn
%\end{smallmatrix}\right],
$ or $\Psi_1:=\left[\begin{smallmatrix} 
n-2 & 2 \\ 
\id & \id   \\ 
\sgn & \one
\end{smallmatrix}\right]
$.

Thus, $\cM$ contains a subset of model indices strongly equivalent to 
$\cM^0 \sqcup \cM^1 \sqcup \cM^2$
where
 $\cM^0$ is either $\{ \Theta_0\}$ or $\{\Psi_0\}$,
$\cM^1$ is either $\{ \Theta_1\}$ or $\{\Psi_1\}$, and 
 $\cM^2 := \{ \Theta_k : 2 \leq k \leq \lfloor n/2\rfloor\}$.
Moreover, all elements $\Phi \in \cM$ outside this set must have $\pi_\D(\Phi)=0$,
so are of the form
\be\label{psi-eq}
\Phi=
 \left[\begin{smallmatrix} 
2 & n-2 \\ 
\beta_0 & \beta_1   \\ 
\gamma_0 & \gamma_1
\end{smallmatrix}\right] \quad\text{for some 
$\gamma_0 \in \{\one_{-+},\one_{+-}\}$ and $\gamma_1 \in \{\one,\sgn\}$,
}
%\left[\begin{smallmatrix} 
%2 & n-2 \\ 
%\id/\id^+/\fpf/\fpf^\diamond & \id/\id^+   \\ 
%\one_{-+}/\one_{+-} & \one/\sgn
%\end{smallmatrix}\right].
\ee
where $\beta_0 \in \{\id,\id^+,\fpf,\fpf^\diamond\}$ and $\beta_1 \in \{\id,\id^+\}$ are arbitrary.
If $\cM^0 = \{ \Theta_0\}$ and $\cM^1 = \{\Theta_1\}$ then $\cM^0\sqcup \cM^1 \sqcup \cM^2 = \cP^{\D}_n$
so $\cM\sim \cP^{\D}_n$.
If this does not occur then the character 
$ \sum_{\Theta \in \cM^0\sqcup \cM^1 \sqcup \cM^2} \chi_\D^\Theta $
is missing several irreducible constituents. Specifically, 
we know that 
\[\chi_\D^{\Theta_0} = \sum_{p+q=n} \chi^{\{ (1^p),(1^{q})\}}
\quand
\chi_\D^{\Theta_1} = 
 \sum_{p+q=n-2} \chi^{\{ (2,1^{p}), (1^{q}) \}}+
\sum_{p+q = n-3} \chi^{\{ (3,1^{p}), (1^{q}) \}}\]
by Proposition~\ref{d-index-prop}, but one can compute using \eqref{LR-d-eq} that
$
\chi_\D^{\Psi_0} = \chi^{\{(1^n),\emptyset\}}
$
and
\[
\chi_\D^{\Psi_1} = \chi^{\{(1^{n-2}),(2)\}} +  \chi^{\{(2,1^{n-3}),(1)\}} +  \chi^{\{(1^{n-1}),(1)\}} + \chi^{\{(3,1^{n-3}),\emptyset\}} + \chi^{\{(2,1^{n-2}),\emptyset\}}.\]
All irreducible constituents of $\chi_\D^{\Theta_0} -\chi_\D^{\Psi_0}$ 
and $\chi_\D^{\Theta_1} -\chi_\D^{\Psi_1}$ must be accounted for in 
$ \sum_{\Theta \in \cM} \chi_\D^\Theta$.
However, it is impossible for these constituents to come from 
model indices of the form \eqref{psi-eq}. 
Indeed, it follows from \eqref{LR-d-eq} that the character of any such $\Phi$ is either 
\be\label{psi1-eq} 
\chi_\D^\Phi =\chi^{[(1),\pm]} \bullet_\D \sum_{p+q=n-2} \chi^{\{(1^p),(1^{q})\}}
=
\sum_{\substack{\{\lambda,\mu\} \vdash n \\ \lambda,\mu \in \cH}} \chi^{\{\lambda,\mu\}}
\ee
where $\cH$ is the set of partitions of the form $(1^{k+1})$ or $(2,1^k)$ for $k \in \NN$,
or
\be\label{psi2-eq}
\chi_\D^\Phi =\chi^{[(1),\pm]} \bullet_\D \sum_{p+q=n-2} \chi^{\{(p),(q)\}}
=
\sum_{\substack{\{\lambda,\mu\} \vdash n \\ \lambda,\mu \in \cH}} \chi^{\{\lambda^\top,\mu^\top\}}. \ee
In the first case $\chi_\D^\Phi$ shares the irreducible constituent $\chi^{\{((2,1^{n-3}),(1)\}}$ with $\chi_\D^{\Theta_1}$ and in the second case
$\chi_\D^\Phi $  shares the irreducible constituent $\chi^{\{(n-1),(1)\}}$ with $\chi_\D^{\Theta_{\lfloor n/2\rfloor}}$.
Thus if any $\Phi $ of the form \eqref{psi-eq} belongs to $ \cM$ then \eqref{psi1-eq} must hold and 
 $\cM^1 = \{\Theta_1'\}$. But then the missing irreducible constituent 
$\chi^{\{ (3), (1^{n-3})\}}$ of  $\chi_\D^{\Theta_1} -\chi_\D^{\Theta_1'}$ does not occur in $\chi_\D^\Phi$.

We conclude that it is necessary to have $\cM^0 = \{ \Theta_0\}$ or $\cM^1 = \{\Theta_1\}$,
so any perfect model $\cP$ for $\WD_n$ when $n\geq 5$ is odd has $\cP \sim \cP^{\D}_n$ 
or $\overline{\cP} \sim \cP^{\D}_n$.

Now we turn to the even case.
One can check directly that $\WD_4$ has no perfect models; we have verified this using 
the computer algebra system \textsf{GAP} \cite{GAP}. Assume that $n\geq 6$ is even. 
The formulas for $\chi_\D^{\Psi_0}$ and $\chi_\D^{\Psi_1}$ given above are still valid,
but now one has $\Theta_k \approx \Theta_k^\diamond$ rather than $\Theta_k\sim \Theta_k^\diamond$. 
By repeating the argument above, however, we can still deduce
that $\cM$ contains a unique element strongly equivalent to 
$\Theta_k $ or $ \Theta_k^\diamond$ for each $2 \leq k \leq \lfloor n/2\rfloor$;
a unique element strongly equivalent to
$\Theta_k $ or $ \Theta_k^\diamond$ or $\Psi_k$ for each $k \in \{0,1\}$;
and all other elements of the form \eqref{psi-eq}.

In view of Proposition~\ref{d-index-prop}, this means that 
the model indices in $\cM$ not of the form \eqref{psi-eq}
contribute to the sum $\sum_{\Theta \in \cM} \chi_\D^\Theta$ at most one element from each pair of degenerate irreducible characters $\chi^{[\nu,\pm]} \in \Irr(\WD_n)$.
But the only degenerate irreducible
characters appearing in $\chi_\D^\Phi$ when
 $\Phi$ has the form \eqref{psi-eq} are $\chi^{[(n/2),\pm]}$ or  $\chi^{[(1^{n/2}),\pm]}$.
 Thus it is impossible to have $\sum_{\Theta \in \cM} \chi_\D^\Theta= \sum_{\chi \in \Irr(\WD_n)}\chi$, and we conclude that no perfect models 
 for $\WD_n$ exist when $n\geq 6$ is even.
\end{proof}

%\begin{example}
%The additional perfect models
%\end{example}

\section{Model classification for exceptional groups}\label{e-sect}

Here we discuss which of the remaining irreducible finite
Coxeter groups have perfect models.

\subsection{Perfect models for dihedral groups}

Fix a positive integer $m$.
The finite dihedral group $\mathsf{I}_2(m)$  is the Coxeter group generated by two elements $s$ and $t$ subject only to the relations $s^2=t^2=(st)^m=1$.
We have already encountered the groups $I_2(1) =\{1\}$, $I_2(2) \cong S_2\times S_2$, $I_2(3) \cong S_3$, and $I_2(4) \cong \W_2$, so assume $m\geq 5$.  
%Let $\ell : \mathsf{I}_2(m) \to \NN$ be the usual Coxeter length function.

The group $\mathsf{I}_2(m)$ has $2m$ elements and a unique nontrivial Coxeter automorphism
interchanging $s\leftrightarrow t$, which we denote by $\ast$. The longest element
is  $w_0 = ststs\cdots = tstst\cdots $ ($m$ factors).
 If $m$ is even then $\ast$ is an outer automorphism, $w_0$ is central, and $w_0^+ : =(w_0,\Ad(w_0)) = w_0 \in \mathsf{I}_2(m)$.
 If $m$ is odd then $\ast = \Ad(w_0)$ and $w_0^+ = (w_0,\ast)\in \mathsf{I}_2(m)^+$.

For either parity of $m$, the only perfect conjugacy classes in $\mathsf{I}_2(m)^+$ are
$\{ 1\}$ and $\{ w_0^+\}$, which are both central. The only model triples $(J,\cK,\sigma)$ for $\mathsf{I}_2(m)$ are therefore
\[ 
\ba 
\TT^{\{s,t\}}_{\sigma} &:= (\{s,t\}, \{1\}, \sigma) \equiv (\{s,t\}, \{w_0^+\}, \sigma) &&\text{for any linear character $\sigma$ of $\mathsf{I}_2(m)$}, \\
\TT^{\{s\}}_{\sigma} &:= (\{s\}, \{1\}, \sigma) \equiv (\{s\}, \{s\}, \sigma)&&\text{for $\sigma \in \{\one,\sgn\}$}, 
\\
\TT^{\{t\}}_{\sigma} &:=(\{t\}, \{1\}, \sigma) \equiv (\{t\}, \{t\}, \sigma)&&\text{for $\sigma \in \{\one,\sgn\}$},
\\
\TT^{\varnothing} &:=(\varnothing, \{1\}, \one).
\ea
\]
%or $\TT^{\varnothing} :=(\varnothing, \{1\}, \one)$, which we can ignore
We can ignore $\TT^{\varnothing}$ 
since its character is not multiplicity-free. % as $\mathsf{I}_2(m)$ is non-abelian.
One has $(\TT^{\{s\}}_{\sigma})^\ast = \TT^{\{t\}}_{\sigma}$
so if $m$ is odd then $\TT^{\{s\}}_{\sigma}$ and $ \TT^{\{t\}}_{\sigma}$ are strongly equivalent.

Let $\zeta := e^{2\pi \sqrt{-1}/m} \in \CC$ %  be a primitive $m$th root of unity,
 and define $\rho_h : \mathsf{I}_2(m) \to \CC$ for $h \in \NN$ to be the map 
 that sends all elements of odd length to zero and has $\rho_h( (st)^k) =\rho_h((ts)^k) =  \zeta^{hk} + \zeta^{-hk}$ for $k \in \NN$.
% \[ \rho_h : w \mapsto  \begin{cases} 0 &\text{if $\ell(w)$ is odd} \\
% \zeta^{hk} + \zeta^{-hk}&\text{if $w = (st)^k$ or $w = (ts)^k$ for $k \in \NN$}.
% \end{cases}
% \]
It is well-known \cite[\S5.3]{JPS} that if $m$ is odd then the distinct irreducible characters of $\mathsf{I}_2(m)$
consist of $\rho_h$ for $h \in [\frac{m-1}{2}]$ 
plus the linear characters $\one :s,t\mapsto 1 $ and $\sgn : s,t\mapsto -1$;
while if $m$ is even then the distinct irreducible characters of $\mathsf{I}_2(m)$
consist of $\rho_h$ for $h \in [\frac{m}{2}-1]$ 
plus the linear characters $\one $, $\sgn$, $\one_{+-}$ and $\one_{-+}$,
where   $\one_{\pm\mp}$ is the class function
sending $s \mapsto \pm 1$ and $t \mapsto \mp 1$. 

Evidently $\chi^{\TT^{\{s,t\}}_{\sigma}}  = \sigma$. The following
identities are straightforward exercises from Frobenius reciprocity.
 If $m$ is odd  then
 $
 \chi^{\TT^{\{s\}}_{\sigma}} = \chi^{\TT^{\{t\}}_{\sigma}} = \sigma + \sum_h \rho_h
$ for $\sigma \in \{\one,\sgn\}$.
If $m$ is even then 
\be\ba
 \chi^{\TT^{\{s\}}_{\one} } &=\Ind_{\langle s\rangle}^{\langle s,t\rangle}(\one)= \one + \one_{+-} + \sum_h \rho_h,
 \quad
  \chi^{\TT^{\{s\}}_{\sgn} } =\Ind_{\langle s\rangle}^{\langle s,t\rangle}(\sgn)= \sgn + \one_{-+} + \sum_h \rho_h,
\\
 \chi^{\TT^{\{t\}}_{\one} }  &=\Ind_{\langle t\rangle}^{\langle s,t\rangle}(\one)= \one + \one_{-+} + \sum_h \rho_h,
 \quad
  \chi^{\TT^{\{t\}}_{\sgn} } =\Ind_{\langle t\rangle}^{\langle s,t\rangle}(\sgn)= \sgn + \one_{+-} + \sum_h \rho_h.
  \ea
  \ee
Recall our notions of model equivalence from Section~\ref{model-equiv-sect}. The following 
now is evident:

\begin{proposition}\label{i2-prop}
Assume $m\geq 5$.
If $m$ is odd then
$
 \left\{ \TT^{\{s,t\}}_{\one}, \TT^{\{s\}}_{\sgn}\right\} 
 %\sim \left\{ \TT^{\{s,t\}}_{\one}, \TT^{\{t\}}_{\sgn}\right\}
\approx \left\{ \TT^{\{s,t\}}_{\sgn}, \TT^{\{s\}}_{\one}\right\} 
%\sim \left\{ \TT^{\{s,t\}}_{\sgn}, \TT^{\{t\}}_{\one}\right\} 
 $
are perfect models for $\mathsf{I}_2(m)$ and every perfect model is strongly equivalent to one of these.
If $m$ is even then $
 \left\{ \TT^{\{s,t\}}_{\one}, \TT^{\{s,t\}}_{\one_{+-}}, \TT^{\{s\}}_{\sgn}\right\} \approx
  \left\{ \TT^{\{s,t\}}_{\one}, \TT^{\{s,t\}}_{\one_{-+}}, \TT^{\{t\}}_{\sgn}\right\}
\approx \left\{ \TT^{\{s,t\}}_{\sgn}, \TT^{\{s,t\}}_{\one_{-+}}, \TT^{\{s\}}_{\one}\right\} \approx \left\{ \TT^{\{s,t\}}_{\sgn}, \TT^{\{s,t\}}_{\one_{+-}},\TT^{\{t\}}_{\one}\right\} $
are perfect models for $\mathsf{I}_2(m)$ and every perfect model is strongly equivalent to one of these.
\end{proposition}

\subsection{Perfect models for exceptional groups}

The finite Coxeter system $(W,S) = (\WH_3,\{h_1,h_2,h_3\})$ of type $\mathsf{H}_3$ has Coxeter diagram
\begin{center}
    \begin{tikzpicture}[xscale=1.2, yscale=1,>=latex,baseline=(z.base)]
    \node at (0,0) (z) {};
      \node at (-1,0) (T0) {$h_1$};
      \node at (0,0) (T1) {$h_2$};
      \node at (1,0) (T2) {$h_3$};
      \draw[-]  (T0) -- (T1) node[midway,above,scale=0.75] {$5$}; 
      \draw[-] (T1) -- (T2) node[midway,above,scale=0.75] {$3$}; 
     \end{tikzpicture}
     \end{center}
     %The group $\WH_3$ has order 120 and 
and may be embedded in $\WD_6$ by setting $h_1 := s_1s_3$, $h_2 := s_2s_4$, and $h_3 := s_{-1}s_5$.
There are no nontrivial Coxeter automorphisms of $\WH_3$, the only linear characters are $\one$ and $\sgn$,
and the only perfect involutions are $1$ and the longest element $w_0$, which is central.

For each subset $J \subseteq \{ h_1,h_2,h_3\}$ and linear character $\sigma : \langle J\rangle \to \QQ$
let $\TT^J_\sigma$ denote the model triple $(J,\{1\},\sigma)$.
Let $\one_{+-}$ and $\one_{-+}$ be the two linear characters of $\langle h_1,h_3\rangle \cong S_2\times S_2$
not given by $\one$ or $\sgn$.
We have checked the following propositions using \textsf{GAP} \cite{GAP}:
%All model triples for $\WH_3$ have the form
%\[ 
%\ba 
%\TT^{\{h_1,h_2,h_3\}}_{\sigma} &:= (\{h_1,h_2,h_3\}, \{1\}, \sigma) \equiv (\{h_1,h_2,h_3\}, \{w_0\}, \sigma) &&\text{for $\sigma \in \{\one,\sgn\}$}, \\
%\TT^{\{h_1,h_2\}}_{\sigma} &:= (\{h_1,h_2\}, \{1\}, \sigma) \equiv (\{h_1,h_2\}, \{w_0^+\}, \sigma)&&\text{for $\sigma \in \{\one,\sgn\}$}, 
%\\
%\TT^{\{h_2,h_3\}}_{\sigma} &:=(\{h_2,h_3\}, \{1\}, \sigma) \equiv (\{h_2,h_3\}, \{w_0^+\}, \sigma)&&\text{for $\sigma \in \{\one,\sgn\}$},
%\\
%\TT^{\{h_1,h_3\}}&:= (\{h_1,h_3\}, \{1\}, \sigma)\equiv (\{h_1,h_3\}, \{z\}, \sigma) 
%%\equiv (\{h_1,h_3\}, \{h_3\}, \sigma) \equiv (\{h_1,h_3\}, \{h_1h_3\}, \sigma)
%&&\text{for any linear character $\sigma$ of $\langle h_1,h_3\rangle$}, \\
%\ea
%\]

\begin{proposition}\label{h3-prop}
The sets 
$
\left\{ \TT^{\{h_1,h_2,h_3\}}_{\one},\TT^{\{h_1,h_2,h_3\}}_{\sgn}, \TT_{\one_{+-}}^{\{h_1,h_3\}}\right\}
\approx
\left\{ \TT^{\{h_1,h_2,h_3\}}_{\one},\TT^{\{h_1,h_2,h_3\}}_{\sgn}, \TT_{\one_{-+}}^{\{h_1,h_3\}}\right\}
$
and
$
 \left\{ \TT^{\{h_1,h_2\}}_{\one}, \TT^{\{h_2,h_3\}}_{\sgn}\right\} \approx
  \left\{ \TT^{\{h_1,h_2\}}_{\sgn}, \TT^{\{h_2,h_3\}}_{\one}\right\} 
$
are perfect models for $\WH_3$, and every perfect model is strongly equivalent to one of these four.
\end{proposition}

\begin{proposition}
The Coxeter groups of types 
 $\mathsf{E}_6$,
  $\mathsf{E}_7$,
   $\mathsf{E}_8$,
    $\mathsf{F}_4$,
    and  $\mathsf{H}_4$ have no perfect models.
    \end{proposition}

\appendix
\section{Proofs of Theorems~\ref{tech1-thm} and \ref{tech2-thm}}\label{app-sect}

This final section contains the proofs of two technical results from Sections~\ref{intro-sect} and \ref{prelim-sect}.
If $\lambda=(\lambda_1,\dots,\lambda_j)$ and $\mu=(\mu_1,\dots,\mu_k)$ are partitions with $j\leq k$,
then we define $\lambda+\mu=\mu+\lambda:=(\lambda_1+\mu_1,\dots,\lambda_j+\mu_j,\mu_{j+1},\dots,\mu_k)$ and $\lambda\cup\mu$ to be the partition sorting $(\lambda_1,\dots,\lambda_j,\mu_1,\dots,\mu_k)$.
The arguments below frequently makes use of the following identity \cite[Lem. 3.2]{Stembridge}:
\be\label{32-eq}
c^{\nu+(1^r)}_{\lambda(\mu+(1^r))} \geq c_{\lambda\mu}^\nu\quand c^{\nu \cup (r)}_{\lambda(\mu \cup (r))} \geq c_{\lambda\mu}^\nu\ee
for all partitions $\lambda$, $\mu$, $\nu$ and integers $r \in \NN$.

\begin{lemma}\label{lrc-lem}
If $\lambda$ and $\mu$ are any partitions then $c_{\lambda\mu}^{\lambda+\mu} \geq 1$ and $c_{\lambda\mu}^{\lambda\cup \mu} \geq 1$.
\end{lemma}

\begin{proof}
Write  $\ell(\mu)$ for the number of nonzero parts of $\mu$.
If $\ell(\mu) = 0$ then $\mu =\emptyset$ and the result holds as $c^{\lambda}_{\lambda\emptyset}=1$.
If $\ell(\mu) =r >0$ then  \eqref{32-eq} implies that
$c_{\lambda\mu}^{\lambda+\mu} \geq c_{\lambda\tilde \mu}^{\lambda+\tilde \mu} $ where 
$\tilde \mu := (\mu_1-1,\mu_2-1,\dots,\mu_r-1)$, and in this case we may assume
by induction on $|\lambda|+|\mu|$ that $ c_{\lambda\tilde \mu}^{\lambda+\tilde \mu} \geq 1$.
The other identity holds since $
c_{\lambda\mu}^{\lambda\cup\mu}=c_{\lambda^\top\mu^\top}^{(\lambda\cup\mu)^\top}=c_{\lambda^\top\mu^\top}^{\lambda^\top+\mu^\top}\ge 1.
$
\end{proof}

Let $(W,S)$ be an irreducible finite Coxeter system and let $J\subseteq S$ be a subset with $|S\setminus J| \geq 2$.
Theorem~\ref{tech1-thm} asserts that $\Ind_{W_J}^W(\chi)$ is not multiplicity-free for all $\chi \in \Irr(W)$. 

\begin{proof}[Proof of Theorem~\ref{tech1-thm}]
By the transitivity of induction we may assume that $|S\setminus J| = 2$. We have checked the desired property 
using \textsf{GAP} \cite{GAP} for each finite exceptional Coxeter group.
It remains to prove the result when $W \in \{S_n,\W_n,\WD_n\}$ for all $n\geq 3$.
%We consider these cases below.

First suppose $W=S_n$ so that $S=\{s_1,s_2,\dots,s_{n-1}\}$ and  $S\setminus J=\{s_i,s_j\}$ for some $1\leq i<j<n$.
Then we have $W_J=S_{i}\times S_{j-i}\times S_{n-j}$
and it suffices to show that $\chi^\lambda \bullet_\A \chi^\mu \bullet_A \chi^\nu$ is never multiplicity-free if $\lambda$, $\mu$, and $\nu$ are nonempty partitions. 
This is easy to derive from the Pieri rules when two of these partitions are equal to $(1)$.

A partition is a \defn{rectangle} if it has the form $(a^i) = (a,a,\dots,a)$.
Stembridge \cite[Thm. 3.1]{Stembridge} gives 
necessary and sufficient criteria for $\chi^\lambda \bullet_\A \chi^\mu $ to be multiplicity-free.
This result implies that  $\chi^\mu \bullet_A \chi^\nu$
can only be multiplicity-free if at least one of $\mu$ or $\nu$ is a rectangle, so
$\chi^\lambda \bullet_\A \chi^\mu \bullet_A \chi^\nu$ can only be multiplicity-free if 
at least two of $\lambda,\mu,\nu$ are rectangles, say $\lambda=(a^i)$ and $\mu=(b^j)$. Suppose $\nu$ is not a rectangle and without loss of generality assume  $a\ge b$.

We claim that there is a non-rectangle $\rho$ such that $c_{(a^i)(b^j)}^{\rho}\ge1$. If $a>b$ and $i\neq j$, then one can take $\rho=(a^i,b^j)$ 
since Lemma~\ref{lrc-lem} implies that
$c_{(a^i)(b^j)}^{(a^ib^j)}\ge 1$.
 If $a=b$ and $\max\{i,j\}>1$, then
one can take $\rho=(2a,a^{i+j-2})$
since using the second identity in \eqref{32-eq} one can check that
 $c_{(a^i)(a^j)}^{(2a,a^{i+j-2})}\ge c_{(a)(a)}^{(2a)}=1$. 
% \yifeng{Since $c^{\nu \cup (r)}_{\lambda(\mu \cup (r))} \geq c_{\lambda\mu}^\nu$, then we have $c_{(a^i)(a^j)}^{(2a,a^{i+j-2})}\ge c_{(a^i)(a^{j-1})}^{(2a,a^{i+j-3})}$, and then by induction on $i$ and $j$, we have $c_{(a^i)(a^j)}^{(2a,a^{i+j-2})}\ge c_{(a^i)(a)}^{(2a,a^{i-1})}\ge c_{(a)(a)}^{(2a)}=1$.}
 If $a=b>1$ and $i=j=1$,
 then one can set $\rho = (a+1,a-1)$ since the Pieri rules give $c_{(a)(a)}^{(a+1,a-1)}\ge c_{(a)(1)}^{(a+1)}=1$. 
 The only remaining case is when $a=b=1$ and $i=j=1$ so that $\lambda=\mu=(1)$,
 which we already considered.
 We conclude that if $\lambda$, $\mu$, and $\nu$ are not all rectangles then 
 $\chi^\lambda \bullet_\A \chi^\mu \bullet_A \chi^\nu$ is not multiplicity-free.

Now assume further that $\nu=(c^k)$ and $a\ge b\ge c$. A \defn{$k$-line rectangle} is a rectangular partition with either $k$ rows or $k$ columns.
Assume at least two of $\lambda$, $\mu$, $\nu$ are not 1-line rectangles, say $\mu$ and $\nu$. If we define
$\rho =(a+1,a^{i-1},b^{j-1},b-1)$, then one can check using \eqref{32-eq} that 
\[
c_{(a^i)(b^j)}^{\rho} = c_{(a^i)(b^j)}^{(a+1,a^{i-1},b^{j-1},b-1)}\ge c_{(a^i)(b)}^{(a+1,a^{i-1},b-1)}\ge c_{(a)(b)}^{(a+1,b-1)}=1
\]
%$c_{(a^i)(b^j)}^{\rho}\ge c_{(a)(b)}^{(a+b)}=1$,
but then \cite[Thm. 3.1]{Stembridge} implies that $\chi^\rho \bullet_A \chi^\nu$ is not multiplicity-free since $\rho$ is not a ``fat hook,'' so $\chi^\lambda \bullet_\A \chi^\mu \bullet_A \chi^\nu$ is also not multiplicity-free.

Thus we reduce to the case when at least two of $\lambda$, $\mu$, and $\nu$ are 1-line rectangles, say $\lambda$ and $\mu$. 
If $\lambda=(a)$, $\mu=(b)$ and $\nu=(c^k)$ with $a\ge b$, then 
the Pieri rules give $c_{(a)(b)}^{(a+b-1,1)}=c_{(a)(b)}^{(a+b)}=1$ while
\eqref{32-eq} implies that 
\[
c_{(a+b-1,1)(c^k)}^{(a+b+c-1,c^{k-1},1)}\ge c_{(a+b-1,1)(c)}^{(a+b+c-1,1)}\ge c_{(a+b-1)(c)}^{(a+b+c-1)}=1
\]
and
\[ c_{(a+b)(c^k)}^{(a+b+c-1,c^{k-1},1)}\ge c_{(a+b)(c)}^{(a+b+c-1,1)}=1,
\]
%c_{(a)(b)}^{(a+b-1,1)}=c_{(a)(b)}^{(a+b)}=1\quand 
%c_{(a+b-1,1)(c^k)}^{(a+b+c-1,c^{k-1},1)}\ge1\quand
%c_{(a+b)(c^k)}^{(a+b+c-1,c^{k-1},1)}\ge1,\] 
so 
$\chi^\lambda \bullet_\A \chi^\mu \bullet_A \chi^\nu$
is not multiplicity-free. Next, 
if $\lambda=(a)$, $\mu=(1^b)$ and $\nu=(c^k)$ with $a>b>1,c>1$, then
the Pieri rules give $c_{(a)(1^b)}^{(a+1,1^{b-1})}=c_{(a)(1^b)}^{(a,1^b)}=1$
while  \eqref{32-eq} implies that 
\[
c_{(a+1,1^{b-1})(c^k)}^{(a+c,c^{k-1},1^b)}\ge c_{(a+1,1^{b-1})(c)}^{(a+c,1^b)}\ge c_{(a+1)(c)}^{(a+c,1)}\ge c_{(1)(c)}^{(c,1)}=1
\]
and
\[
c_{(a,1^b)(c^k)}^{(a+c,c^{k-1},1^b)}\ge c_{(a,1^b)(c)}^{(a+c,1^b)}\ge c_{(a)(c)}^{(a+c)}=1
\]
%\[
%c_{(a)(1^b)}^{(a+1,1^{b-1})}=c_{(a)(1^b)}^{(a,1^b)}=1
%\quand c_{(a+1,1^{b-1})(c^k)}^{(a+c,c^{k-1},1^b)}\ge1\quand
%c_{(a,1^b)(c^k)}^{(a+c,c^{k-1},1^b)}\ge1,
%\]
 so 
$\chi^\lambda \bullet_\A \chi^\mu \bullet_A \chi^\nu$
is again not multiplicity-free.
The same result follows in the remaining cases 
since $\chi^\lambda \bullet_\A \chi^\mu \bullet_A \chi^\nu$ is multiplicity-free if and only if
$\chi^{\lambda^\top} \bullet_\A \chi^{\mu^\top} \bullet_A \chi^{\nu^\top}$ is multiplicity-free.

Next let $W = \W_n$   so that $S=\{s_0,s_1,s_2,\dots,s_{n-1}\}$.
There are two cases for $J$, depending on whether $s_0\in S\setminus J$.
First assume $S\setminus J=\{s_0,s_i\}$ for some $i \in [n-1]$. Then $W_J = S_i\times S_{n-i}$ and it suffices to show that $\Ind_{S_{i}\times S_{n-i}}^{\W_n}(\chi^{\lambda}\boxtimes\chi^{\mu})$ is not multiplicity-free for all $\lambda\vdash i$ and $\mu\vdash n-i$.
Since 
$
\Ind_{S_i\times S_{n-i}}^{\W_n}\(\chi^{\lambda}\boxtimes\chi^{\mu}\)=\Ind_{S_n}^{\W_n}\Ind_{S_{i}\times S_{n-i}}^{S_n}\(\chi^{\lambda}\boxtimes\chi^{\mu}\)
=\Ind_{S_n}^{\W_n}\Bigl(\sum_{\nu\vdash n}c_{\lambda\mu}^\nu \chi^\nu\Bigr)
$,
 Lemma~\ref{lrc-lem} tells us that our induced character has $\Ind_{S_n}^{\W_n}(\chi^{\lambda+\mu}+\chi^{\lambda\cup\mu})$
 as a constituent.
But this is not multiplicity-free by \eqref{LR-bc-eq} since $c_{\lambda\mu}^{\lambda+\mu}\geq 1$ and $c_{\lambda\mu}^{\lambda\cup\mu}\geq 1$.

Alternatively suppose $S\setminus J=\{s_i,s_{i+j}\}$ where $0<i<i+j<n$.
Then $W_J = \W_i\times S_{j} \times S_{n-i-j}$ and it suffices to show that 
\[ \Ind_{\W_{i}\times S_{j}\times S_{n-i-j}}^{\W_n}\(\chi^{(\lambda_1,\lambda_2)}\boxtimes\chi^{\mu}\boxtimes\chi^{\nu}\)\\
=\Ind_{\W_{i}\times \W_{n-i}}^{\W_n}\(\chi^{(\lambda_1,\lambda_2)}\boxtimes\Ind_{S_{j}\times S_{n-i-j}}^{\W_{n-i}}\(\chi^{\mu}\boxtimes\chi^{\nu}\)\)
\]
is not multiplicity-free for all
$(\lambda_1,\lambda_2)\vdash i$, $\mu\vdash j$ and $\nu\vdash n-i-j$.
But this is immediate since we have already seen that $\Ind_{S_{j}\times S_{n-i-j}}^{\W_{n-i}}\(\chi^{\mu}\boxtimes\chi^{\nu}\)$ is not multiplicity-free.  

Finally let $W=\WD_n$ for $n\geq 4$ so that $S=\{s_{-1},s_1,s_2,\dots,s_{n-1}\}$.
There are again two cases for $J$, depending on whether $s_{-1}\in S\setminus J$.
First suppose $S\setminus J=\{s_{-1},s_i\}$ for some $i \in [n-1]$.
Then $W_J = S_i \times S_{n-i}$ and it suffices to show that 
$\Ind_{S_{i}\times S_{n-i}}^{\WD_n}\(\chi^{\lambda}\boxtimes\chi^{\mu}\)
$ is not multiplicity-free for all
 $\lambda\vdash i$ and $\mu\vdash n-i$. If $\lambda \neq \mu$ then this holds by 
\eqref{bcd-reciprocity} since we have already seen that $\Ind_{S_{i}\times S_{n-i}}^{\W_n}\(\chi^{\lambda}\boxtimes\chi^{\mu}\)$ 
is not multiplicity-free.

It remains to show that $\Ind_{S_{k}\times S_{k}}^{\WD_n}\(\chi^{\lambda}\boxtimes\chi^{\lambda}\)
$ is not multiplicity-free when $n=2k$ is even and $\lambda\vdash k$.
Again using \eqref{bcd-reciprocity}, it suffices to check
that $\chi^{(\lambda,\lambda)}$ appears in $\Ind_{S_{k}\times S_{k}}^{\W_n}\(\chi^{\lambda}\boxtimes\chi^{\lambda}\)$
with multiplicity at least $3$.
As we already know that $c_{\lambda\lambda}^{\lambda+\lambda}\geq 1$ and $c_{\lambda\lambda}^{\lambda\cup\lambda}\geq 1$, the desired property follows via \eqref{LR-bc-eq} 
once we use \eqref{32-eq} to check that $c_{\lambda\lambda}^\nu\geq 1$ for 
\[\nu := 
\begin{cases}(2\lambda_1,\lambda_2,\lambda_2,\cdots,\lambda_r,\lambda_r)&\text{if }r=\ell(\lambda)>1\\(\lambda_1+1,\lambda_1-1)&\text{if }\ell(\lambda)=1.\end{cases}
\]
%This is straightforward from  \cite[Lem. 3.2]{Stembridge}.
%\yifeng{$c_{(\lambda_1,\lambda_2,\cdots,\lambda_r)(\lambda_1,\lambda_2,\cdots,\lambda_r)}^{2\lambda_1,\lambda_2,\lambda_2,\cdots,\lambda_r,\lambda_r}\ge c_{(\lambda_1,\lambda_2,\cdots,\lambda_{r-1})(\lambda_1,\lambda_2,\cdots,\lambda_{r-1})}^{2\lambda_1,\lambda_2,\lambda_2,\cdots,\lambda_{r-1},\lambda_{r-1}}\ge \cdots c_{(\lambda_1)(\lambda_1)}^{2\lambda_1}=1$ and $c_{(\lambda_1)(\lambda_1)}^{(\lambda_1+1,\lambda_1-1)}\ge c_{(\lambda_1)(1)}^{(\lambda_1+1)}=1$.}

In the second case when
  $S\setminus J=\{s_{1},s_i\}$ for some $i \in \{-1\}\sqcup\{2,3,\dots,n-1\}$,
  the desired property follows by a symmetric argument,
  since this case differs from the one just considered by applying the automorphism $\diamond$.
  If $S\setminus J = \{i,j\}$ for some $1\leq i <i+j  <n$ then $W_J = \WD_i \times S_j \times S_{n-i-j}$
and the fact that $\Ind_{\WD_i \times S_j \times S_{n-i-j}}^{\WD_n}(\chi)$ is never multiplicity-free
is immediate from the cases already examined, as in the argument for type $\BC$.
\end{proof}

Suppose $(W,S)$ is an irreducible finite Coxeter system
and $\TT = (J,\cK,\sigma)$ is a model triple for $W$ that is not factorizable. 
Theorem~\ref{tech2-thm} is equivalent to the claim that the character $\chi^\TT$
given by \eqref{chTT-eq}
 is not multiplicity-free. We prove this below.

\begin{proof}[Proof of Theorem~\ref{tech2-thm}]
Since $W$ is irreducible we must have $J \neq S$. 
%We will show that $\chi^\TT$ is not multiplicity-free.
Let $z$ be the unique minimal-length element in $\cK$ and  let $\theta := \aut(z)$.
If $|S-J| \geq 2$ then $\chi^\TT$ is not multiplicity-free
by Theorem~\ref{tech1-thm}. 

Assume $|S-J| = 1$. 
Since $\theta$ interchanges two irreducible components of $(W_J,J)$,
the irreducible Coxeter system $(W,S)$ must be of type 
$\A_{2n-1}$ for $n\geq 2$ (with $S \setminus J = \{s_n\}$),
$\BC_3$ (with $S\setminus J = \{s_1\}$), $\D_4$ (with $S\setminus J = \{s_2\}$, $\D_7$ (with $S\setminus J = \{s_3\}$), $\mathsf{E}_6$, or $\mathsf{H}_3$.
In all but the first case, one can verify that $\chi^\TT$ is never multiplicity-free
by a finite calculation; we have done this using the computer algebra system \textsf{GAP} \cite{GAP}. 

For the remaining case, suppose $W= S_{2n}$ and $J = \{ s_i : n\neq i \in [2n-1]\}$ so that $W_J = S_n\times S_n$.
Then the automorphism $\theta$ must either be the map 
with $s_i \leftrightarrow s_{n+i}$ for all $i \in [n-1]$
or the map with $s_i \leftrightarrow s_{2n-i}$ for all $i \in [n-1]$.
In both cases the only possibility for $\cK$ is the set $\{ (w\cdot \theta(w)^{-1}, \theta) : w \in S_n\}$,
whose unique minimal-length element  is $z =(1,\theta)$.

The $S_n\times S_n$-centralizer of this element is isomorphic to the diagonal subgroup $\Delta(S_n) = \{ (x,y) \in S_n \times S_n : x=y\} \cong S_n$, on which the linear characters of $S_n\times S_n$ each restrict to $\one$ or $\sgn$.
To show that $\chi^\TT$ is not multiplicity-free it suffices to check that 
$\Ind_{\Delta(S_n)}^{S_{2n}}(\sigma) =\Ind_{S_n\times S_n}^{S_{2n}} \Ind_{\Delta(S_n)}^{S_n\times S_n}(\sigma)$
is not multiplicity-free for each linear character $\sigma$ of $\Delta(S_n)$.

It is a standard exercise using Frobenius reciprocity and the orthogonality relations for irreducible characters that 
$\Ind_{\Delta(S_n)}^{S_n\times S_n}(\sigma)$ is either $\sum_{\lambda \vdash n} \chi^\lambda \boxtimes \chi^\lambda$
or
$\sum_{\lambda \vdash n} \chi^\lambda \boxtimes \chi^{\lambda^\top}$,
so 
$\Ind_{\Delta(S_n)}^{S_{2n}}(\sigma)$ is either $\sum_{\nu\vdash 2n} \Bigl(\sum_{\lambda\vdash n} c_{\lambda\lambda}^\nu \Bigr) \chi^\nu$  or $\sum_{\nu\vdash 2n} \Bigl(\sum_{\lambda\vdash n} c_{\lambda\lambda^\top}^\nu \Bigr) \chi^\nu$.
The second character is not multiplicity-free since $n\geq 2$ and $c_{\lambda\lambda^\top}^\nu = c_{\lambda^\top\lambda}^\nu$. The first character is also not multiplicity-free: when $n=2$ one has $c^{(2,2)}_{(2)(2)} =c^{(2,2)}_{(1,1)(1,1)} =1$, and for $n>2$ one can check that $c^{(n,n-1,1)}_{(n-1,1)(n-1,1)} = 2$
using the Pieri rules and the identity
$\chi^{(n-1,1)} = \chi^{(n-1)}\bullet_\A \chi^{(1)}  - \chi^{(n)}$.
\end{proof}

\end{document}